%% file: main.tex
\documentclass[10.5pt,a4paper]{article}
\usepackage{graphicx} % Required for inserting images
\usepackage{amsmath,amssymb,mathrsfs,amsthm,enumerate}
\usepackage{mathtools}
\usepackage[dvipsnames]{xcolor}
\usepackage{tikz,multirow}
\usepackage{xurl}
\usepackage{algorithm,algpseudocode}
\usepackage{comment}

\usepackage[pdftex, colorlinks, linkcolor=blue, citecolor=blue, urlcolor=blue]{hyperref}
\hypersetup{linkcolor=MidnightBlue, citecolor=magenta, urlcolor=MidnightBlue}
\usepackage[a4paper, 
		top=3.0cm,
		bottom=3.0cm,
		left=2.5cm, 
		right=2.5cm,
		marginparwidth=1.4cm]{geometry}
\usepackage{cleveref}
\usepackage[margin=1cm]{caption}
\usepackage[bibencoding=utf8,style=alphabetic]{biblatex}
\AtBeginBibliography{\small}
\AtEveryBibitem{%
\ifentrytype{article}{
    \clearfield{issn}%
    \clearfield{isbn}%
}{}
\ifentrytype{inproceedings}{
    \clearfield{issn}%
    \clearfield{isbn}%
}{}
\ifentrytype{inbook}{
    \clearfield{issn}%
    \clearfield{isbn}%
}{}
}

\usepackage{etoc}

\newtheorem{definition}{Definition}[section]
\newtheorem{theorem}{Theorem}
\newtheorem{lemma}[theorem]{Lemma}

\newtheorem{proposition}[theorem]{Proposition}
\newtheorem{conjecture}[theorem]{Conjecture}
\theoremstyle{definition}
\theoremstyle{remark}
\newtheorem{remark}[theorem]{Remark}
\newtheorem{example}{Example}

\crefname{conjecture}{conjecture}{conjectures}
\Crefname{conjecture}{Conjecture}{Conjectures}

\newenvironment{appendixformulas}
  {\begingroup\small\allowdisplaybreaks}
  {\endgroup}

\DeclareMathOperator{\vol}{vol}
\newcommand{\Q}{\mathbb{Q}}
\newcommand{\R}{\mathbb{R}}

\newcommand{\N}{\mathbb{N}}

\DeclareMathOperator{\conv}{conv}
\DeclareMathOperator{\slicetxt}{slice}
\DeclareMathOperator{\slabtxt}{slab}
\newcommand{\slice}{\slicetxt(a,t,\|\cdot\|)}
\newcommand{\slab}{\slabtxt(a,t,\|\cdot\|)}

\newcommand{\sliceinf}{\slicetxt(a,t,\|\cdot\|_{\infty})}
\newcommand{\slabinf}{\slabtxt(a,t,\|\cdot\|_{\infty})}

\newcommand{\mslice}{\int_{\slice}\ \sum_{i=1}^d x_i^M}
\newcommand{\mslab}{\int_{\slab}\ \sum_{i=1}^d x_i^M}

\newcommand{\normball}{B_{\|\cdot\|}}

\newcommand{\hypercube}{B_{\infty}}

\usepackage{todonotes}

\allowdisplaybreaks

\makeatletter
\newcommand{\EnableMathBreaks}{%
  \relpenalty=100 \binoppenalty=100
  \mathcode`*="8000
  \begingroup\lccode`~=`*\lowercase{\endgroup
    \def~}{\allowbreak\mkern2mu*\mkern2mu\allowbreak}%
  \mathcode`+="8000
  \begingroup\lccode`~=`+\lowercase{\endgroup
    \def~}{\allowbreak\mkern2mu+\mkern2mu\allowbreak}%
}
\newcommand{\DisableMathBreaks}{%
  \mathcode`*="2203
  \mathcode`+="202B
}
\makeatother

\newcommand\blfootnote[1]{%
  \begingroup
  \renewcommand\thefootnote{}\footnote{#1}%
  \addtocounter{footnote}{-1}%
  \endgroup
}

\title{\sc{Critical moments of slices and slabs of the cube \\ (and other polyhedral norms)}}
\author{Marie-Charlotte Brandenburg \and Jes\'us A. De Loera \and Yu Luo \and Chiara Meroni}
\date{}

\begin{document}
\maketitle

\noindent
\blfootnote{
\textbf{MSC classes:} 
52A38, %Length, area, volume and convex sets (aspects of convex geometry)
52B55, %Computational aspects related to convexity
52-08 %Computational methods for problems pertaining to convex and discrete geometry
(primary);
68W30, %Symbolic computation and algebraic computation
90C57 (secondary). %Polyhedral combinatorics, branch-and-bound, branch-and-cut
}

\begin{abstract}
In this article, we present a unified algebraic–combinatorial framework for computing explicit, piecewise rational, and combinatorially indexed parametric formulas for volumes and higher moments of slices and slabs of polyhedral norm balls. Our main method builds on prior work concerning a combinatorial decomposition of the parameter space of all slices of a polytope. We extend this framework to slabs, and find a polynomial-time algorithm in fixed dimension. We also exhibit computational methods to obtain moments of arbitrary order for all slices or slabs of any polyhedral norm ball, and an algebraic framework for analyzing their critical points. 
In addition, we present an experimental study of the $d$-dimensional unit cube. Our analysis recovers and reinterprets the known volume formulas for slabs and slices of the two- and three-dimensional cubes, first obtained by König and Koldobsky. Moreover, our method identifies a new complete family of fourteen rational functions giving the volumes of slices and slabs of the four-dimensional cube. We further compute explicit higher moments of slices and slabs in dimensions two and three,  
and derive explicit formulas for moments of arbitrary order for slices of the two-dimensional cube, describing their critical points.
For the four-dimensional cube, we further use these formulas to identify candidate global maxima and minima for slice and slab volumes; the candidates are verified to be critical points by exact symbolic computation and are checked numerically to give the global extrema.
\end{abstract}

\section{Introduction}
\label{sec:intro}
Understanding the moments of convex bodies, in particular high-dimensional polyhedral norm balls, is a key tool in convex geometry and analysis. In this paper, moments are defined as integrals of powers of coordinates or norms (e.g., expectations like $E[|x_i|^k]$ for a random point $x$ in the norm ball). If $K\subset\R^d$ is a centrally symmetric convex body and $X$ is a measure uniformly distributed on $K$, then for every direction $a$ in the unit sphere the one-dimensional marginal $\langle a,X\rangle$ has a density proportional to the volume of $K$ intersected with a hyperplane orthogonal to $a$, which varies with the distance of the hyperplane from the origin. More generally, the \emph{moments of slices} are exactly the
moment data of one-dimensional marginals. This point of view underlies much of asymptotic geometric analysis:
information about the distribution of $\langle a,X\rangle$ as $a$ varies controls the concentration
phenomena, the deviation inequalities, and the geometry of random sections and projections,
themes that are at the heart of high-dimensional convexity, 
and its functional-analytic consequences \cite{MilmanSchechtman1986,Milman2006,Schneider14}. In short, the moments of convex bodies are basic quantities that connect convex geometry, probability, and the theory of Banach spaces.

This paper discusses how to efficiently compute the volumes and higher-order moments of \emph{slices} and \emph{slabs} of \emph{polyhedral norm balls}, and shows how to recover explicit piecewise rational functions for them in terms of the slicing parameters, $(a,t)$, $a \in \R^d$, where $d$ is the ambient dimension of the slab, and $t \in \R$, $t\geq 0$. 
Our methods are general and apply to any scaled unit ball
\[
B^d_{\|\cdot\|} \ :=\ \{x\in\R^d:\|x\|\le\tfrac12\},
\]
of an arbitrary polyhedral norm, such as the $\ell_\infty$-ball that is the centered hypercube
\[
B_\infty^d \ =\ \Bigl[-\tfrac12,\tfrac12\Bigr]^d.
\]
In what follows, let $S^{d-1}$ denote the unit sphere in $\R^d$, fix $a=(a_1,\dots,a_d)\in S^{d-1}$ and $t\ge 0$, and write
\[
H(a,t)\ :=\ \{x\in\R^d:\langle a,x\rangle=\tfrac{t}{2}\},
\]
where $\langle \cdot , \cdot \rangle$ is the standard scalar product on $\R^d$. In other words, $H(a,t)$ is a hyperplane with unit normal vector $a$ and Euclidean distance $t/2$ from the origin. This is the scaling of the slicing parameter that we use throughout the paper. We consider the \emph{slice}
\[
\slice :=\ B^d_{\|\cdot\|}\cap H(a,t),
\]
and the \emph{slab}
\[
\slab :=\ \{x\in B^d_{\|\cdot\|}:\ |\langle a,x\rangle|\le \tfrac{t}{2}\}.
\]
For $M\in\mathbb Z_{\geq 0}$ we focus on the $M$-th moment
\[
\mslice
\quad \text{and} \quad
\mslab
\]
although, more generally, our method applies to all polynomial integrands as given in \cite{intsimplex}.

\medskip
While we explain our methods for arbitrary polyhedral norm balls, experimentally we focus on the cube $B_\infty^d$. The geometry of hyperplane sections and slabs of the cube has become a central topic in convex geometry, combinatorics, and geometric functional analysis. A classical question is how large or small the $(d-1)$-dimensional volume of a hyperplane section can be, and which choices of $a$ and $t$ make the slicing hyperplane extremal.
For \emph{minimal} sections, Vaaler proved that every $k$-dimensional central section of the cube
$[-\tfrac12,\tfrac12]^d$ has volume at least $1$, with equality precisely for coordinate-parallel sections
\cite{Vaa79}; this completed earlier work of Hadwiger \cite{Hadwiger72} and Hensley \cite{Hensley79}. This also has extensions
to $\ell_p$-balls via work of Meyer and Pajor \cite{MP88}.
For \emph{maximal} central hyperplane sections, Ball showed that for the cube
$[-\tfrac12,\tfrac12]^d$ every central hyperplane section has $(d-1)$-dimensional volume at most $\sqrt{2}$,
with equality for normals proportional to $e_i\pm e_j$ \cite{Ball-1986}.
Koldobsky’s Fourier-analytic framework gives general formulas for section volumes and connects
these extremal questions to harmonic analysis \cite{Koldobsky-2005}. In the setting of noncentral slicing, Barthe and Koldobsky related extremal slabs to Laplace-transform techniques \cite{BartheKoldobsky-2003}; K\"onig and Koldobsky studied how the volume of a slice or slab depends on the direction $a$ and the
offset parameter $t$, and obtained explicit formulas in dimensions $d=2,3$ \cite{KK11}.  Pournin and collaborators investigated shallow sections at a prescribed distance from the center. They proved local optimality results for diagonal directions in high dimension, for a range of offsets \cite{Pournin-2023}, see also \cite{DezaHiriartUrrutyPournin2021,Schneider14}. Ambrus and collaborators \cite{ambrus2022critical,ambrus2024nondiagonal,ambrus2025estimates} studied critical points of the volume of central hyperplane sections of the $d$-dimensional cube, showing that non-diagonal critical sections exist in all dimensions $d \geq 4$, yet are saddle points — supporting the conjecture that all locally extremal sections are diagonal.

Moments of slices can be viewed as integrals of monomials over $(d-1)$-dimensional sections of a convex body, and thus capture how mass, curvature, and directional structure are distributed along hyperplane cuts. Such moments naturally arise in geometric tomography \cite{Gardner2006} and Radon-type reconstruction problems \cite{Natterer2001}, where line or hyperplane integrals encode information used in medical imaging, inverse problems, computer vision, and shape analysis, where moment profiles along slices are used to recover or classify geometric objects \cite{FlusserSuk1994}. Motivated by these applications, we compute moment formulas for slices and slabs of the cube and analyze them in detail for slices in dimension $2$.

This work is complementary in both method and emphasis to prior 
work. Rather than starting from Fourier transforms or analytic inequalities, we develop a unified
\emph{algebraic--combinatorial} framework that partitions the parameter space into finitely many
maximal chambers in which the combinatorics is constant, and computes volumes
and moments on each chamber by abstract triangulations and the use of simplex-integration formulas from \cite{intsimplex}. This framework was first introduced in \cite{slices} in the realm of slices of polytopes; in this article, we extend this framework to slabs of centrally symmetric polytopes, i.e., polyhedral norm balls, and illustrate the methods on the slices and slabs of the hypercube, extending the results of \cite{KK11}, and proving polynomial algorithms in computing volume and moment formulas for slices and slabs of polyhedral norm balls.

\subsection{Our Contributions}
\label{sec:our_contributions}

\indent
Our first contribution is an explicit computational algorithm to obtain
formulas for moments of {\bf all} slices and slabs of any polyhedral norm ball. They turn out to be piecewise rational and are valid within a chamber decomposition of the unit sphere 
which parametrizes and governs volumes and higher moments.

\paragraph{Algorithm for moments of slices and slabs of polyhedral norm balls.}

Given a $d$-dimensional polyhedral norm ball, our algorithm outputs piecewise rational formulas for the moments of slices and slabs in polynomial time in the input size. Note that volume is computed by setting $M=0$. Each piece of the function is supported on one chamber of a hyperplane arrangement (partially restricted to the sphere). Combining this parametrization with an integration formula from \cite{intsimplex} yields the following.

\begin{theorem}[Polynomial-time moment computation] \label{thm:main}
    Fix $d\in\N$ and a polyhedral norm $\| \cdot \|$ on $\R^d$. For any $M \in \mathbb Z_{\geq 0}$, the algebraic method introduced in \Cref{sec:methodology} computes the $M$-th moments $\mslice$ and $\mslab$ in polynomial time in the input size, as functions of $(a,t)\in S^{d-1}\times \R$.
\end{theorem}

\paragraph{The complete family of volume formulas of slices and slabs of the $4$-cube.}
The main computational contribution of this article concerns volumes and higher moments of the hypercube up to dimension $4$. 
Our computations recover the known volume formulas of K\"onig and Koldobsky in dimensions 2 and 3 \cite{KK11}. 
We extend their work by finding explicit formulas for both slabs and slices in dimension $4$, and their volumes are captured by $14$ distinct rational functions. Note that, up to the symmetry of the cube, these are all possible formulas.

\begin{theorem}[Volume formulas of the 4-dimensional cube]\label{thm:4dvol}
    Let $a = (a_1, a_2, a_3, a_4) \in S^3$ such that $a_1 \geq a_2 \geq a_3 \geq a_4 \geq 0$, and let $t \geq 0$. The volumes of slices and slabs of the $4$-dimensional cube are piecewise rational functions of $(a,t)$. Modulo signed permutations of the coordinates, these volumes admit a representation with exactly \emph{fourteen} distinct rational functions, each representing the volume of the corresponding slices or slabs for which $(a,t)$ is contained in a maximal (closed) chamber of a hyperplane arrangement associated with $\hypercube^4$.
\end{theorem}

\paragraph{Critical values of the volumes for the 4-dimensional cube.} From these formulas, we can examine the critical points of the volume function. We verified that the maximizers/minimizers in the statement are critical points by exact computation in \texttt{Macaulay2}. Unfortunately, we cannot turn this implicit description into an explicit one for all the chambers. This is due to the algebraic complexity of the set of critical points.
However, the contributions for the maxima and minima seem to be given by critical points for which we have an explicit description. This has been verified numerically using \texttt{Mathematica} and a plot of maxima and minima values can be seen in \Cref{fig:4cube_maxmin}. 

\begin{figure}[h]
    \centering
    \includegraphics[width=0.48\linewidth]{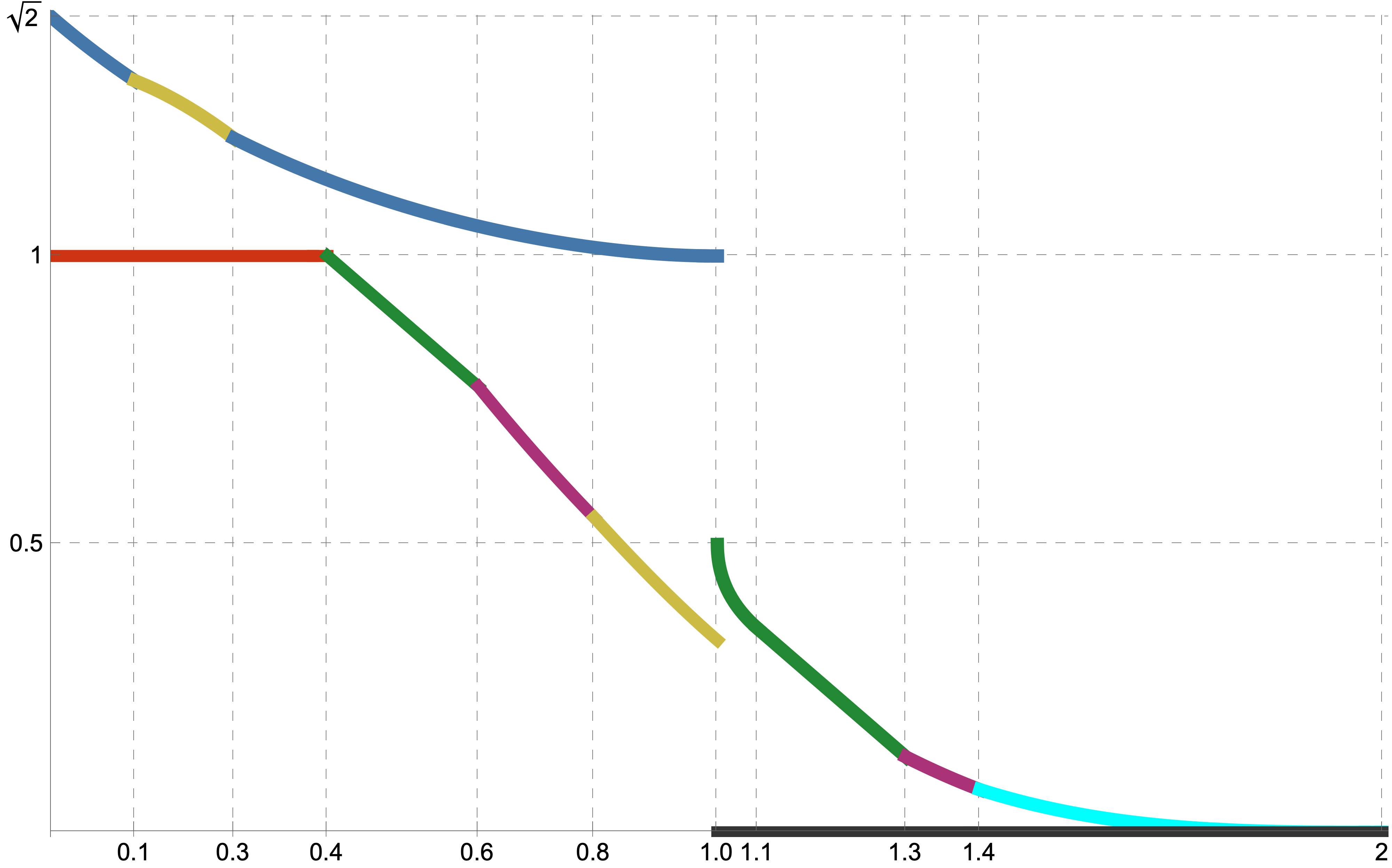}
    \includegraphics[width=0.48\linewidth]{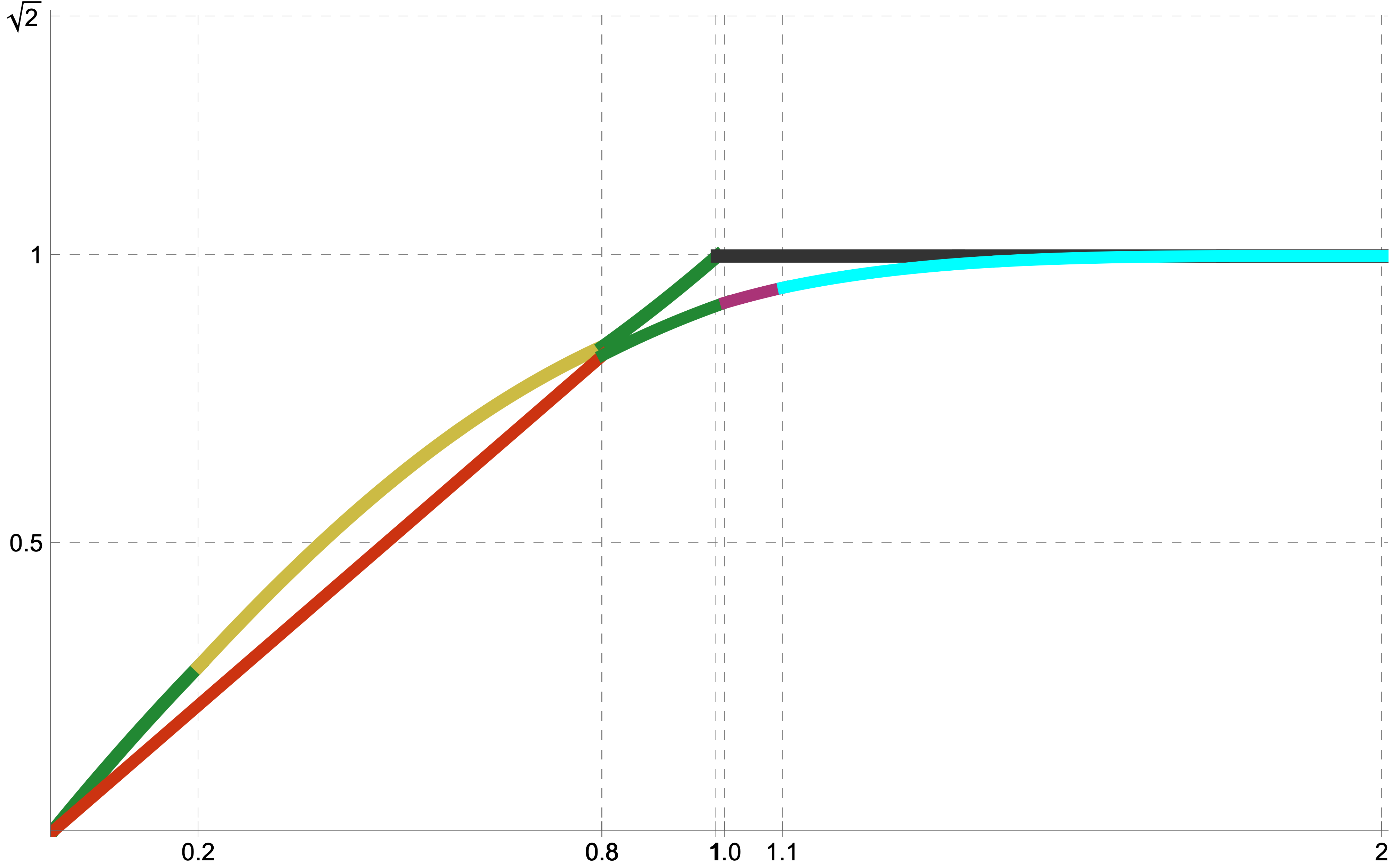}
    \caption{The graphs of the global maxima and minima of the volume of the slices (left) and slabs (right) of the unit cube in $\R^4$ at distance $\frac{t}{2}$, for $t\in[0,2]$. Colors correspond to different volume functions in our construction. See \Cref{subsec:extreme4cube} for details.}
    \label{fig:4cube_maxmin}
\end{figure}

\noindent
From this information, we can state the following conjectures, a four-dimensional counterpart to \cite{KK11}.

\begin{conjecture}\label{conj:extremal_slices}
    The extremal values of the volume of the slices of the cube of volume $1$ in $\R^4$ at distance $\frac{t}{2}$ are as follows (see \Cref{fig:4cube_maxmin} for a plot of their graph). On the left we report one extremizer, on the right the extremal value, in the center the domain for this pairing.\\
    Absolute minimum:
\[
\begin{array}{ccl}
(1,0,0,0) & 0 \leq t \leq t_3
& \displaystyle 1,
\\[0.8em]

\tfrac{1}{\sqrt{2}}(1,1,0,0) & t_3 \leq t \leq t_4
& \displaystyle \sqrt{2}-t,
\\[0.8em]

\tfrac{1}{\sqrt{3}}(1,1,1,0) & t_4 \leq t \leq t_5
& \displaystyle \tfrac{3\sqrt{3}}{8}
\left(t^2-2\sqrt{3}\,t+3\right),
\\[0.8em]

\tfrac{1}{2}(1,1,1,1) & t_5 \leq t \leq t_6
& \displaystyle t^3 - 2t^2 + \frac{4}{3},
\\[0.8em]

(1,0,0,0) & t_6 < t \leq 2
& \displaystyle 0.
\end{array}
\]
    Absolute maximum:
    \[
\begin{array}{ccl}
\tfrac{1}{2}(\sqrt{2-t^2}+t,\sqrt{2-t^2}-t,0,0) & 0 \leq t \leq t_1
& \displaystyle \frac{\sqrt{2-t^2}-t}{1-t^2},
\\[0.8em]

\tfrac{1}{2}(1,1,1,1) & t_1 \leq t \leq t_2
& \displaystyle t^3 - 2t^2 + \frac{4}{3},
\\[0.8em]

\tfrac{1}{2}(\sqrt{2-t^2}+t,\sqrt{2-t^2}-t,0,0) & t_2 \leq t \leq t_6
& \displaystyle \frac{\sqrt{2-t^2}-t}{1-t^2},
\\[0.8em]

\tfrac{1}{2}(t+\sqrt{t^2-1}+\xi(t), t+\sqrt{t^2-1}-\xi(t),0,0) & t_6 < t \leq t_7
& \displaystyle \frac{t-\sqrt{t^2-1}}{2},
\\[0.8em]

\tfrac{1}{\sqrt{2}}(1,1,0,0) & t_7 \leq t \leq t_8
& \displaystyle \sqrt{2}-t,
\\[0.8em]

\tfrac{1}{\sqrt{3}}(1,1,1,0) & t_8 \leq t \leq t_9
& \displaystyle \tfrac{3\sqrt{3}}{8}
\left(t^2-2\sqrt{3}\,t+3\right),
\\[0.8em]

\tfrac{1}{2}(1,1,1,1) & t_9 \leq t \leq 2
& \displaystyle \frac{(2-t)^3}{3},
\end{array}
\]
where $\xi(t)=\sqrt{-2 t^2-2 t\sqrt{t^2-1} +3}$, and the exact expressions for the $t_i$ can be obtained, by continuity as the value for which the left limit and right limit coincide. Up to 3-digit precision, we get 
\[
(t_1, \ldots, t_9) = (0.125,
0.274,
0.414, 0.641, 0.815, 1, 1.061, 1.284, 1.395).
\]
\end{conjecture}

\begin{conjecture}\label{conj:extremal_slabs}
The extremal values of the volumes of the slabs of the cube of volume $1$ in $\R^4$
are as follows. On the left we have one extremizer, on the right the
extremal value, and in the center the domain for this pairing.\\
Absolute minimum:
\[
\begin{array}{ccl}
(1,0,0,0) & 0 \leq t \leq \tau_2
& \displaystyle t,
\\[0.8em]

\tfrac{1}{\sqrt{2}}(1,1,0,0) & \tau_2 \leq t \leq \tau_5
& \displaystyle \sqrt{2}\,t-\frac{t^2}{2},
\\[0.8em]

\tfrac{1}{\sqrt{3}}(1,1,1,0) & \tau_5 \leq t \leq \tau_6
& \displaystyle 1+\frac{\sqrt{3}}{8}\left(t-\sqrt{3}\right)^3,
\\[0.8em]

\tfrac{1}{2}(1,1,1,1) & \tau_6 \leq t \leq 2
& \displaystyle 1-\frac{(2-t)^4}{12}.
\end{array}
\]
Absolute maximum:
\[
\begin{array}{ccl}
\tfrac{1}{\sqrt{2}}(1,1,0,0) & 0 \leq t \leq \tau_1
& \displaystyle \sqrt{2}\,t-\frac{t^2}{2},
\\[0.8em]

\tfrac{1}{2}(1,1,1,1) & \tau_1 \leq t \leq \tau_3
& \displaystyle \frac{t}{12}\left(16-8t^2+3t^3\right),
\\[0.8em]

\tfrac{1}{2t}(1+\sqrt{2t^2-1}, 1-\sqrt{2t^2-1}, 0, 0) & \tau_3 \leq t \leq \tau_4
& \displaystyle \frac{1+t^2}{2},
\\[0.8em]

(1,0,0,0) & \tau_4 \leq t \leq 2
& \displaystyle 1.
\end{array}
\]
Here the exact transition values are obtained, by continuity, as the relevant
values for which the left limit and right limit coincide. Up to $3$-digit precision we get
\[
(\tau_1,\ldots,\tau_6)=(0.222,0.828,0.829,1, 1.013,1.100).
\]
\end{conjecture}

\paragraph{Higher moments and critical-point computations in low dimensions.}
We complement the computational results with explicit piecewise rational functions for moments of order $M = 1,2,3,4$ for slices and slabs of the cube in dimensions $2,3,4$, and demonstrate how to extract critical points by algebraic methods.
Moreover, for dimension $2$ we prove explicit closed formulas for $\int_{\slicetxt(a,t,\|\cdot\|_\infty)} \sum_{i=1}^d x_i^M$ for arbitrary $M$, as well as for their critical points. For this purpose, let us introduce chambers
\begin{equation}\label{eq:chambers_square}
\begin{aligned}
C_{1,1} &= \{(a,t)\in S^1\times\R \mid 0 < a_2 < a_1,\; 0 < t < a_1-a_2\},\\
C_{1,2} &= \{(a,t)\in S^1\times\R \mid 0 < a_2 < a_1,\; a_1-a_2 < t < a_1+a_2\}.
\end{aligned}
\end{equation}
For details on the construction of these chambers, we refer the reader to Section \ref{sec:equiClass_vertexOrdering}. 

\begin{theorem}[Algebraic critical points of moments of square slices] \label{thm:2Dslicemoment}
    The critical points of the $M$-th moment of slices of the square, for fixed $t\in[0,\sqrt{2}]$, are exactly the points satisfying one of the following conditions:
    \[
    a_1-a_2 = t, \; \text{ or } \; a_1+a_2 = t, \; \text{ or } \; (a_1,a_2)=(1,0), \; \text{ or }
    \] 
    \begin{align*}
        \bigl(a_{2}^{2}+M\bigr)\,A_+(a_2,t) - t\,\bigl((M+2)\,a_{2}^{2} - 1\bigr)\,A_-(a_2,t) &= 0, \qquad  (a,t) \in C_{1,1}, \; M \text{ odd}, \\
        2 a_1^M a_2^2
    +(a_2^2+M)A_+(a_2,t)
    -t (a_2^2(M+2)-1)A_-(a_2,t) &= 0, \qquad  (a,t) \in C_{1,1}, \; M \text{ even}, \\
    a_1^M (t - a_1)^M p(a_1,t,M) + 
    - a_2^M (t - a_2)^M p(a_2,t,M) + a_1^M a_2^M (a_1^3 - a_2^3)  &= 0, \qquad  (a,t) \in C_{1,2},
    \end{align*}
    where 
    \begin{gather*}
    A_+(\alpha,t) = (t-\alpha)^{M} + (t+\alpha)^{M}, \qquad
    A_-(\alpha,t) = \frac{(t-\alpha)^{M} - (t+\alpha)^{M}}{\alpha}, \\
    p(\alpha,t,M) = \alpha^3 -t(M+2)\alpha^2+M\alpha+t.
    \end{gather*}
\end{theorem}

Although we are not able to provide closed formulas for arbitrary moments of slices or slabs of hypercubes of any higher dimension, we show that they share notable properties with the rational functions of the respective volumes. Here, by degree of a rational function we mean the difference of degrees of numerator and denominator. 

\begin{theorem}[Moments of slices and slabs] 
\label{thm:moment_formula}
    For any $d \in \N$, $M \in \mathbb Z_{\geq 0}$, the moments $\int_{\slicetxt(a,t,\|\cdot\|_\infty)} \sum_{i=1}^d x_i^M$ and $\int_{\slabtxt(a,t,\|\cdot\|_\infty)} \sum_{i=1}^d x_i^M$  of slices and slabs of the cube $B_\infty^d$ are piecewise rational functions whose pieces are supported on the same domains as the rational function expressions of the respective volume formulas of $\sliceinf$ and $\slabinf$. Moreover, in all computed cases $M=1,2,3,4$, the degrees of the rational functions for the $M$-th moment coincide with the degrees of the corresponding volume formulas.
\end{theorem}

For the domains on which each rational expression is valid, see Table~\ref{table:4d_kkchamber}. This result is far from a coincidence but reflects a deeper structural fact: all moments of both slices and slabs are governed by the same hyperplane arrangement, and it is precisely this shared combinatorial geometry that underlies our algebraic framework. 
Throughout this paper, we assume the reader is familiar with basic background on polytopes and hyperplanes. If needed, we refer the reader to \cite{Gruenbaum03:Polytopes}. 

\paragraph{Organization.}
\Cref{sec:methodology} presents the algebraic-combinatorial framework that guides the construction and computation of the rational expressions for volumes and moments of slices and slabs of a polyhedral norm ball.
In \Cref{sec:volumes} we present explicit formulas for the volumes and moments of slices and slabs of $\hypercube^d$ in $d=2,3,4$ and of order $M=0,1,2,3,4$.
Moreover, in \Cref{sec:2d} we give closed-form formulas for all moments of slices of $\hypercube^2$ for arbitrary $M\in \mathbb Z_{\geq 0}$ and their critical points.
All explicit formulas discussed in the paper are collected in the \hyperref[appendix]{Appendix}. 

%=============================================================
\section{Methodology} \label{sec:methodology}

In this section, we lay out the theoretical framework underlying our computations. In \Cref{sec:equiClass_vertexOrdering} we summarize the results of \cite{slices} for an algorithmic approach of computing volumes of slices of polytopes. 
In \Cref{sec:triangulation,sec:triangulation_moment,sec:thm1pf} we extend this framework to volumes of slabs, then to higher moments of slices and slabs of centrally symmetric polytopes, i.e., polyhedral norm balls. Finally, in \Cref{sec:algebraic} we illustrate our algebraic approach for the study of critical points of these functions.

\subsection{Equivalence Classes of Slices via Vertex Ordering}
\label{sec:equiClass_vertexOrdering}
We begin by introducing a construction that is central to our method. Intuitively, when slightly perturbing a hyperplane, the expectation is that its intersection with a fixed polytope has the same combinatorial type and that the two sections admit a common triangulation, therefore bearing the same volume and moment formulas.
We use this observation to group hyperplanes into equivalence classes, where in each class the combinatorial type for the associated slices or slabs is the same. We refine these equivalence classes by \emph{chambers}, whose definition is based on vertex orderings induced by linear functionals associated with the slicing hyperplanes. We now give a brief overview of this construction. For more details, we refer to \cite[Sections 2.2 and 3.2]{slices}. 

\begin{definition}[Sweep arrangement]
Let $P \subset \mathbb{R}^d$ be a full-dimensional polytope with vertex set $V \subset \R^d$.
The
\emph{sweep arrangement} associated to $P$ 
is
\[
\mathcal{R}(P) = \{ (v-w)^{\perp} \mid v \neq w  \in V\},
\]
where $v^{\perp} = \{ x \in \mathbb{R}^d \mid \langle x, v \rangle = 0 \}$ is the hyperplane orthogonal to $v$. We denote by $R_i$ the regions of $\mathcal{R}(P)$ restricted to the unit sphere, namely the connected components of $S^{d-1} \setminus \mathcal{R}(P)$.
\end{definition}

Given a generic vector $a \in S^{d-1}\subset \R^d$, the \emph{induced ordering} on $V$ is the unique labeling $V = \{v_1,\dots, v_n\}$ such that $\langle a, v_i \rangle < \langle a , v_{i+1} \rangle$ for all $i=1,\dots,n-1$.
The region $R_i$ of the sweep arrangement is by construction such that for all $a \in R_i$, the linear functional $\langle a, \cdot \rangle$ induces the same ordering $v_1, \ldots, v_n$ of the vertices of $P$.

\begin{definition}[Maximal Chamber, \cite{slices}] \label{def:maximalChamber}
Consider a full-dimensional polytope $P \subset \R^d$ with vertex set $V = \{v_1,\ldots,v_n\}$, where the vertices $v_i$ follow the same ordering as described above. For a fixed region $R_i$ of $S^{d-1} \setminus \mathcal{R}(P)$ and a vector $a \in R_i$, any index $j \in \lfloor \frac{n}{2} \rfloor$ defines an open interval $I_{i,j}(a) = \{t \in \R \mid  0 \leq \langle a, v_j \rangle < \frac{t}{2}  < \langle a, v_{j+1} \rangle\}$ for $j>0$, with $I_{i,0}(a) = \{t \in \R \mid  0 < \frac{t}{2} < \langle a, v_{1} \rangle\}$ where $v_1$ is the first vertex of $P$ such that $\langle a, v_{1} \rangle>0$. 
A \emph{maximal open slicing chamber}, or just \emph{maximal chamber}, is the polyhedron
\[
C_{i,j} = \{(a,t) \in S^{d-1}\times \R \mid a \in R_i, t \in I_{i,j}(a) \} \subset \R^{d+1}.
\]
\end{definition}

Regions control vertex orderings in direction space; chambers refine this by tracking which faces of $P$ are intersected as $t$ varies.
We say that two polytopes $Q_1,Q_2$ from the same chamber admit the \emph{same triangulation} if there exist triangulations $T_i$ of $Q_i$ and a bijection between the sets of all vertices in $T_1$ and those in $T_2$, which sends simplices to simplices.
We collect in the following lemma a few useful properties of the maximal chambers.
\begin{lemma}[{\cite[Subsections 2.2 and 3.2]{slices}}]\label{lemma:summaryofBDLMpaper}
A maximal chamber satisfies the following properties:
\begin{enumerate}[(i)]
    \item For all $(a,t) \in C_{i,j}$, the hyperplane $H(a,t) = \{ x \in \R^d \mid \langle a, x \rangle = \frac{t}{2} \}$
    intersects the same fixed set of edges of $P$.
    \item Consequently, for any $(a_1, t_{1}), (a_2,t_{2}) \in C_{i,j}$, the slices $P \cap H(a_{1},t_{1})$ and $P \cap H(a_{2},t_{2})$ 
    are combinatorially equivalent thus admit the same triangulation.
	 \item Each coordinate of a vertex of $P \cap H(a,t)$ is a rational function in the variables $(a,t)$. This rational function is valid for all $(a,t) \in C_{i,j}$.
	 \item Consequently, the volume of $P\cap H(a,t)$ and the moment $\int_{P \cap H(a,t)} \sum_{i=1}^d x_i^M$ are rational functions in the variables $(a,t)$, when restricted to $(a,t) \in C_{i,j}$.
\end{enumerate}
\end{lemma}

As stated in Theorem~1 of \cite{slices}, given any polytope \(P \subset \mathbb{R}^d\), there are only finitely many slicing chambers of \(P\). 
\Cref{lemma:summaryofBDLMpaper} implies that the 
volume formulas and moment formulas for a family of slices restricted to a chamber $C_{i,j}$ are determined by the fixed set of edges that the corresponding hyperplanes intersect.
In order to obtain this set, it suffices to find a single representative $(a_{\mathrm{rep}}, t_{\mathrm{rep}}) \in C_{i,j}$. Given $a_{\mathrm{rep}}$, a canonical choice is $t_{\mathrm{rep}} = \frac{1}{2}(\langle a_{\mathrm{rep}}, \mathbf{v}_j \rangle+\langle a_{\mathrm{rep}}, \mathbf{v}_{j+1} \rangle)$. 
This is enough to compute the intersecting set of edges, the combinatorial type of the slices, the parametrization of the vertices, and finally the volume and moment formulas in the maximal chamber $C_{i,j}$. See \cite{Schrijver1986} for details.

Suppose we know that $H(a,t)$, where \((a,t) \in C_{i,j}\), intersects an edge \(e\) of polytope $P$. Let $v$ and $w$ be the two vertices of $e$. The parametrization of the vertices of a slice works as follows. The intersection point $\mathbf{x}_a(t) = e \cap H(a,t)$ is a vertex of the slice \(P \cap H(a,t)\). Its coordinates are given by
\begin{equation}
    \mathbf{x}_a(t)
    = \frac{\frac{t}{2}}{\langle a, w - v \rangle}
      (w - v)
      + \frac{\langle a, w \rangle v
     - \langle a, v \rangle w}
      {\langle a, w - v \rangle},
\label{formula:symbolic_vertex}
\end{equation}
which is a rational function in the variables $(a_1,\dots,a_d,t) \in S^{d-1}\times \R$.

These observations imply a simple algorithm to obtain the piecewise rational volume formula of $P \cap H(a,t)$ in terms of $a$ and $t$: First, for each maximal slicing chamber $C_{i,j}$ of the sweep arrangement $\mathcal R(P)$, compute a single $(a_{\mathrm{rep}},t_{\mathrm{rep}}) \in C_{i,j}$. From $P \cap H(a_{\mathrm{rep}},t_{\mathrm{rep}})$ we obtain the combinatorial type of all slices $P \cap H(a,t)$ for all $(a,t) \in C_{i,j}$, with which we can compute a common triangulation for this family of slices. The set of edges of $P$ which are intersected by $H(a_{\mathrm{rep}},t_{\mathrm{rep}})$ can be used to compute the parametrization of the vertices of $P \cap H(a,t)$ for all $(a,t) \in C_{i,j}$. Using this information, one can compute the rational function expression for the volume of a simplex in the common triangulation. Summing the volume of the simplices, we obtain the volume formulas for the entire family $P \cap H(a,t), (a,t) \in C_{i,j}$. Iterating over all maximal slicing chambers then yields the piecewise rational function.

%-------------------------------------------------------------------------------------

\subsection{From Slices to Slabs}
\label{sec:triangulation}

The method described in \Cref{sec:equiClass_vertexOrdering} can now be generalized from $(d-1)$-dimensional slices of a polytope to its $d$-dimensional slabs. We begin by discussing in some detail what the vertices of a slab are, and how to triangulate it in order to compute a volume expression of a slab within a single maximal slicing chamber. 

\paragraph{Vertices of a slab}
A slab has two types of vertices. The first type consists of vertices of the polyhedral norm ball \(P\) itself. These are easily identified by collecting all vertices \(v_i\) of $P$ such that $\lvert \langle a, v_i \rangle \rvert \le \tfrac{t}{2}$.
The second type consists of vertices created by the intersection of the hyperplanes \(H(a,t)\) and \(H(a,-t)\) with edges of \(P\).
\eqref{formula:symbolic_vertex} ensures that the second type of vertices has a rational expression in the parameters $(a,t)$. Together we have the set of vertices $V$, which are used directly in the volume computation. 

\paragraph{Triangulation computation}
In order to find the rational function expressions for the volume of $\slab$, we triangulate it and find the volume of each simplex. The triangulation we use in this paper is the \emph{barycentric triangulation}. This is a classical and widely used method in computational and combinatorial geometry \cite{Edelsbrunner87,Ziegler95,DeLoeraRambauSantos2010Triangulations}. Intuitively, after determining the faces of the slab via their supporting hyperplanes, each two-dimensional face is subdivided by introducing its barycenter and coning to the edges. $k$-dimensional faces are triangulated inductively by coning to the barycenters of all its $k-1$-dimensional bounding faces. This produces the canonical barycentric subdivision of the slab. 
We then compute the volume of the slab by adding the (signed) determinants of simplices in the triangulation. This approach allows us to compute volume and moments symbolically as piecewise rational functions in $(a,t)$. 

\begin{example}\label{ex:triangulation}
Consider the slab of the three-dimensional cube $\operatorname{slab}((\frac{2}{\sqrt{6}},\frac{1}{\sqrt{6}},\frac{1}{\sqrt{6}}),\frac{\sqrt{6}}{2},\hypercube^3)$, displayed in \Cref{fig:triangulation}, left, 
and on the right its barycentric triangulation. 
Each supporting hyperplane of the cube contributes one facet of the slab, and coning from a single interior point of each facet induces a barycentric triangulation. 
\end{example}
\begin{figure}[ht!]
\centering
    \includegraphics[width=0.4\linewidth]{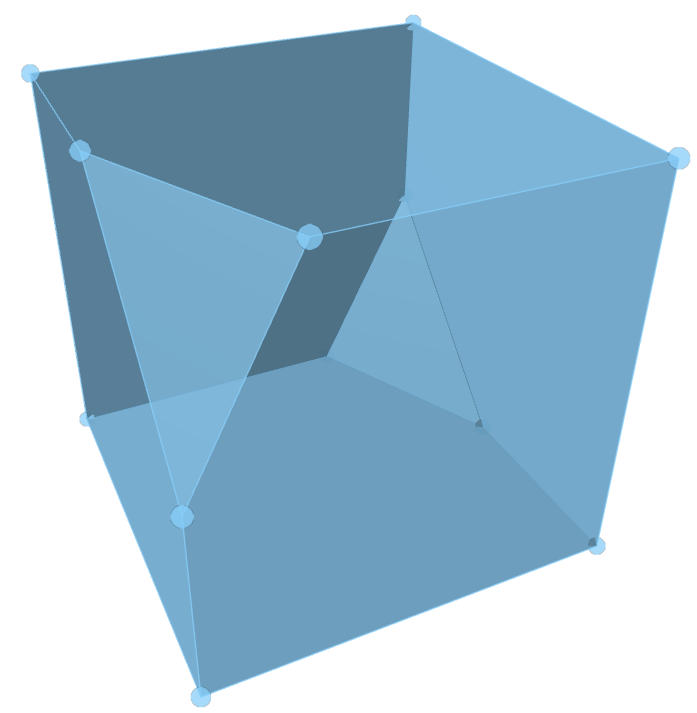}
    \qquad
    \includegraphics[width=0.4\linewidth]{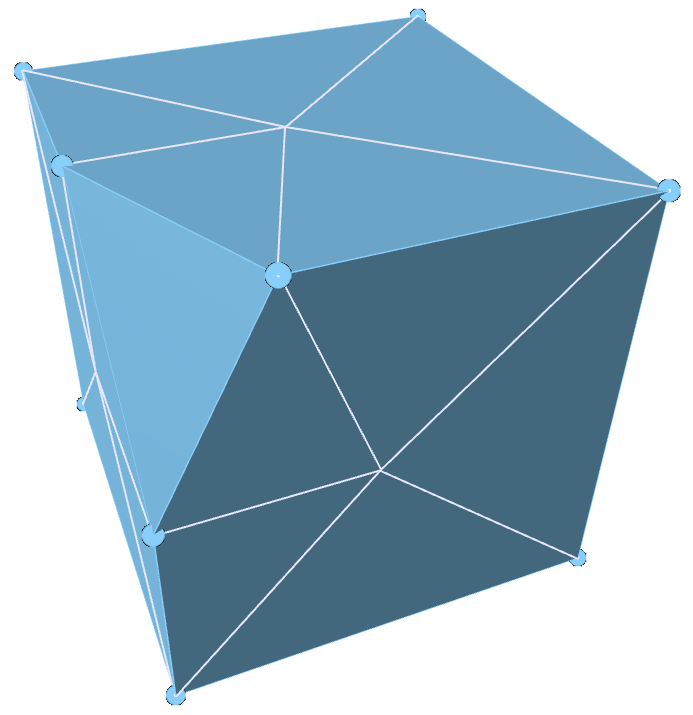}
\caption{Left: $\operatorname{slab}((\frac{2}{\sqrt{6}},\frac{1}{\sqrt{6}},\frac{1}{\sqrt{6}}),\frac{\sqrt{6}}{2},\hypercube^3)$. Right: its barycentric triangulation, from \Cref{ex:triangulation}.}
\label{fig:triangulation}
\end{figure}

\begin{lemma}\label{lem:slab-chambers}
   Let $P = B^d_{\|\cdot\|}$ be a polyhedral norm ball. Given two hyperplanes $H(a,t)$ and $H(a',t')$ that belong to the same maximal slicing chamber,
   the associated slabs $\slab$ and $\slabtxt(a',t',\|\cdot\|)$ have the same combinatorial type and the same parametrization of their vertices. Moreover, the volume of each parametrized slab is a rational function. As a consequence, by listing all possible maximal chambers, we recover the volume formulas of all slabs of $P$. 

By continuity of the volume function, the formulas associated with two adjacent chambers agree along their common boundary. Therefore, if a slicing direction lies on the boundary between two maximal slicing chambers, its volume may be computed using the formula from either chamber.
\end{lemma}

\begin{proof}
    As mentioned in \Cref{sec:equiClass_vertexOrdering}, the number of maximal slicing chambers is finite. Fix such a maximal chamber $C_{i,j}\subset S^{d-1}\times\R$.
    By part (i) of \Cref{lemma:summaryofBDLMpaper}, the hyperplane $H(a, t)$ intersects the same set of edges of $P$ for any choice of $(a,t) \in C_{i,j}$. Central symmetry of $P$ implies the same conclusion for $H(a, -t)$.
    The parametrization \eqref{formula:symbolic_vertex} of a vertex obtained by intersecting an edge of $P$ with either $H(a, t)$ or $H(a, -t)$ is a rational function in $C_{i,j}$. Moreover, the set of vertices of $P$ contained in the strip $\{x \in \R^d \mid -\tfrac{t}{2} < \langle a,x \rangle < \tfrac{t}{2}\}$ between the two hyperplanes is fixed, since all $a \in R_i$ determine the same vertex ordering. Thus, the combinatorial type and the parametrization of the vertices of $\slab$ are fixed for all $(a,t) \in C_{i,j}$, in particular for $(a',t')$.
    
    Now fix a triangulation of the slabs from $C_{i,j}$, such that the set of vertices forming a simplex depends only on the face lattice of the slabs of $P$ in this chamber (as above, this can be done with barycentric triangulations). Similarly to the above, all vertices of the triangulation admit rational parameterizations. The volume of a slab can then be computed as follows. For each $d$-dimensional simplex with vertices $s_0,\dots,s_{d}$, up to sign the volume equals 
    \begin{equation}\label{eqn:determinant}
    \frac{1}{d!}\det(s_1 - s_0,\dots, s_d - s_0).
    \end{equation}

    As for slices, the simplices have codimension one, and therefore if their vertices are denoted by $s_1,\ldots,s_d$, their volume can be computed as
\begin{equation}
    \operatorname{vol}(\Delta)
    =
    \frac{1}{(d-1)!}
    \det \left(\left[ s_2 - s_1, s_3 - s_1, \dots, s_{d} - s_1, \frac{a}{\|a\|}\right]\right),
    \label{eqn:determinant_slice}
\end{equation}
where $a$ is any vector orthogonal to the affine span of the simplex. Since the affine span is by definition of the slice $H(a,t)$ and $a \in S^{d-1}$, the norm can be dropped. In both cases, the volume of one simplex is a rational function in the variables  $(a,t)$.

Since each of these vertices is a rational function in the variables $a_1,\dots,a_d,t$, the same holds for the volume of the simplex. The volume of $\slab$ is then the sum of the volumes of these parametric simplices, and thus a rational function as well.
We note that this rational function expression is not only valid on the open set $C_{i,j}$, but, by continuity of the volume function, also on the boundary of its Euclidean closure. 
\end{proof}

\Cref{lem:slab-chambers} implies that the number of rational function expressions for the volume of slabs is finite, and that the number of such expressions is bounded from above by the number of maximal slicing chambers $C_{i,j}$. Moreover, the proof of \Cref{lem:slab-chambers} reveals that small modifications of the algorithm described in \Cref{sec:equiClass_vertexOrdering} yield an algorithm to compute rational function expressions for the volume of slabs in terms of $a$ and $t$.

%---------------------------------------------------------------
\subsection{Volume and Higher Moment Formulas}
\label{sec:triangulation_moment}
To compute the moments of slices and slabs, our approach begins by triangulating them, as outlined in \Cref{sec:triangulation}. 
For each simplex in the triangulation, we compute rational functions for the moments using \cite[Lemma 8]{intsimplex}, which provides a closed-form expression for the integral of a power of a linear form over a simplex. Let $M \in \mathbb{Z}_{\geq 0}$, and let $\Delta = \conv(s_1, s_2,\dots, s_{n+1})$ be an $n$-dimensional simplex in $\mathbb{R}^d$. For any linear form $\ell$ on $\mathbb{R}^d$,
\begin{equation}
\int_{\Delta} \ell^M \, \mathrm{d}x
  = n!\,\mathrm{vol}(\Delta)\,
    \frac{M!}{(M+n)!}
    \sum_{\substack{k \in \mathbb{N}^{n+1}\\ |k|=M}}
    \langle \ell, s_1 \rangle^{k_1}
    \cdots
    \langle \ell, s_{n+1} \rangle^{k_{n+1}},
\label{eqn:moment}
\end{equation}
where we integrate with respect to the standard Lebesgue measure.
In the case of $M=0$, the integral in \Cref{eqn:moment} reduces to computing the volume of a simplex. 
We saw in the proof of \Cref{lem:slab-chambers} that the volume of a full-dimensional simplex can be expressed (up to sign) as a determinant (see \Cref{eqn:determinant_slice,eqn:determinant}). Therefore, we can use that formula for the full-dimensional simplices in the triangulation of the slices and slabs.

Recall from \Cref{formula:symbolic_vertex} that every vertex (of slices and slabs) is a rational function of $(a,t)\in C_{i,j}$ for each chamber. Thus, \Cref{eqn:moment} implies that for any slice or slab all the $M$-moments are piecewise rational functions in $(a,t)\in S^{d-1}\times\R$. 
\begin{remark}
    We stress that to make the moment formulas of the slices piecewise rational, one needs to require $a\in S^{d-1}$. Otherwise, for $a\in \R^d$, the formulas are merely semialgebraic functions, as they have a dependence on $\|a\|$. However, in the case of slabs, since there is no normal vector $a$ involved in the volume expression of the simplex, one could allow $a\in \R^d$. In order to make the narration coherent and the description of slices and slabs non-redundant, we require $a\in S^{d-1}$ in both cases. 
\end{remark}
We conclude this section by presenting a concrete computational example of a moment calculation of one specific family of slabs of the $2$-dimensional cube.
\begin{example}\label{ex:2cube_2moment}
    Assume $d=2$ and $M=2$. To compute all rational functions for the second moment of the slabs of the square, we begin by finding the maximal open slicing chambers. 
    Up to symmetry (permutations and reflections), there is exactly one possible vertex ordering, namely one region of the sweeping arrangement which we denote
    \[
    R_1 = \{(a_1,a_2)\in S^1 \mid a_1 > a_2 > 0\}.
    \]
    This induces the vertex ordering $v_1, v_4, v_2, v_3$, as shown in \Cref{fig:2d_slabs}, center. Then, the ranges for $t>0$ associated with this region are
    $I_{1,1}(a) = (0, 2 \langle v_2, a \rangle), 
    I_{1,2}(a) = (2 \langle v_2, a \rangle, 2\sqrt{2}).$
    Indeed, when $t/2 \in I_{1,1}(a)$, the slab is a parallelogram (shown in \Cref{fig:2d_slabs}, left) and when $t \in I_{1,2}(a)$, the slab is a hexagon (shown in \Cref{fig:2d_slabs}, right). Thus, to find formulas for all combinatorially different slabs in this region, we need to investigate the two associated slicing chambers $C_{1,1}$ and $C_{1,2}$, introduced already in \eqref{eq:chambers_square}. Let $a_{\mathrm{rep}} = \frac{1}{\sqrt{5}}(2,1) \in R_1$ be a representative vector, then $I_{1,1}(a_{\mathrm{rep}}) = (0,\frac{1}{\sqrt{5}})$ and $I_{1,2}(a_{\mathrm{rep}}) = (\frac{1}{\sqrt{5}}, \frac{3}{\sqrt{5}})$. After choosing two representatives $t_{\mathrm{rep}}$ for $t$, namely  $\frac{1}{2\sqrt{5}} \in I_{1,1}(a_{\mathrm{rep}})$ and $\frac{2}{\sqrt{5}} \in I_{1,2}(a_{\mathrm{rep}})$, we can find the edges that are intersected by $H(a_{\mathrm{rep}},t_{\mathrm{rep}})$, and with this the parametric expressions for the vertices of the two types of slabs. Since the procedure is analogous in both chambers, we only illustrate the hexagonal case where $(a_{\mathrm{rep}},t_{\mathrm{rep}})\in C_{1,2}$.
    \begin{figure}[ht]
    \centering
    \input{Images/2D_chambers}
    \caption{Center: hyperplane sweep across $B_{\|\cdot\|_\infty}$ in $\mathbb{R}^2$. Left: quadrilateral slab associated to $(a_{\mathrm{rep}},t_{\mathrm{rep}})\in C_{1,1}$. Right: hexagonal slab associated to $(a_{\mathrm{rep}},t_{\mathrm{rep}})\in C_{1,2}$. The rational functions describing the volumes of these two types of slabs are distinct and are computed in \Cref{ex:2cube_2moment}.}
    \label{fig:2d_slabs}
\end{figure}

    When $t \in I_{1,2}(a_{\mathrm{rep}})$, the slab has six vertices. Four of them arise from the intersection of the square with the sweeping hyperplanes at distance $\frac{t}{2}$ from the origin: these are $(\frac{1}{2} , \frac{t-a_1}{2 a_2})$, $(\frac{t-a_2}{2 a_1}, \frac{1}{2})$, $(-\frac{1}{2} , \frac{a_1-t}{2 a_2})$, $(\frac{a_2-t}{2 a_1}, -\frac{1}{2})$. Additionally, there are two vertices of the slab which are vertices of the square itself. To find these, we check the inequalities $|\langle a_{\mathrm{rep}},v_i \rangle| \leq t/2$ at all vertices $v_i$ of the square. Thus, the complete list of vertices of the hexagonal slab is 
    \[
    \left\lbrace (-\tfrac{1}{2},\tfrac{1}{2}), (\tfrac{1}{2},-\tfrac{1}{2}), (\tfrac{1}{2}, \tfrac{t-a_1}{2 a_2}), (\tfrac{t-a_2}{2 a_1}, \tfrac{1}{2}), (-\tfrac{1}{2} , \tfrac{a_1-t}{2 a_2}), (\tfrac{a_2-t}{2 a_1}, -\tfrac{1}{2})\right\rbrace.
    \]
    \begin{table}[b!]
        \centering
        \begin{tabular}{|c|c|}
        \hline
            equation & edge vertices \\
            \hline
            $x_1 = 1/2$ & $\{(\frac{1}{2},-\frac{1}{2}),(\frac{1}{2}, \frac{t-a_1}{2 a_2})\}$\\
            $x_1 = -1/2$ & $\{(-\frac{1}{2},\frac{1}{2}), (-\frac{1}{2} , \frac{a_1-t}{2 a_2})\}$ \\
            $x_2 = 1/2$ & $\{(-\frac{1}{2},\frac{1}{2}),(\frac{t-a_2}{2 a_1}, \frac{1}{2})\}$\\
            $x_2 = -1/2$ & $\{(\frac{1}{2},-\frac{1}{2}), (\frac{a_2-t}{2 a_1}, -\frac{1}{2})\}$\\
            $\langle a,x \rangle = t/2$ & $\{(\frac{1}{2}, \frac{t-a_1}{2 a_2}), (\frac{t-a_2}{2 a_1}, \frac{1}{2})\}$ \\
            $\langle a,x \rangle = -t/2$ & $\{(-\frac{1}{2} , \frac{a_1-t}{2 a_2}), (\frac{a_2-t}{2 a_1}, -\frac{1}{2})\}$ \\
            \hline
        \end{tabular}
        \caption{Facet-defining equalities of $\slabtxt(a_{\mathrm{rep}},t_{\mathrm{rep}},\|\cdot\|_\infty)$ and corresponding vertices, from \Cref{ex:2cube_2moment}.}
        \label{tab:2dslab_edges}
    \end{table}
    Next, we triangulate the slab by first identifying all its edges. \Cref{tab:2dslab_edges} shows the defining equations (left column) of all six edges, and their vertices (right column) as a function of $(a,t) \in C_{1,2}$. To each pair of vertices of a fixed edge, we add the barycenter of the slab, namely the origin $(0,0)$, to obtain a triangulation of the slab into six two-dimensional simplices. 
    Finally, we use \Cref{eqn:determinant,eqn:moment} to compute the second moment of each simplex and add them up to get the final rational function:
    \begin{align*}
        \int_{\slab}\  x_1^2+x_2^2 = &\tfrac{a_1 - a_2 + t}{32a_1}
    - \tfrac{a_1 - a_2 - t}{32a_2} + \tfrac{(a_1^2 + a_1 a_2 + a_2^2 - a_1 t - 2a_2 t + t^2)(a_1 - a_2 + t)}{96a_1^3} \\
    & - \tfrac{(a_1^2 + a_1 a_2 + a_2^2 - 2a_1 t - a_2 t + t^2)(a_1 - a_2 - t)}{96a_2^3} + \tfrac{(a_1^2 - a_1 a_2 + a_2^2 - 2a_1 t + a_2 t + t^2)(a_1 + a_2 - t)t}{96a_1 a_2^3} \\
    & + \tfrac{(a_1^2 - a_1 a_2 + a_2^2 + a_1 t - 2a_2 t + t^2)(a_1 + a_2 - t)t}{96a_1^3 a_2},
    \end{align*}
    for all $(a,t)$ in the topological closure of $C_{1,2}$.
\end{example}

%-----------------------------------------------------------------

\subsection{Proof of \Cref{thm:main}}
\label{sec:thm1pf}

\begin{proof}[Proof] 
First, we bound the number of maximal chambers associated to $\normball$. By \cite[Proposition 3.9]{slices}, at a fixed dimension $d$, the number of maximal chambers is polynomial in the number of vertices $n$ of the norm ball ($O(n^{2d+1} 2^d)$). Next, we show the time needed to compute one slab/slice in one chamber is polynomial in the input size.

Let $\normball$ be the unit ball of the polyhedral norm $\|\cdot\|$ and assume that it has $n$ vertices. \Cref{lem:slab-chambers} combined with \Cref{eqn:moment} proves that on each of the maximal chambers $C_{i,j}$, the moments of the slices and slabs of $\normball$ are parametric rational functions. It remains to show that our method runs in polynomial time in the input size, i.e., in the number of bits needed to describe the polytope ($n$ is bounded by the bit-size of the polytope, see \cite{Schrijver1986}), which is the slice or slab in consideration.

Consider \(\slice\) as a \((d-1)\)-dimensional polytope and \(\slab\) as a \(d\)-dimensional polytope. Since the arguments are identical in both cases, it suffices to establish the claim for one of them; the other follows. Thus, from now on, we focus on the polytope \(S = \slab\). Let $n$ denote the number of vertices of $S$.
Each $d$-dimensional simplex in the barycentric triangulation of $S$ has $d+1$ vertices. Our goal is to bound the number of simplices in the triangulation of $S$ using the Upper Bound Theorem. By definition of the barycentric triangulation, for any simplex $\nabla = \conv\{v_1,v_2, \dots, v_{d+1}\} \subset S$, there exists a chain (by inclusion) of faces $F_i$ of dimension $i$ of $S$, namely,
\[
    F_2 \;\subset\; F_3 \;\subset\; \cdots \;\subset\; F_d = S,
\]
such that $v_1,v_2 \in F_2$, $v_3 \in F_3$, \dots, $v_{d+1} \in F_d$ and the $v_i$ are the barycenters of $F_i$ for all $i>2$. We can therefore upper-bound the number of simplices using the product of the maximal number of faces of each dimension $i = 1,\dots, d$.
By the Upper Bound Theorem \cite{McMullen1970}, any $d$-dimensional polytope with $n$ vertices has at most as many $i$-dimensional faces as the cyclic polytope $\Delta(n,d)$. 
For $1 \le i \le d-1$, the face numbers $f_i$ of the cyclic polytope satisfy
\[
    f_i(\Delta(n,d)) = \frac{n - (d - 2\lfloor\tfrac{d}{2}\rfloor)(n-i-2)}{n-i-1}
    \sum_{j = 0}^n
    \binom{n - 1 - j}{i + 1 - j}
    \binom{n - i - 1}{2j - i - 1 + d - 2\lfloor\tfrac{d}{2}\rfloor},
\]
see \cite[Chapter 9.6]{Gruenbaum03:Polytopes}. For fixed $d$ and $i$, this is a polynomial in $n$.
Consequently, the number of simplices produced from the barycentric triangulation procedure is bounded from above by the product of the number of faces of the cyclic polytope:
\[
    \prod_{i=1}^{d-1} f_i(\Delta(n,d)).
\]
Note that, fixing $d$, this bound is polynomial in $n$, and thus it is bounded by the input size of the polyhedral norm ball $\normball$.

Since the moment of an individual simplex can be computed in polynomial time in its input bit-size \cite{intsimplex}, this bound implies that the total running time for computing any moment of one $\slice$ or $\slab$ is polynomial in the input size. Hence, moment computation for all slices and slabs of a given polyhedral norm is polynomial-time in the bit-size of the input.
\end{proof}

%-----------------------------------------------------------------

\subsection{Algebraic analysis to recover critical points of moments}\label{sec:algebraic}

Given the structure of the maximal open slicing chambers, the output of our algorithm for polyhedral norm balls consists of a list of rational functions, each associated to a maximal chamber $C_{i,j}$, $i=1,\ldots,m$ and $j=1,\ldots,n-1$. We are interested in computing the maxima and minima (or, more generally, the critical points) of the piecewise rational function $f$, which is rational in each $C_{i,j}$. There are two cases for a critical point $(a,t)$ of $f$:
\begin{enumerate}
\item $(a,t) \in C_{i,j}$ is a critical point of the rational function defining $f$ in this open chamber;
\item $(a,t) \in \partial (\overline{C_{i,j}})$.
\end{enumerate}
For the second type of critical points, we need to analyze the value of each rational function along the boundary of its chamber. However, for the first type of critical points, we can also perform a qualitative algebraic analysis. This is the focus of the current section. For more background on computation with algebraic ideals and varieties, we refer to \cite{CLO15:IdealsVarietiesAlgo}.

We are interested in the critical points at fixed values of $t$, and therefore we only consider derivatives with respect to the variable $a$. 
For a rational function $p$ on the sphere, the critical points are the zeros of its spherical gradient, namely
\[
\nabla_{S^{d-1}} p(a) = \nabla_{\mathbb{R}^d} p(a) - \langle \nabla_{\mathbb{R}^d} p(a), a \rangle \, a,
\]
where $\nabla_{\mathbb{R}^d}$ denotes the standard gradient in $\mathbb{R}^d$. This expression is a vector of rational functions. However, since we are interested in the zeros of the gradient, we can study, without loss of generality, the common zeros of the numerators of the gradient entries. Indeed, by construction, the denominators do not vanish in the interior of a chamber, so the zeros of the gradient coincide with the zeros of these numerators.

This information can be packaged in an ideal of polynomials. In particular, since we are only interested in the underlying set of (real) critical points, and not in multiplicities or scheme-theoretic structure, it is enough to pass to the radical of the ideal, which by definition collects all polynomials vanishing at the prescribed points. Therefore, for each maximal chamber $C_{i,j}$, the radical ideal in the variables $(a,t)\in S^{d-1}\times\R$ necessary for the critical points is the radical of the ideal generated by the numerators of the expressions appearing in $\nabla_{S^{d-1}} f_{\vert_{C_{i,j}}}$.

The dimension and degree of $\mathcal{I}_{i,j}$ provide information about the critical points: in the case of dimension $0$, the degree of the ideal equals the number of critical points over the complex numbers, and therefore gives an upper bound on the number of real critical points. For higher-dimensional ideals, one can instead study the structure of the corresponding solution sets.
Finally, we can \emph{decompose} $\mathcal{I}_{i,j}$ in order to divide the critical points into families sharing similar algebraic properties, for instance points satisfying the same minimal polynomial relations. 

This analysis can be performed over $\Q$ in several computer algebra software packages, primarily \texttt{Macaulay2} \cite{M2} or \texttt{Oscar.jl} \cite{OSCAR,OSCAR-book}.
In the remainder of this subsection, we present an exhaustive example in which we exploit these techniques. This computation is feasible in dimensions $2$, $3$, and partially in dimension $4$, see \Cref{rmk:nid}.

\begin{figure}[h]
    \centering
    \includegraphics[width=0.35\linewidth]{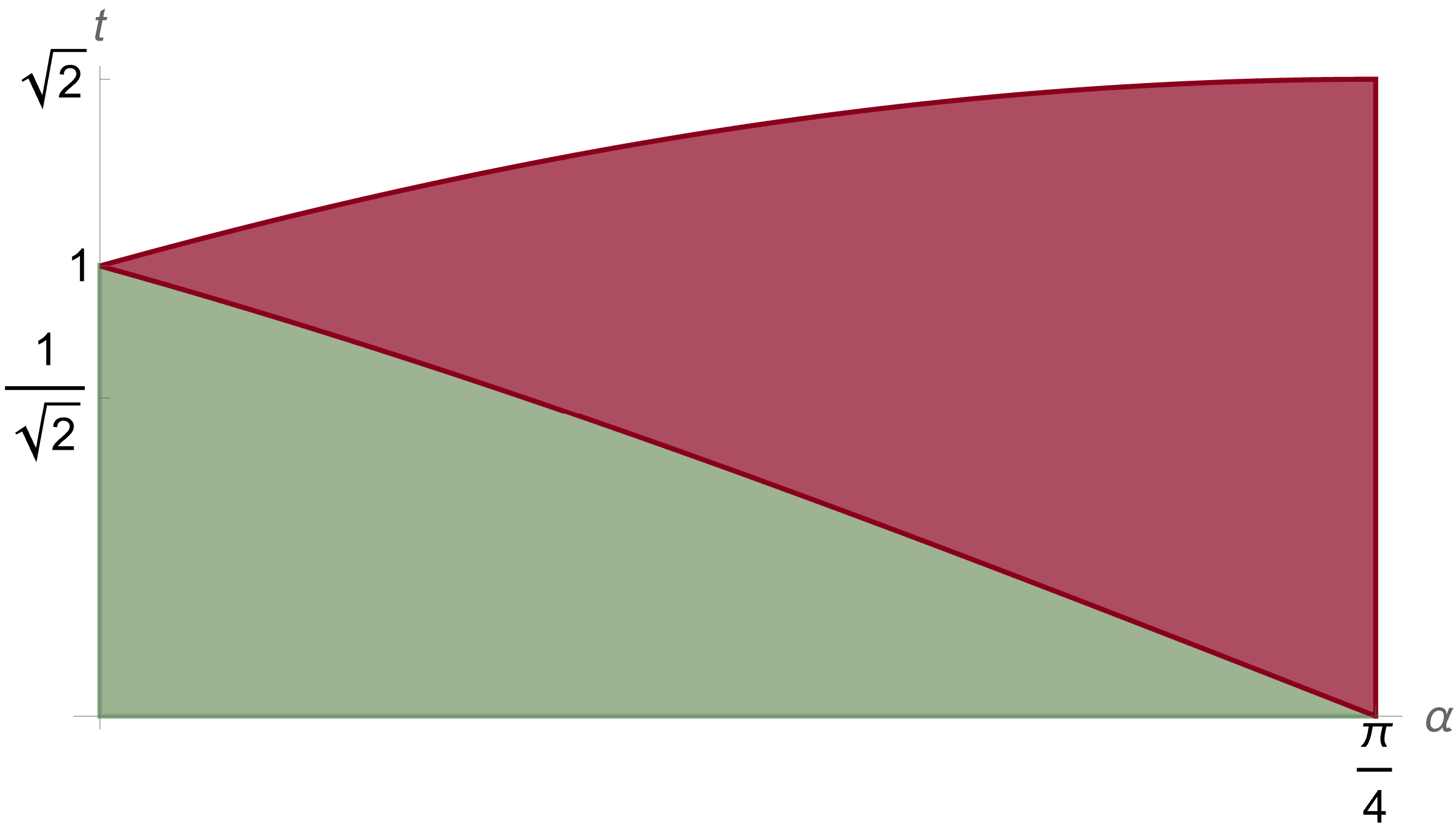}
    \qquad
    \includegraphics[width=0.45\linewidth]{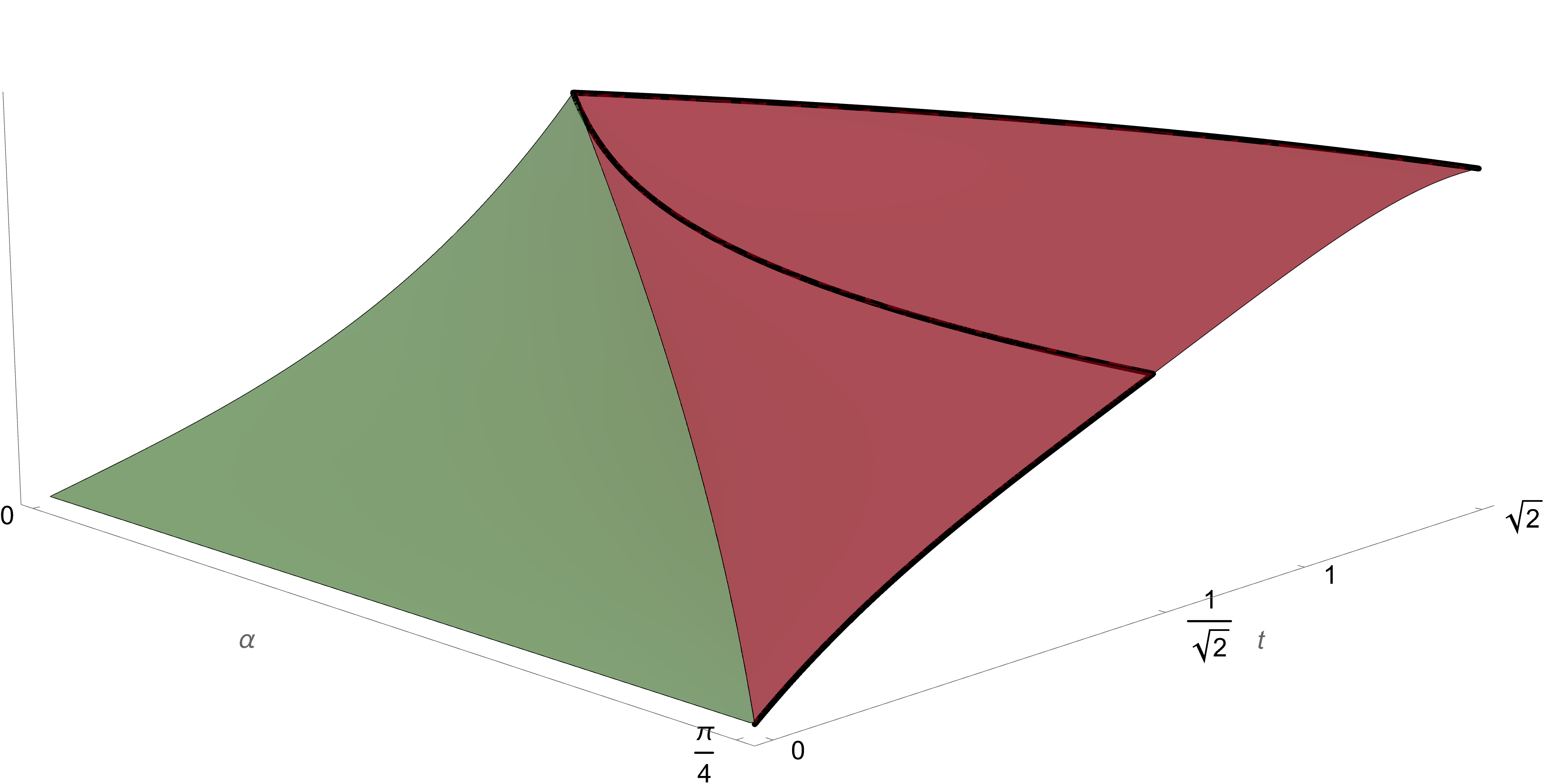}
    \caption{Left: The two fundamental open slicing chambers of the two-dimensional cube, displayed in coordinates $(\alpha,t)$, where $(a_1,a_2) = (\cos \alpha, \sin \alpha)$. Right: The piecewise rational function $f = \{f_{1}^{(2,2)},f_{2}^{(2,2)}\}$ of the second moment from \Cref{ex:algebraic_2cube_2moment}, together with its maxima (black curve) for each value of $t$.}
    \label{fig:moment2D_regions}
\end{figure} 
\begin{example}\label{ex:algebraic_2cube_2moment}
We continue \Cref{ex:2cube_2moment} for the square and the second moments of its slabs. There are two maximal open slicing chambers \eqref{eq:chambers_square}, denoted by $C_{1,1}$ (\Cref{fig:moment2D_regions}, left, green) and $C_{1,2}$ (\Cref{fig:moment2D_regions}, left, red). Hence, there are two relevant rational functions which, for $a\in S^{1}$, read
\[
    g_{1}^{(2,2)} = \tfrac{t^3 + t}{12 a_1^3}, \qquad
    g_{2}^{(2,2)} = \tfrac{
    - t^4
    + 4(a_1^3+a_2^3)t^3
    + (12a_1^2a_2^2-6)t^2
    + 4\bigl(a_1^5+a_2^5+a_1^2a_2^2(a_1+a_2)\bigr)t
    + (8a_1^3a_2^3-1)
    }{96 a_1^3a_2^3},
\]
respectively, where we use the notation of \Cref{appendix:2Dslabmoment}.
The associated radical ideals are
\begin{align*}
    \mathcal{I}_{1,1} &= \bigl(\ 
    a_1^2+a_2^2-1,\ a_2(t^3+t) \ 
    \bigr), \\
    \mathcal{I}_{1,2} &= \bigl( \ 
    a_1^2+a_2^2-1,\
    (2a_2^2-1)(t^4+1)
    +4\bigl(a_1(a_2^2-1)^2-a_2^5 \bigr)(t^3+t)
    +2(2a_2^2-1)(2a_2^4-2a_2^2+3)t^2
     \ \bigr).
\end{align*}
Both ideals have codimension $2$ in the three-dimensional space $\mathbb{C}^2\times\mathbb{C}$, and we interpret them as one-parameter families of critical points, where $t$ plays the role of the parameter.

The ideal $\mathcal{I}_{1,1}$ has degree $8$ and decomposes into four prime ideals:
\begin{gather*}
    \bigl(a_2,a_1-1\bigr), \qquad \bigl(a_2,a_1+1\bigr), \\
    \bigl(a_1^2+a_2^2-1,\, t\bigr), \qquad \bigl(a_1^2+a_2^2-1,\, t^2+1\bigr).
\end{gather*}
The first two ideals describe two families of critical points parametrized by $(\pm 1,0,t)$. The third ideal corresponds to the sphere $S^{1}\times \{0\}$, and the fourth one to the spheres $S^{1}\times \{\pm i\}$, where $i=\sqrt{-1}$. Notice that none of these points belong to the open slicing chamber $C_{1,1}$. Therefore, the interior of this chamber does not contribute any critical point.

Repeating the same procedure for the ideal $\mathcal{I}_{1,2}$, we find that its degree is $16$ and that it decomposes as
\begin{gather*}
    \bigl(a_1^2+a_2^2-1,\, a_1-a_2\bigr), \qquad
    \bigl(a_1^2+a_2^2-1,\, t-a_1-a_2\bigr), \\
    \bigl(a_1^2+a_2^2-1,\, t(a_1+a_2)-1\bigr), \qquad
    \bigl(a_1^2+a_2^2-1,\, t^2 -2t (a_2^3-a_1a_2^2 +a_1) + 1 \bigr).
\end{gather*}
These ideals describe, respectively, families of critical points of the form $(\pm(\tfrac{1}{\sqrt{2}}, \tfrac{1}{\sqrt{2}}), t)$,
$(a_1, a_2, a_1+a_2)$,
$(a_1, a_2, \tfrac{1}{a_1+a_2})$,
and $(a_1,a_2,\phi_{\pm}(a))$, where
\[
\phi_{\pm}(a)
= a_2^3 - a_1 a_2^2 + a_1
\pm \sqrt{\left(-a_1 a_2^2 + a_1 + a_2^3\right)^2 - 1}.
\]
The first two families lie on the boundary of $C_{1,2}$ and therefore do not contribute critical points in the interior of the chamber. The third family, on the other hand, lies in $C_{1,2}$ for all $a\in R_1$ (as in \Cref{def:maximalChamber}), hence for $t\in (\tfrac{1}{\sqrt{2}},1)$. This family is shown in \Cref{fig:moment2D_fixed_t}, center, for $t=0.8$. The entire curve of critical points in the interior of the slicing chamber corresponds to the black curve in \Cref{fig:moment2D_regions}, right, lying inside the red surface. Finally, the last family of points is never real in $C_{1,2}$.
\begin{figure}[h]
    \centering
    \includegraphics[width=0.32\linewidth]{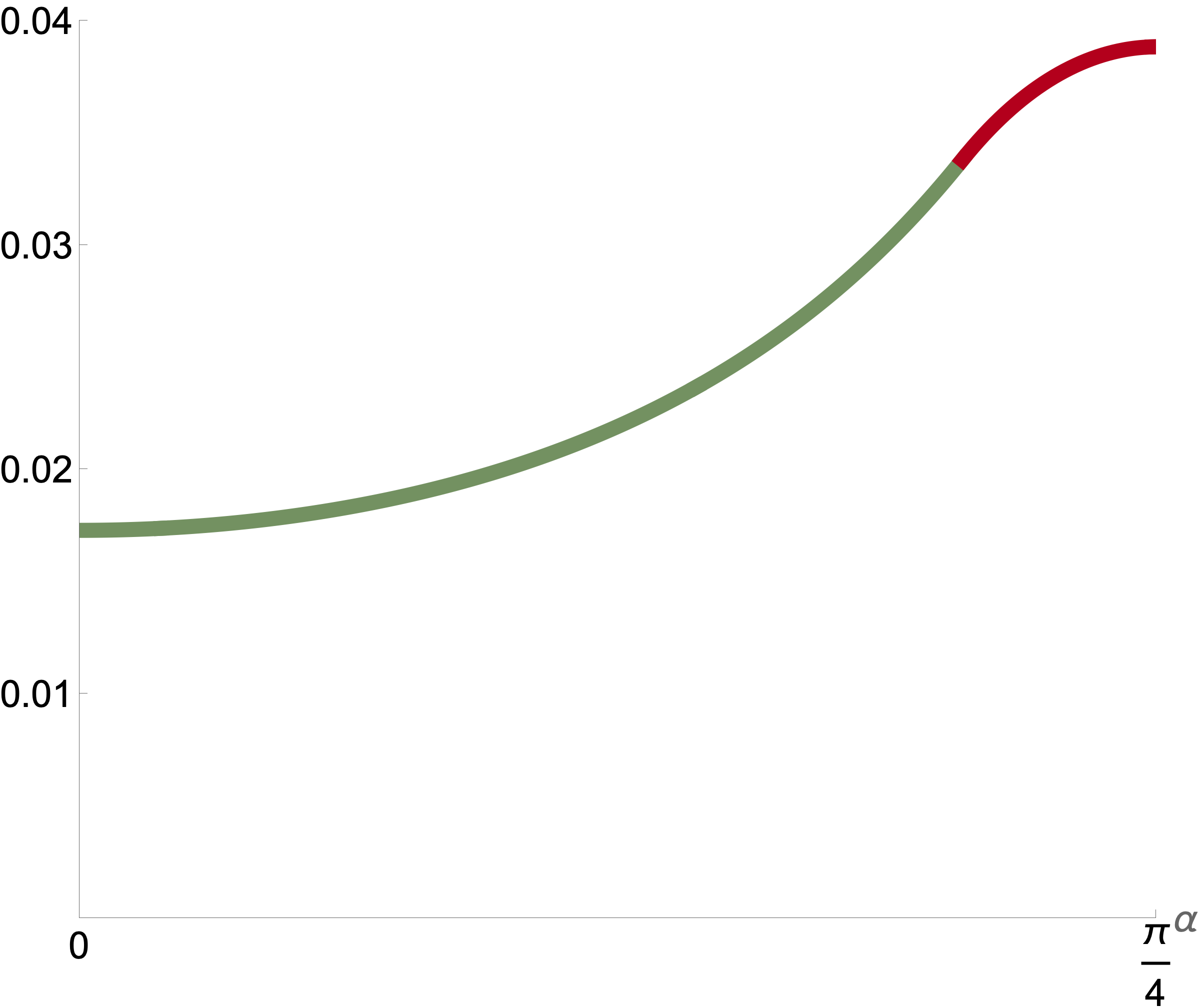}
    \includegraphics[width=0.32\linewidth]{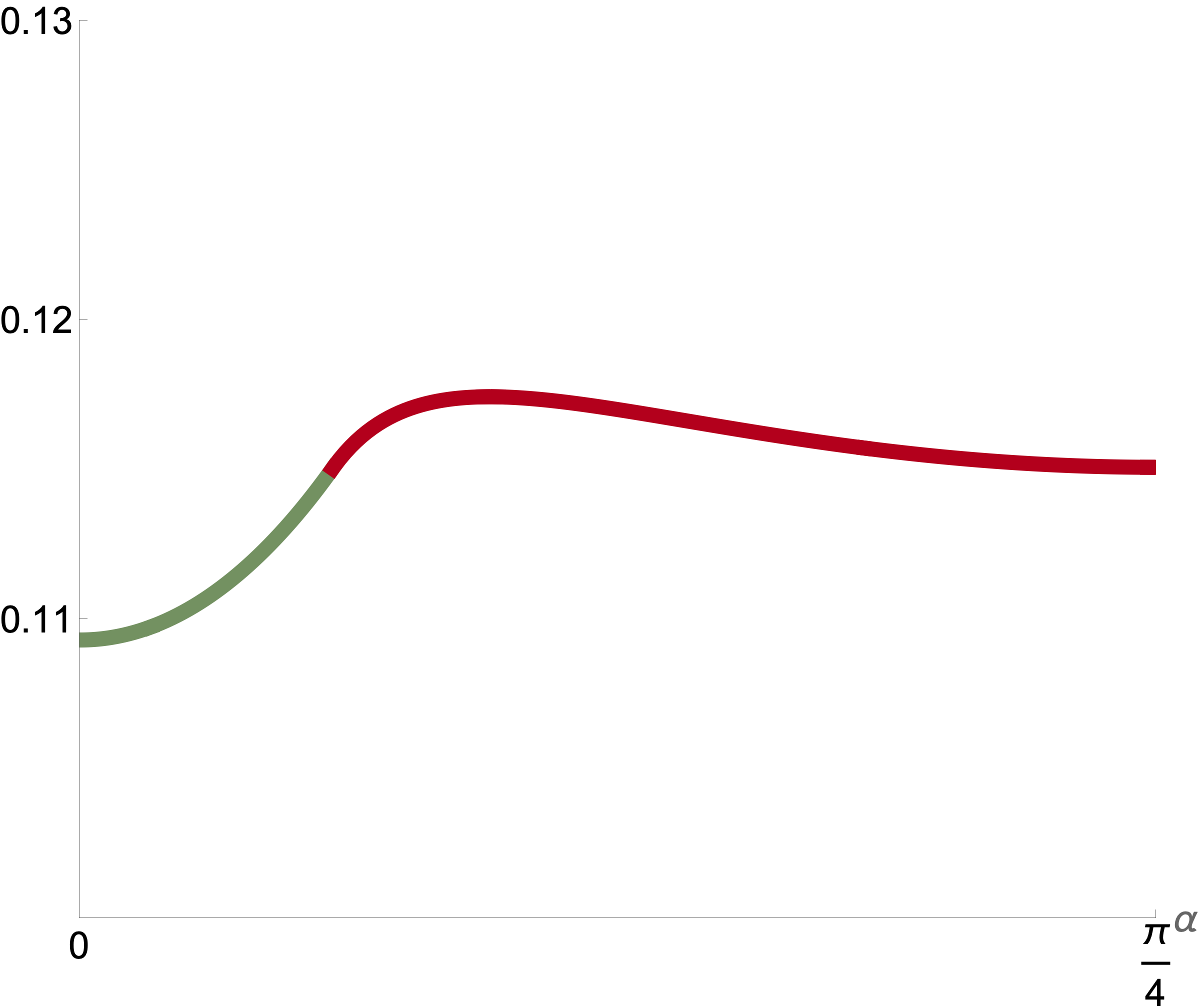}
    \includegraphics[width=0.32\linewidth]{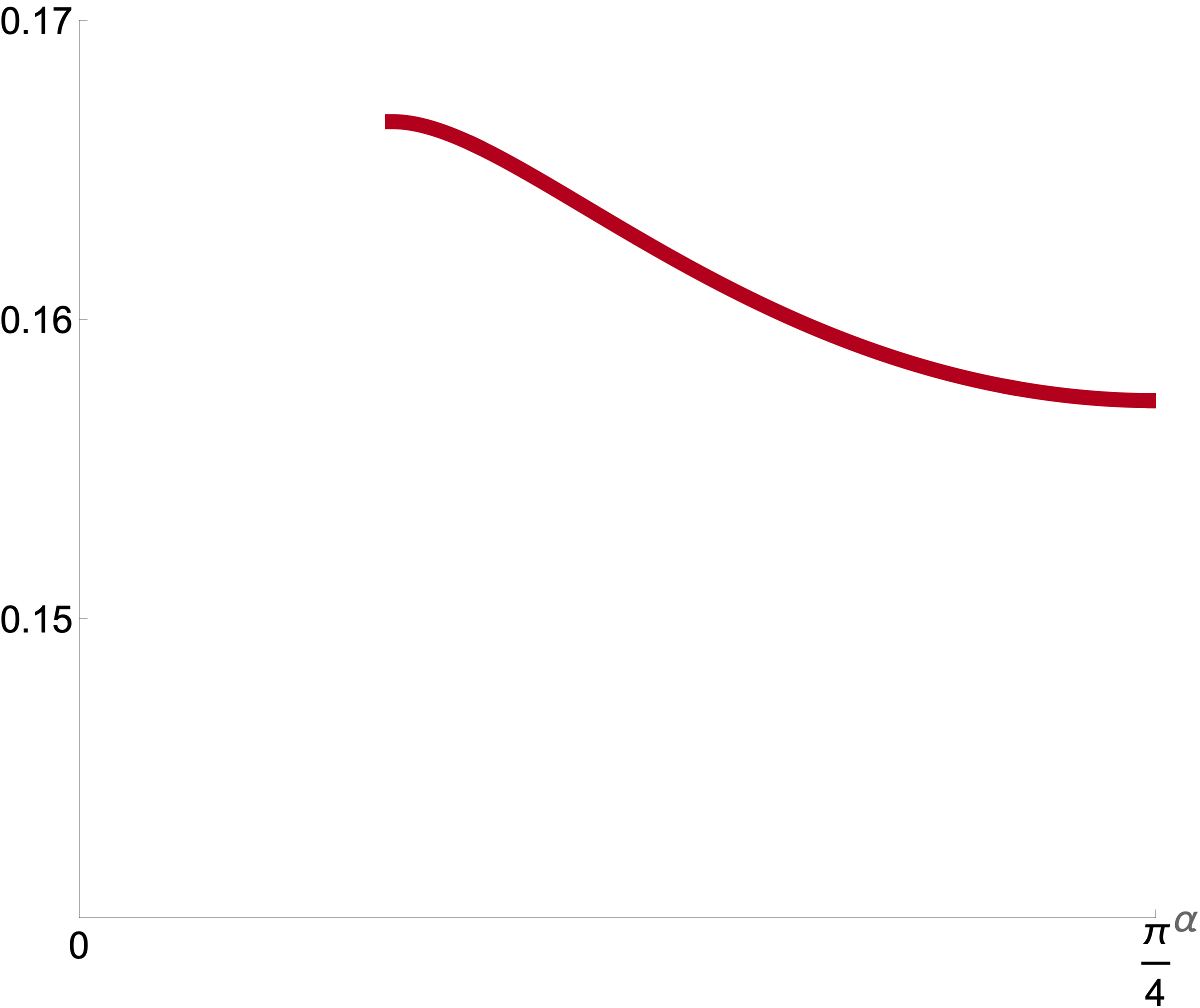}        
    \caption{The function $f(\alpha,t_{\mathrm{rep}}) = \{f_{1}^{(2,2)}(\alpha,t_{\mathrm{rep}}),f_{2}^{(2,2)}(\alpha,t_{\mathrm{rep}})\}$ of the second moment of the two-dimensional cube from \Cref{ex:algebraic_2cube_2moment}, where $(a_1,a_2) = (\cos \alpha, \sin \alpha)$ and $t_{\mathrm{rep}} = 0.2, 0.8, 1.2$ from left to right. The green part of each curve is its restriction to $C_{1,1}$, and the red part is its restriction to $C_{1,2}$. Only in the central figure is there a maximum in the interior of a slicing chamber.}
    \label{fig:moment2D_fixed_t}
\end{figure}
\end{example}

One can repeat analogous computations for all the rational functions produced by our algorithm. \Cref{tab:data_dim3} summarizes these computations for the rational functions describing the volume of slices of the three-dimensional cube $\hypercube^3$, which are listed explicitly in \Cref{appendix:3dslice}. The table was obtained using \texttt{Macaulay2}. 
For each rational function, we compute the dimension and degree of the associated radical ideal, as well as its decomposition into irreducible components. This decomposition reveals how many irreducible components the ideal has, corresponding to distinct families of critical points. For each such component, we also record its dimension and degree. Note that when components of different dimensions are present, the reported degree of the ideal refers to the degree of the top-dimensional component.
Analogous computations can be carried out for the remaining rational functions listed in the Appendix.
\begin{table}[h]
    \centering
    \begin{tabular}{|c|c|c|c|c|}
        \hline
        \textbf{Volume formulas} & \textbf{$\dim \mathcal{I}$} & \textbf{$\deg \mathcal{I}$} & \textbf{$\dim$ of irr. components} & \textbf{$\deg$ of irr. components} \\
        \hline
        $f_1$ & 1 & 2 & (1,1) & (1,1) \\
        \hline
        $f_2$ & 1 & 8 & (1,1,1) & (2,2,4) \\
        \hline
        $f_3$ & 2 & 2 & (2,1,1,1,1) & (2,2,4,4,4) \\
        \hline
        $f_4$ & 1 & 24 & (1,1,1,1,1) & (2,4,6,6,6) \\
        \hline
        $f_5$ & 1 & 20 & (1,1,1,1) & (2,2,8,8) \\
        \hline
    \end{tabular}
    \caption{$\mathcal{I}$ is the ideal of the critical points of the corresponding volume formula for the slices of $\hypercube^3$. The irreducible components are over $\Q$. The functions $f_i$ are given in \Cref{appendix:3dslice}.}
    \label{tab:data_dim3}
\end{table}

\begin{remark}\label{rmk:nid}
    For larger ideals (in terms of the number of variables or the degrees of the generators), some \texttt{Macaulay2} computations may fail to terminate. In such cases, one can instead exploit the \emph{numerical irreducible decomposition} implemented in \texttt{HomotopyContinuation.jl}~\cite{Breiding2018} to obtain information about the irreducible components of the ideals over $\mathbb{R}$. This decomposition may differ from the one over $\mathbb{Q}$. 
    Indeed, the ideal generated by the univariate polynomial $x^2-2$ is irreducible over the field of rational numbers, but it decomposes over the reals into the two prime ideals $(x-\sqrt{2})$ and $(x+\sqrt{2})$. For this reason, if the decomposition were performed over $\mathbb{R}$, the last two columns corresponding to $f_2$, $f_3$, and $f_4$ in \Cref{tab:data_dim3} would change. In particular, their irreducible one-dimensional components over $\mathbb{R}$ would be more numerous and have degrees $(1,1,2,4)$, $(1,1,4,4,4)$, and $(1,1,2,2,6,6,6)$, respectively.
\end{remark}
We leave as a possible direction for future research the problem of determining which of these components actually contribute to real critical points of the volume or moment functions in each chamber.

%===========================================================

\section{Slices and Slabs of Low-Dimensional Cubes}
\label{sec:volumes}

In this section, we present computational experiments investigating rational-function expressions for the volumes and moments of slices and slabs of the infinity norm ball, the $d$-dimensional cube $\hypercube^d$ for $d=2,3,4$ and moments $M=0,1,2,3,4$.
Moreover, in the case of slices of the $2$-dimensional cube, we give a complete theoretical description of all moments of slices, as well as their critical points, for all $M \in \mathbb Z_{\geq 0}$. All computations were carried out symbolically in \texttt{SageMath} \cite{sagemath}, following the method introduced in \Cref{sec:methodology}.
All rational functions are printed in the \hyperref[appendix]{Appendix}.
The code and rational functions in digital form are publicly available at \url{https://github.com/RainCamel/slab_of_the_poly_norms}.

Fixing a region $R_{i}\subset S^{d-1}$ and a representative direction \(a_{\mathrm{rep}} \in R_{i}\), one obtains \(2^{d-1}\) distinct $t$-ranges $I_{i,j}(a_{\mathrm{rep}})$. Let \(t_{\mathrm{rep}}\) be the midpoint of such an interval, as described in \Cref{sec:equiClass_vertexOrdering}. Since slices and slabs share the same maximal chamber $C_{i,j}$, each such case gives rise to one corresponding volume formula and a family of moment formulas, indexed by $M$. Due to the symmetry of the infinity norm ball, it suffices to compute all chambers containing points $a\in S^{d-1}$ satisfying $a_1 \geq a_2 \geq \dots \geq a_d \geq 0$. But even after this restriction, we will see 
that some rational function expressions are defined on multiple chambers. For example, the expressions $f_5$ and $g_5$ (see the next paragraph for their definitions) in \Cref{table:3d} appear in some interval of all six regions $R_i$. This happens because in all those chambers $C_{i,j}$, the hyperplane $H(a,t)$ separates one vertex of the cube from all the others, resulting in an equivalent family of slices or slabs. Other repetitions arise in a similar manner. 

\begin{remark}
We note that in highly symmetric situations, such as the unit cube, there is repetition of the rational function expressions on different chambers, up to permutation or reflection of variables. Ultimately it is the combinatorial type that controls the volume/moment formulas, but enumerating through all maximal chambers gives a systematic way to find all possible combinatorial types of slices or slabs of a given polyhedral norm ball.
\end{remark}

\noindent \textbf{Notation.} For notation, we use $f_k^{(M,d)}$ to denote the $k$-th rational expression of the $M$-th moment of the slice of the $d$-dimensional cube $\hypercube^d$, and use $g_k^{(M,d)}$ to denote the $k$-th rational expression of the $M$-th moment of the slab of the $d$-dimensional cube $\hypercube^d$. 

\noindent \textbf{Warning.} In this section, we sometimes omit the superscripts $(M,d)$ from the formulas listed in the tables when the context is clear, i.e. we use $f_k,g_k$ instead of  $f_k^{(M,2)},g_k^{(M,2)}$. When ambiguity arises, please refer to the table captions or closest explanation of dimension and moment.

%--------------------------------------------------------------

\subsection{$B_\infty^2$}\label{sec:2d}

We give explicit formulas for volumes, as well as higher moments of slices and slabs of $B_\infty^2$ for moment $M=1,2,3,4$. They appear in \Cref{appendix:2dvol,appendix:2Dmoment} and were computed with the method explained in \Cref{sec:methodology}. Restricted to $a_1 \geq a_2 \geq 0$, \Cref{table:2d_slice_formulas} follows K\"onig and Koldobsky \cite{KK11}, and identifies the largest range of $t$ as a function of $a$ in which a certain rational function is valid (see \Cref{table:2d_slice_formulas}). This table further indicates for each case which of the volume formulas given in \Cref{appendix:2dvol} are valid. The given volume formulas in dimension 2 agree with those in \cite{KK11}. We present them here for completeness. We extend their results by providing similar formulas for higher moments ($M=1,2,3,4$) in \Cref{appendix:2Dmoment}. By construction, the formulas for higher moments are defined on the same regions as the volumes. \Cref{table:2d_slice_formulas} is therefore also a valid table for the moment formulas in \Cref{appendix:2Dmoment}.

\begin{table}[H]
\centering
\renewcommand{\arraystretch}{1.15}
\begin{tabular}{|l|c|}
\hline
\textbf{t-range} & \textbf{Formulas} \\
\hline

$0 \le t \le a_1-a_2$
& $f_1, g_1$ \\ \hline

$a_1 - a_2 \le t \le a_1+a_2$
& $f_2, g_2$ \\ \hline
\end{tabular}

\caption{In the two-dimensional setting, where the polyhedral norm ball is the unit square centered at the origin, this table records the piecewise rational formulas for the slice/slab volume/moment as functions of $(a,t)$. For each interval of $t$ shown in the left column, the corresponding formulas in the right column are valid for that interval. Note that here we use $f_i^{(M,2)}$ and $g_i^{(M,2)}$ to denote, respectively, the $i$-th rational expressions for the $M$-th moment of the slice and slab of the square. For the actual formulas, please check \Cref{appendix:2dvol,appendix:2Dmoment}.}
\label{table:2d_slice_formulas}
\end{table}

Focusing on slices, we now establish closed-form formulas for the $M$-th moments of slices of the two-dimensional cube, valid for all $M \in \mathbb{Z}_{\ge 0}$. This includes the volume as a special case when $M = 0$. Furthermore, we describe their critical points explicitly. 
These closed formulas are based on the fact that no triangulation is necessary
for line segments, and the integral reduces to evaluating elementary one-dimensional polynomial expressions along explicitly parametrized vertices. As a consequence, the derivation depends only on the algebraic form of the vertices of the line segment and not on any combinatorial decomposition, allowing the same method to extend to the family of slices of any polyhedral norm ball in dimension $2$. 

\begin{proposition}\label{prop:poly_moment_2D}
    Consider the fundamental region $R_1 = \{(a_1,a_2)\in S^1 \mid a_1 > a_2 > 0\}$. The $M$-th moment of the slices of the square $B_\infty^2$, for $M\in \mathbb Z_{\geq 0}$, is the piecewise rational function
   \begin{align*}
    f_1^{(M,2)} &= \frac{1}{(M+1) 2^{M}} \cdot 
    \begin{cases}
        \frac{1}{a_1} + \frac{(t+a_2)^{M+1} - (t-a_2)^{M+1}}{2 a_1^{M+1} a_2} &  M \text{ even},\\
        \frac{(t+a_2)^{M+1} - (t-a_2)^{M+1}}{2 a_1^{M+1} a_2} & M \text{ odd},
    \end{cases}
     && a \in R_1, \; 0\leq t \leq a_1-a_2, \\
    f_2^{(M,2)} &= \frac{1}{(M+1) 2^{M+1}} \left( \frac{1}{a_1} + \frac{1}{a_2} - \frac{(t-a_1)^{M+1}}{a_1 a_2^{M+1}} - \frac{(t-a_2)^{M+1}}{a_1^{M+1} a_2} \right),
    && a \in R_1, \; a_1-a_2 \leq t \leq a_1+a_2.
    \end{align*}
\end{proposition}

\begin{proof}
We already described the maximal slicing chambers of the square $B_\infty^2$ (see \Cref{ex:2cube_2moment} for details) as in \eqref{eq:chambers_square}.
The chamber $C_{1,1}$ corresponds to slices of the square (namely segments) with vertices $v_1 = ( \frac{t+a_2}{2 a_1}, -\frac{1}{2})$, $v_2 = ( \frac{t-a_2}{2 a_1}, \frac{1}{2})$ on the top and bottom edges of $B_\infty^2$, whereas $C_{1,2}$ corresponds to segments with vertices $v_1 = ( \frac{1}{2} , \frac{t-a_1}{2 a_2})$, $v_2 = (\frac{t-a_2}{2 a_1}, \frac{1}{2})$ on the top and right edges of $B_\infty^2$. See 
\Cref{fig:slices_square} left and right, respectively.
\begin{figure}[h]
    \centering
    \input{Images/2D_slices}
    \caption{A representative for each maximal slicing chamber of the square, up to symmetry. Left: $(a,t)\in C_{1,1}$. Right: $(a,t)\in C_{1,2}$.}
    \label{fig:slices_square}
\end{figure}

We can use \eqref{eqn:moment} to compute the $M$-th moment in each chamber. We start with the chamber $C_{1,1}$, coordinate by coordinate:
\begin{align*}
    \int_{[v_1,v_2]} x_1^M \mathrm{d}x &= \frac{1}{(M+1)} \frac{1}{a_1} \sum_{k=0}^M \left( \frac{t+a_2}{2 a_1} \right)^k \left( \frac{t-a_2}{2 a_1} \right)^{M-k} =\\ 
    &= \frac{1}{(M+1) 2^M}\frac{\sum_{k=0}^M \left( t+a_2 \right)^k \left( t-a_2 \right)^{M-k}}{a_1^{M+1}}  \\
    &= \frac{1}{(M+1)2^{M+1}}\frac{\left( t+a_2 \right)^{M+1} - \left( t-a_2 \right)^{M+1}}{a_1^{M+1} a_2}, \\
    \int_{[v_1,v_2]} x_2^M \mathrm{d}x &= \frac{1}{(M+1)} \frac{1}{a_1} \sum_{k=0}^M \left( -\frac{1}{2} \right)^k \left( \frac{1}{2} \right)^{M-k} 
    = \begin{cases}
        \frac{1}{(M+1) 2^M} \frac{1}{a_1} & M \text{ even}, \\
        0 & M \text{ odd}.
    \end{cases}
\end{align*}
Therefore, in this chamber the $M$-th moment coincides with the function $f_1^{(M,2)}$ in the statement.
We do the same computation for $C_{1,2}$ and we obtain
\begin{align*}
    \int_{[v_1,v_2]} x_1^M \mathrm{d}x &= 
    \frac{1}{(M+1)} \frac{a_1+a_2-t}{2 a_1 a_2} \sum_{k=0}^M \left( \frac{t-a_2}{2 a_1} \right)^k \left( \frac{1}{2} \right)^{M-k}  \\
    % &= \frac{a_1+a_2-t}{M! (M+1)! 2^{M+1}} \frac{1}{a_1^{M+1} a_2} \sum_{k=0}^M \left( t-a_2 \right)^k a_1^{M-k}  \\
    &= \frac{1}{(M+1) 2^{M+1}} \frac{a_1^{M+1} - (t-a_2)^{M+1}}{a_1^{M+1} a_2}, \\
    \int_{[v_1,v_2]} x_2^M \mathrm{d}x &= 
    % \frac{1}{M! (M+1)!} \frac{a_1+a_2-t}{2 a_1 a_2} \sum_{k=0}^M \left( \frac{t-a_1}{2 a_2} \right)^k \left( \frac{1}{2} \right)^{M-k}  \\
    % &= \frac{a_1+a_2-t}{M! (M+1)! 2^{M+1}} \frac{1}{a_1 a_2^{M+1}} \sum_{k=0}^M \left( t-a_1 \right)^k a_2^{M-k}  \\
    % &=
    \frac{1}{(M+1) 2^{M+1}} \frac{a_2^{M+1} - (t-a_1)^{M+1}}{a_1 a_2^{M+1}}.
\end{align*}
Therefore, the $M$-th moment in this chamber coincides with $f_2^{(M,2)}$ in the statement.
\end{proof}

With this, we can now provide formulas for the critical points of the moments of the slices of the square.
\begin{proof}[Proof of \Cref{thm:2Dslicemoment}]
    In order to find the critical points inside the chambers, we need to compute the spherical gradients of the rational function in \Cref{prop:poly_moment_2D}, as explained in \Cref{sec:algebraic}. For odd $M$, the spherical gradient of $f_1^{(M,2)}$ is a vector with entries
    \begin{align*}
         &\tfrac{a_{2}\,(a_{1}^{2}-(M+1))\, \bigl((t-a_{2})^{M} + (t+a_{2})^{M}\bigr)-t\,(a_{1}^{2}(M+2)-(M+1))\, \bigl((t-a_{2})^{M} - (t+a_{2})^{M}\bigr)}{(M+1) 2^{M+1} \, a_{1}^{M+2} a_{2}}, \\
        &\tfrac{
          a_{2}\,\bigl(M + a_{2}^{2}\bigr)\,\bigl((t-a_{2})^{M} + (t+a_{2})^{M}\bigr) - t\,\bigl((M+2)\,a_{2}^{2} - 1\bigr)\,\bigl((t-a_{2})^{M} - (t+a_{2})^{M}\bigr)}{(M+1)\,2^{M+1} \, a_{1}^{M+1}\,a_{2}^{2}}.
    \end{align*}
    Denote by 
    \[
    A_+(a_2,t) = (t-a_{2})^{M} + (t+a_{2})^{M}, \qquad A_-(a_2,t) = \frac{(t-a_{2})^{M} - (t+a_{2})^{M}}{a_2}, 
    \]
    as in the statement. For $M\geq 0$, these are both polynomials.
    Using $a_1^2+a_2^2=1$, the zeros of the spherical gradient are either $(1,0)$ or they are the points $a \in S^1$ satisfying
    \[
        \bigl(a_{2}^{2}+M\bigr)\,A_+(a_2,t) - t\,\bigl((M+2)\,a_{2}^{2} - 1\bigr)\,A_-(a_2,t) = 0,
    \]
    since both equations reduce to this. This equation is nontrivial and possibly provides additional solutions.
    
    For even $M$, the spherical gradient of $f_1^{(M,2)}$, using the polynomials $A_{\pm}$ defined above, reads
    \begin{appendixformulas}
    \begin{align*}
    &
    \tfrac{
    -\,2 a_1^M a_{2}
    + 2a_{1}^{M+2} a_{2}
    +
    a_{2}(a_{1}^{2} - (M+1))A_+(a_2,t) 
    -t (a_1^2(M+2) - (M+1))a_2 A_-(a_2,t)
    +2 a_{1}^{M} a_{2}^{3} +
     a_2(a_2^2+M)A_+(a_2,t)
      -t (a_2^2(M+2)-1)a_2 A_-(a_2,t)}{(M+1)\,2^{M+1}\,
    a_{1}^{M+2} a_{2}
    }.
    \end{align*}
    \end{appendixformulas}
 Similar to the odd case, $(1,0)$ is always an admissible solution, and we reduce the system to 
    \[
    2 a_1^M a_2^2
    +(a_2^2+M)A_+(a_2,t)
    -t (a_2^2(M+2)-1)A_-(a_2,t) = 0,
    \]
    for $a\in S^1$.
    For $M=0,2$, the only admissible solution is $(1,0)$, but for $M\geq 4$ the equation is nontrivial and possibly provides additional real solutions depending on the value of $t$.

    Moving on to $f_2^{(M,2)}$, we compute its spherical gradient as above.
    Also in this case, both entries give the same condition on the sphere, namely
    \[
    a_1^M (t - a_1)^M p(a_1,t,M) + 
    - a_2^M (t - a_2)^M p(a_2,t,M) + a_1^M a_2^M (a_1^3 - a_2^3) = 0,
    \]
    where $p(\alpha,t,M) = \alpha^3 -t(M+2)\alpha^2+M\alpha+t$. When $M=0$, this reduces to 
    \[
    a_1^3-a_2^3+2 a_2^2 t-t = 0,
    \]
    which has nontrivial solutions in the chamber $C_{1,2}$ when $t\in [1,\frac{3}{2 \sqrt{2}}]$. For $M\geq 1$ this equation provides additional real critical points.
\end{proof}

%---------------------------------------------------------------------------

\subsection{$B_{\infty}^3$}
\label{sec:3d_tables}

In this subsection, we provide two different descriptions of the piecewise rational functions for volumes and moments of slices and slabs of $B_\infty^3$. The first one follows K\"onig and Koldobsky \cite{KK11}, and identifies the largest range of $t$ as a function of $a$ in which a certain rational function is valid (see \Cref{table:3d_slice_formulas}).
The second description is coherent with the output of our algorithm, described in \Cref{sec:methodology}.  Indeed, for each region $R_i$ on the sphere we find the associated intervals $I_{i,j}(a)$ and compute the rational function for each chamber $C_{i,j}$ (see \Cref{table:3d}).
As in the two-dimensional case, by construction, the formulas for higher moments are defined on the same regions as the volumes. Dropping the superscripts and using $f_k,g_k$ to denote  $f_k^{(M,3)},g_k^{(M,3)}$, \Cref{table:3d_slice_formulas} is therefore also a valid table for the moment formulas for $M=1,2,3,4$.
The explicit formulas are given in \Cref{appendix:3dvol,appendix:3Dmoment}.

\begin{table}[h!]
\renewcommand{\arraystretch}{1.08}
\setlength{\tabcolsep}{3pt}
\centering
\begin{tabular}{|c|c|c|}
\hline
$\bar a \ \text{s.t.}\ \frac{\bar a}{\|\bar a\|}\in R_i$ & $I_{i,j}(\bar a)$ & Formulas \\
\hline
\multirow{4}{*}{$(4,2,1)$} & $[0, 1]$ & $f_{1}$ \\
\cline{2-3}
 & $[1, 3]$ & $f_{3}$ \\
\cline{2-3}
 & $[3, 5]$ & $f_{4}$ \\
\cline{2-3}
 & $[5, 7]$ & $f_{5}$ \\
\hline
\multirow{4}{*}{$(4,3,2)$} & $[0, 1]$ & $f_{1}$ \\
\cline{2-3}
 & $[1, 3]$ & $f_{2}$ \\
\cline{2-3}
 & $[3, 5]$ & $f_{4}$ \\
\cline{2-3}
 & $[5, 9]$ & $f_{5}$ \\
\hline
\end{tabular}
\caption{
In the three-dimensional setting, where the polyhedral norm ball is the unit cube
centered at the origin, this table records the piecewise rational formulas of slices ($f_i$) and slabs ($g_i$) obtained from each maximal open chamber $(a,t) \in R_i \times I_{i,j}(a)$.
Note that here we omit the superscripts. Instead of using $f_k^{(M,3)}$ and $g_k^{(M,3)}$, we use $f_k$ and $g_k$ to denote, respectively, the $k$-th rational function of the $M$-th moment of the slice and slab of the 3-cube.
}\label{table:3d}
\end{table}

\begin{table}[ht]
\centering
\renewcommand{\arraystretch}{1.15}
\begin{tabular}{|l|c|}
\hline
\textbf{$t$-range} & \textbf{Formulas} \\
\hline

$0 \le t \le a_1-(a_2+a_3)$
& $f_1, g_1$ \\ \hline

$0 \le t \le (a_2+a_3)-a_1$
& $f_2, g_2$ \\ \hline

$\lvert a_1-(a_2+a_3)\rvert \le t \le a_1-a_2+a_3$
& $f_3, g_3$ \\ \hline

$a_1-a_2+a_3 \le t \le a_1+a_2-a_3$
& $f_4, g_4$ \\ \hline

$a_1+a_2-a_3 \le t \le a_1+a_2+a_3$
& $f_5, g_5$ \\ \hline
\end{tabular}
\caption{In the three-dimensional setting, where the polyhedral norm ball is the unit cube centered at the origin, this table records the piecewise rational formulas for the slice/slab volume and moment as functions of $(a,t)$. For each interval of $t$ shown in the left column, the corresponding formulas in the right column are valid for that interval. Note that we omit the superscripts here and use $f_i$ and $g_i$ to denote, respectively, the $i$-th rational expressions for the $M$-th moment of the slice and slab of the cube.}
\label{table:3d_slice_formulas}
\end{table}

\begin{remark}
    Recall that the formulas $g_k^{(M,d)}$ for slabs are valid for all $a \in \R^d$, while the formulas $f_k^{(M,d)}$ for slices are only valid for $a \in S^{d-1}$.
    For the case of slices, \Cref{table:3d,table:4d_chamber} thus require a more careful interpretation: vectors $\bar{a}$ from these tables are not contained in $S^{d-1}$, but, to be more specific, their normalization $\frac{\bar{a}}{\|\bar{a}\|} \in R_i \subset S^{d-1}$. The interval $I_{i,j}(\bar{a})$ in these tables is given for $\bar{a}$ and not for its normalization. However, we have the equality of hyperplanes $H(\bar{a},t) = H\left(\tfrac{\bar{a}}{\|\bar{a}\|},\tfrac{t}{\|\bar{a}\|}\right)$. Thus, in order to obtain the volume of the slice  $\hypercube^d \cap H(\bar{a},t)$ for $t \in I_{i,j}(\bar{a})$, one needs to evaluate the corresponding rational function $f_k^{(M,d)}$ as
    \[
    \vol_{d-1}(B_\infty^d \cap H(\bar{a},t)) = f_k^{(M,d)}\left(\frac{\bar{a}}{\|\bar{a}\|}, \frac{t}{\|\bar{a}\|}\right).
    \]
\end{remark}

\subsection{$B_{\infty}^4$}
\label{sec:4d_tables}

We present here the piecewise rational functions of volumes and higher moments of slices and slabs of $B_\infty^4$ (\Cref{table:4d_kkchamber,table:4d_chamber}). The explicit rational functions for the volume can be found in \Cref{appendix:4dvol}. Due to their length, the explicit functions for higher moments of order $M=1,2,3,4$ can be found in our repository\footnote{\url{https://github.com/RainCamel/slab_of_the_poly_norms}}.

\begin{proof}[Proof of \Cref{thm:4dvol}]
    \Cref{lemma:summaryofBDLMpaper,lem:slab-chambers} imply that the volumes of slices and slabs are piecewise rational functions, whose pieces are supported on maximal slicing chambers. 
    The vertex-ordering method from \Cref{sec:methodology} identifies \textit{fourteen} regions $C_{i,j}$ for which $a_1\geq\dots\geq a_4\geq0$, up to signed permutation of $a_1,a_2,a_3,a_4$. Because the moment computations use the same underlying polytopes coming from the same maximal chamber, the number of moment formulas should match the number of volume formulas. In fact, the computational results in \Cref{appendix:4dvol} show that there are also \textit{fourteen} distinct rational-function expressions for the volumes of the corresponding slices and slabs. By the argument in \Cref{lem:slab-chambers}, these volume formulas are also valid on the boundary of the closure of the open maximal slicing chambers, and are hence valid on all maximal closed chambers of the associated sweep arrangement.
\end{proof}

From the data obtained, one can observe that $f_1^{(0,4)}$, the simplest volume formula for slices of $B_\infty^4$, is the same as $f_1^{(0,2)}$ and $f_1^{(0,3)}$, the simplest volume formulas of $B_\infty^2$ and $B_\infty^3$, respectively. The same phenomenon can be observed for $g_1^{(0,2)},g_1^{(0,3)}$, and $g_1^{(0,4)}$, the simplest volume formulas for slabs. We show that this holds in all dimensions, as these functions parametrize the volume of parallelepipeds of dimensions $(d-1)$ and $d$, respectively.
\begin{lemma}
    Let $C_{1,1} = \{(a,t) \in S^{d-1}\times\R \mid a_1 \geq \dots \geq a_d\geq 0, \, t \leq a_1 - \sum_{i=2}^d a_i\}$. The volume formulas for $\slicetxt(a,t,\|\cdot\|_\infty)$ and $\slabtxt(a,t,\|\cdot\|_\infty)$ of $\hypercube^d$ for $(a,t)\in C_{1,1}$ satisfy respectively
    \[
    f_1^{(0,d)} = \frac{1}{a_1}, \qquad g_1^{(0,d)} = \frac{t}{a_1}.
    \]
\end{lemma} 
\begin{proof}
Let $(a,t) \in C_{1,1}$. Each of the hyperplanes $H(a,t)$ and $H(a,-t)$ intersects all and only the edges of the hypercube parallel to the $a_1$-axis, and there is no vertex of the cube in the slab. Hence, $\sliceinf$ is a $(d-1)$-dimensional parallelepiped and $\slabinf$ is a $d$-dimensional parallelepiped. Each vertex $v$ of the slab has coordinates 
\[
v_1 = \frac{ \pm \frac{t}{2}-\sum_{i=2}^d a_i v_i}{a_1}, \qquad v_2,\ldots,v_d\in\{-\frac{1}{2},\frac{1}{2}\}.
\]
In other words, the slab is the affine image of $[-\frac{1}{2},\frac{1}{2}]^d$ via the map
\[
u \mapsto \left( \frac{ t u_1-\sum_{i=2}^d a_i u_i}{a_1}, u_2, \ldots, u_d \right),
\]
and the slice is one of its facets, i.e., the image of $\{\frac{1}{2}\}\times [-\frac{1}{2},\frac{1}{2}]^{d-1}$ via the same map. The Jacobian of this map reads
\[
J
=
\begin{pmatrix}
\tfrac{t}{a_1} & -\tfrac{a_2}{a_1} &  \cdots & -\tfrac{a_d}{a_1} \\[2pt]
0 & 1 & \cdots & 0 \\
\vdots &  \vdots & \ddots & \vdots \\
0 &  0 & \cdots & 1
\end{pmatrix},
\]
hence the volume of the $d$-dimensional slab is the determinant of $J$, namely $g_1^{(0,d)} = \frac{t}{a_1}$. Since the height of the parallelepiped with respect to its facet given by the slice obtained from $H(a,t)$ is $t$, we get that the volume of the slice is $f_1^{(0,d)} = \frac{1}{a_1}$.
\end{proof}

\begin{table}[H]
\footnotesize
\centering
\renewcommand{\arraystretch}{1.08}
\setlength{\tabcolsep}{3pt}
\begin{minipage}{0.49\textwidth}
\centering
\begin{tabular}{|c|c|c|}
\hline
$\bar a \ \text{s.t.}\ \frac{\bar a}{\|\bar a\|}\in R_i$ & $I_{i,j}(\bar a)$ & Formulas \\
\hline
\multirow{8}{*}{$(7, 6, 5, 3)$} & $[0, 1]$ & $f_{4}$ \\
\cline{2-3}
 & $[1, 3]$ & $f_{13}$ \\
\cline{2-3}
 & $[3, 5]$ & $f_{8}$ \\
\cline{2-3}
 & $[5, 7]$ & $f_{14}$ \\
\cline{2-3}
 & $[7, 9]$ & $f_{7}$ \\
\cline{2-3}
 & $[9, 11]$ & $f_{12}$ \\
\cline{2-3}
 & $[11, 15]$ & $f_{5}$ \\
\cline{2-3}
 & $[15, 21]$ & $f_{3}$ \\
\hline
\multirow{8}{*}{$(8, 4, 2, 1)$} & $[0, 1]$ & $f_{1}$ \\
\cline{2-3}
 & $[1, 3]$ & $f_{10}$ \\
\cline{2-3}
 & $[3, 5]$ & $f_{6}$ \\
\cline{2-3}
 & $[5, 7]$ & $f_{11}$ \\
\cline{2-3}
 & $[7, 9]$ & $f_{2}$ \\
\cline{2-3}
 & $[9, 11]$ & $f_{12}$ \\
\cline{2-3}
 & $[11, 13]$ & $f_{5}$ \\
\cline{2-3}
 & $[13, 15]$ & $f_{3}$ \\
\hline
\multirow{8}{*}{$(8, 6, 4, 1)$} & $[0, 1]$ & $f_{4}$ \\
\cline{2-3}
 & $[1, 3]$ & $f_{13}$ \\
\cline{2-3}
 & $[3, 5]$ & $f_{6}$ \\
\cline{2-3}
 & $[5, 7]$ & $f_{11}$ \\
\cline{2-3}
 & $[7, 9]$ & $f_{2}$ \\
\cline{2-3}
 & $[9, 11]$ & $f_{12}$ \\
\cline{2-3}
 & $[11, 17]$ & $f_{5}$ \\
\cline{2-3}
 & $[17, 19]$ & $f_{3}$ \\
\hline
\multirow{8}{*}{$(8, 6, 4, 3)$} & $[0, 1]$ & $f_{9}$ \\
\cline{2-3}
 & $[1, 3]$ & $f_{13}$ \\
\cline{2-3}
 & $[3, 5]$ & $f_{8}$ \\
\cline{2-3}
 & $[5, 7]$ & $f_{11}$ \\
\cline{2-3}
 & $[7, 9]$ & $f_{7}$ \\
\cline{2-3}
 & $[9, 13]$ & $f_{12}$ \\
\cline{2-3}
 & $[13, 15]$ & $f_{5}$ \\
\cline{2-3}
 & $[15, 21]$ & $f_{3}$ \\
\hline
\multirow{8}{*}{$(10, 7, 4, 2)$} & $[0, 1]$ & $f_{9}$ \\
\cline{2-3}
 & $[1, 3]$ & $f_{13}$ \\
\cline{2-3}
 & $[3, 5]$ & $f_{6}$ \\
\cline{2-3}
 & $[5, 9]$ & $f_{11}$ \\
\cline{2-3}
 & $[9, 11]$ & $f_{2}$ \\
\cline{2-3}
 & $[11, 15]$ & $f_{12}$ \\
\cline{2-3}
 & $[15, 19]$ & $f_{5}$ \\
\cline{2-3}
 & $[19, 23]$ & $f_{3}$ \\
\hline
\multirow{8}{*}{$(10, 4, 3, 2)$} & $[0, 1]$ & $f_{1}$ \\
\cline{2-3}
 & $[1, 5]$ & $f_{10}$ \\
\cline{2-3}
 & $[5, 7]$ & $f_{6}$ \\
\cline{2-3}
 & $[7, 9]$ & $f_{11}$ \\
\cline{2-3}
 & $[9, 11]$ & $f_{7}$ \\
\cline{2-3}
 & $[11, 13]$ & $f_{12}$ \\
\cline{2-3}
 & $[13, 15]$ & $f_{5}$ \\
\cline{2-3}
 & $[15, 19]$ & $f_{3}$ \\
\hline
\multirow{8}{*}{$(10, 6, 3, 2)$} & $[0, 1]$ & $f_{9}$ \\
\cline{2-3}
 & $[1, 3]$ & $f_{10}$ \\
\cline{2-3}
 & $[3, 5]$ & $f_{6}$ \\
\cline{2-3}
 & $[5, 9]$ & $f_{11}$ \\
\cline{2-3}
 & $[9, 11]$ & $f_{2}$ \\
\cline{2-3}
 & $[11, 15]$ & $f_{12}$ \\
\cline{2-3}
 & $[15, 17]$ & $f_{5}$ \\
\cline{2-3}
 & $[17, 21]$ & $f_{3}$ \\
\hline
\end{tabular}
\end{minipage}
\hfill
\begin{minipage}{0.49\textwidth}
\centering
\begin{tabular}{|c|c|c|}
\hline
$\bar a \ \text{s.t.}\ \frac{\bar a}{\|\bar a\|}\in R_i$ & $I_{i,j}(\bar a)$ & Formulas \\
\hline
\multirow{8}{*}{$(8, 4, 3, 2)$} & $[0, 1]$ & $f_{9}$ \\
\cline{2-3}
 & $[1, 3]$ & $f_{10}$ \\
\cline{2-3}
 & $[3, 5]$ & $f_{6}$ \\
\cline{2-3}
 & $[5, 7]$ & $f_{11}$ \\
\cline{2-3}
 & $[7, 9]$ & $f_{7}$ \\
\cline{2-3}
 & $[9, 11]$ & $f_{12}$ \\
\cline{2-3}
 & $[11, 13]$ & $f_{5}$ \\
\cline{2-3}
 & $[13, 17]$ & $f_{3}$ \\
\hline
\multirow{8}{*}{$(8, 5, 4, 2)$} & $[0, 1]$ & $f_{9}$ \\
\cline{2-3}
 & $[1, 3]$ & $f_{13}$ \\
\cline{2-3}
 & $[3, 5]$ & $f_{6}$ \\
\cline{2-3}
 & $[5, 7]$ & $f_{11}$ \\
\cline{2-3}
 & $[7, 9]$ & $f_{7}$ \\
\cline{2-3}
 & $[9, 11]$ & $f_{12}$ \\
\cline{2-3}
 & $[11, 15]$ & $f_{5}$ \\
\cline{2-3}
 & $[15, 19]$ & $f_{3}$ \\
\hline
\multirow{8}{*}{$(8, 6, 5, 4)$} & $[0, 1]$ & $f_{9}$ \\
\cline{2-3}
 & $[1, 3]$ & $f_{13}$ \\
\cline{2-3}
 & $[3, 5]$ & $f_{8}$ \\
\cline{2-3}
 & $[5, 7]$ & $f_{14}$ \\
\cline{2-3}
 & $[7, 11]$ & $f_{7}$ \\
\cline{2-3}
 & $[11, 13]$ & $f_{12}$ \\
\cline{2-3}
 & $[13, 15]$ & $f_{5}$ \\
\cline{2-3}
 & $[15, 23]$ & $f_{3}$ \\
\hline
\multirow{8}{*}{$(8, 7, 4, 2)$} & $[0, 1]$ & $f_{4}$ \\
\cline{2-3}
 & $[1, 3]$ & $f_{13}$ \\
\cline{2-3}
 & $[3, 5]$ & $f_{8}$ \\
\cline{2-3}
 & $[5, 7]$ & $f_{11}$ \\
\cline{2-3}
 & $[7, 9]$ & $f_{2}$ \\
\cline{2-3}
 & $[9, 13]$ & $f_{12}$ \\
\cline{2-3}
 & $[13, 17]$ & $f_{5}$ \\
\cline{2-3}
 & $[17, 21]$ & $f_{3}$ \\
\hline
\multirow{8}{*}{$(10, 7, 6, 2)$} & $[0, 1]$ & $f_{4}$ \\
\cline{2-3}
 & $[1, 5]$ & $f_{13}$ \\
\cline{2-3}
 & $[5, 7]$ & $f_{6}$ \\
\cline{2-3}
 & $[7, 9]$ & $f_{11}$ \\
\cline{2-3}
 & $[9, 11]$ & $f_{7}$ \\
\cline{2-3}
 & $[11, 13]$ & $f_{12}$ \\
\cline{2-3}
 & $[13, 21]$ & $f_{5}$ \\
\cline{2-3}
 & $[21, 25]$ & $f_{3}$ \\
\hline
\multirow{8}{*}{$(10, 8, 4, 3)$} & $[0, 1]$ & $f_{9}$ \\
\cline{2-3}
 & $[1, 3]$ & $f_{13}$ \\
\cline{2-3}
 & $[3, 5]$ & $f_{8}$ \\
\cline{2-3}
 & $[5, 9]$ & $f_{11}$ \\
\cline{2-3}
 & $[9, 11]$ & $f_{2}$ \\
\cline{2-3}
 & $[11, 17]$ & $f_{12}$ \\
\cline{2-3}
 & $[17, 19]$ & $f_{5}$ \\
\cline{2-3}
 & $[19, 25]$ & $f_{3}$ \\
\hline
\multirow{8}{*}{$(10, 8, 6, 3)$} & $[0, 1]$ & $f_{4}$ \\
\cline{2-3}
 & $[1, 5]$ & $f_{13}$ \\
\cline{2-3}
 & $[5, 7]$ & $f_{8}$ \\
\cline{2-3}
 & $[7, 9]$ & $f_{11}$ \\
\cline{2-3}
 & $[9, 11]$ & $f_{7}$ \\
\cline{2-3}
 & $[11, 15]$ & $f_{12}$ \\
\cline{2-3}
 & $[15, 21]$ & $f_{5}$ \\
\cline{2-3}
 & $[21, 27]$ & $f_{3}$ \\
\hline
\end{tabular}
\end{minipage}
\caption{
In the four-dimensional setting, where the polyhedral norm ball is the unit 4-cube
centered at the origin, this table records the piecewise rational formulas of slices ($f_i$) and slabs ($g_i$) obtained from each maximal open chamber $(a,t) \in R_i \times I_{i,j}(a)$.
Note that here we omit the superscripts. Instead of using $f_k^{(M,4)}$ and $g_k^{(M,4)}$, we use $f_k$ and $g_k$ to denote, respectively, the $k$-th rational function of the $M$-th moment of the slice and slab of the 4-cube.}
\label{table:4d_chamber}
\end{table}

\begin{table}[h!]
\centering
\begin{tabular}{|l|c|}
\hline
\textbf{$t$-range} & \textbf{Formulas} \\ \hline
$0 \le t \le a_1-a_2-a_3-a_4$ 
& $f_1, g_1$ \\ \hline
$|a_1-(a_2+a_3+a_4)| \le t \le a_1-a_2-a_3+a_4$ 
& $f_{10}, g_{10}$ \\ \hline
$|a_1-a_2-a_3+a_4| \le t \le a_3-|a_1-a_2-a_4|$ 
& $f_{13}, g_{13}$ \\ \hline
$a_1+a_2-a_3-a_4 \le t \le -a_1+a_2+a_3+a_4$ 
& $f_{14}, g_{14}$ \\ \hline
$0 \le t \le -a_1+a_2+a_3-a_4$ 
& $f_4, g_4$ \\ \hline
$0 \le t \le a_4 - |a_1-a_2-a_3|$ 
& $f_9, g_9$ \\ \hline
$|a_1-a_2-a_3| + a_4 \le t \le a_1-a_2+a_3-a_4$ 
& $f_6, g_6$ \\ \hline
$a_1-a_2+a_3-a_4 \le t \le a_2 - |a_1-a_3-a_4|$ 
& $f_8, g_8$ \\ \hline
$a_3+|a_1-a_2-a_4| \le t \le a_1-|a_2-a_3-a_4|$ 
& $f_{11}, g_{11}$ \\ \hline
$a_2+|a_1-a_3-a_4| \le t \le a_1-a_2+a_3+a_4$ 
& $f_7, g_7$ \\ \hline
$a_1-a_2+a_3+a_4 \le t \le a_1+a_2-a_3-a_4$ 
& $f_2, g_2$ \\ \hline
$a_1 + |a_2-a_3-a_4| \le t \le a_1+a_2-a_3+a_4$ 
& $f_{12}, g_{12}$ \\ \hline
$a_1+a_2-a_3+a_4 \le t \le a_1+a_2+a_3-a_4$ 
& $f_5, g_5$ \\ \hline
$a_1+a_2+a_3-a_4 \le t \le a_1+a_2+a_3+a_4$ 
& $f_3, g_3$ \\ \hline
\end{tabular}
\caption{
In the four-dimensional setting, where the polyhedral norm ball is the unit 4-cube centered at the origin, this table records the piecewise formulas for the slice/slab volume and moment as functions of $(a,t)$. For each interval of $t$ shown in the left column, the corresponding formulas in the right column are valid on that interval. Note that we omit the superscripts here and use $f_i$ and $g_i$ to denote, respectively, the $i$-th rational expressions for the $M$-th moment of the slice and slab of the 4-cube.}
\label{table:4d_kkchamber}
\end{table}

\subsection{Conjectured extremal values}\label{subsec:extreme4cube}
After recovering the explicit piecewise rational functions giving the volume of slices and slabs of the cube, we used \texttt{Mathematica} to find the maximum and minimum numerically at any given height $t$. With this, we confirm the previously known formulas for the $2$- and $3$-dimensional cubes. 
To formulate \Cref{conj:extremal_slices,conj:extremal_slabs} we employ the algebraic method from \Cref{sec:algebraic}.
This gives an implicit description of all the critical points of the volume.
Unfortunately, we cannot turn this implicit description into an explicit one for all the chambers. This is due to the algebraic complexity of the set of critical points.
However, the contributions for the maxima and minima seem to be given by critical points for which we have an explicit description. This has been verified numerically and it is the origin of the conjectures. 
The colors used in \Cref{fig:4cube_maxmin} refer to the chambers the extremizer lies in, according to the numbering in \Cref{table:4d_kkchamber}. More specifically, red corresponds to the first chamber, green to the second, cyan to the third, purple to the fifth, yellow to the fourteenth. Moreover, we use blue for the common boundary between the first and the second chamber, and black when the extremal hyperplane does not intersect the cube at all.

\subsection{Proof of \Cref{thm:moment_formula}}
\label{sec:pf_of_thm4}

With the theory of \Cref{sec:methodology} and all the computational results of \Cref{sec:volumes} at hand, we are fully equipped to prove \Cref{thm:moment_formula}, namely, that the higher moments of slices and slabs of the cube are piecewise rational functions, whose pieces are supported on the same domain as the respective volume functions. Notably, for $d\leq 4$ and $M=1,2,3,4$, we observe that also the degrees of these rational functions on each maximal slicing chamber agree.

\begin{definition}
\label{def:rational_fn_degree}
Let \(f=p/q\) be a rational function in the variables \((a_1,\dots,a_d,t)\), written in reduced form with \(p\) and \(q\) polynomials having no common factors, and having respective degrees $\deg(p)$ and $\deg(q)$. We define the \emph{degree} of $f$ as \(\deg(f):=\deg(p)-\deg(q)\). 
\end{definition}

\begin{proof}[Proof of \Cref{thm:moment_formula}]
\Cref{lemma:summaryofBDLMpaper,lem:slab-chambers} combined with \Cref{eqn:moment} imply that the higher moments are piecewise rational functions and that the domains of their pieces are supported on the same domains as the pieces of the rational function expressions of the volumes. For the degrees, we observe that for any fixed $d \in \{2,3,4\}$ and any fixed $k$ we have
\[
    \deg(f_k^{(M,d)} )= \deg(f_k^{(0,d)}) \qquad \text{for } M=1,2,3,4
\]
for the moments of slices of the $d$-dimensional cube $\hypercube^d$, and 
\[
    \deg(g_k^{(M,d)} )= \deg(g_k^{(0,d)}) \qquad \text{for } M=1,2,3,4
\]
for the moments of slabs of the $d$-dimensional cube $\hypercube^d$.
\end{proof}

\section*{Conclusion}

In this paper, we developed a unified algebraic and combinatorial framework for computing volumes and higher-order moments of slices and slabs of arbitrary polyhedral norm balls. Our approach combines explicit barycentric triangulations, chamber decompositions determined by linear orderings of vertices of the polyhedral norm ball, and structural formulas showing that every moment expression contains the volume as a multiplicative factor. These methods also imply that the volume and higher moments of slices and slabs are piecewise rational functions in the same parameters that define the slicing hyperplanes.
We applied our methods computationally to obtain all piecewise rational functions for the volumes and first four moments of slabs and slices of the hypercubes $\hypercube^2, \hypercube^3$ and $\hypercube^4$, and closed form formulas for moments of arbitrary degree and their critical points in dimension $2$.

These results demonstrate that, despite the apparent analytic complexity of high-dimensional slices and slabs, the underlying algebraic structure exhibits strong regularity: volume and moment formulas are piecewise rational, and are determined entirely by finitely many maximal chambers in the parameter space. Our methods also provide a computational pipeline for symbolic integration of slabs of polytopes with parametric vertices and show that, for fixed dimension $d$, the computational complexity is polynomial in the input size.

\subsection*{Acknowledgments}
We are grateful to Alexander Koldobsky for bringing this research direction to our attention, and to him and Hermann König for stimulating email exchanges. We also thank Martin Henk for his useful comments.
JDL and YL were partially supported by NSF grants 2348578 and 2434665.
MB was supported by the SPP 2458 "Combinatorial Synergies", funded by the Deutsche Forschungsgemeinschaft (DFG, German Research Foundation).
CM is supported by Dr. Max R\"ossler, the Walter Haefner Foundation, and the ETH Z\"urich Foundation. 

\vspace{1cm}

\printbibliography

% $ $
\vskip 1in
\noindent
\begin{minipage}[t]{0.48\textwidth}
\centering
\textsc{Marie-Charlotte Brandenburg}\\
\textsc{Ruhr-Universität Bochum}\\
\url{marie-charlotte.brandenburg@rub.de}
\end{minipage}
\hfill
\begin{minipage}[t]{0.48\textwidth}
\centering
\textsc{Jesús A. De Loera}\\
\textsc{University of California, Davis}\\
\url{jadeloera@ucdavis.edu}
\end{minipage}

\vspace{1em}

\noindent
\begin{minipage}[t]{0.48\textwidth}
\centering
\textsc{Yu Luo}\\
\textsc{University of California, Davis}\\
\url{ayuluo@ucdavis.edu}
\end{minipage}
\hfill
\begin{minipage}[t]{0.48\textwidth}
\centering
\textsc{Chiara Meroni}\\
\textsc{ETH Institute for Theoretical Studies}\\
\url{chiara.meroni@eth-its.ethz.ch}
\end{minipage}
 
\newpage
\appendix

\section{Appendix}\label{appendix}

We present in the appendix all volume (for slices and slabs, in dimensions two, three, and four) and moment (for slices and slabs, in dimensions two and three, of order at most four) formulas appearing in the paper. All computations were carried out symbolically in \texttt{SageMath} \cite{sagemath}. 
The code and list of all computed volume and moment functions are publicly available at \\
{
\centering \url{https://github.com/RainCamel/slab_of_the_poly_norms}
}

 \etocsettocstyle{\subsection*{Appendix contents}}{}
 \setcounter{tocdepth}{2}
 \localtableofcontents

%===================================================================

%-----------------------------------------------------------------
\subsection{Volume formulas: dimension two}
\label{appendix:2dvol}

\subsubsection{Slices}
\label{appendix:2dslice}

\begin{appendixformulas}
\[
\sliceinf =
\begin{cases}
f_1^{(0,2)} = \dfrac{1}{a_1}, & 0 \le t \le a_1 - a_2, \\[8pt]
f_2^{(0,2)} = \dfrac{a_1 + a_2 - t}{2 a_1 a_2}, & a_1 - a_2 \le t \le a_1 + a_2.
\end{cases}
\]
\end{appendixformulas}

\subsubsection{Slabs}
\label{appendix:2dslab}

\begin{appendixformulas}
\[
\slabinf =
\begin{cases}
g_1^{(0,2)} = \dfrac{t}{a_1}, & 0 \le t \le a_1 - a_2, \\[8pt]
g_2^{(0,2)} = 1 - \dfrac{(a_1 + a_2 - t)^2}{4 a_1 a_2}, & a_1 - a_2 \le t \le a_1 + a_2.
\end{cases}
\]
\end{appendixformulas}

%-----------------------------------------------------------------
\subsection{Volume formulas: dimension three}
\label{appendix:3dvol}

\subsubsection{Slices}
\label{appendix:3dslice}

\begin{appendixformulas}
\[
\sliceinf =
\begin{cases}
f_1^{(0,3)} = \dfrac{1}{a_1}, & 0 \le t \le a_1 - (a_2 + a_3), \\[8pt]
f_2^{(0,3)} = \dfrac{1}{a_1} - \dfrac{t^2 + (a_2 + a_3 - a_1)^2}{4 a_1 a_2 a_3}, & 0 \le t \le (a_2 + a_3) - a_1, \\[8pt]
f_3^{(0,3)} = \dfrac{1}{a_1} - \dfrac{(t + a_2 + a_3 - a_1)^2}{8 a_1 a_2 a_3}, & |a_1 - (a_2 + a_3)| \le t \le a_1 - a_2 + a_3, \\[8pt]
f_4^{(0,3)} = \dfrac{a_1 + a_2 - t}{2 a_1 a_2}, & a_1 - a_2 + a_3 \le t \le a_1 + a_2 - a_3, \\[8pt]
f_5^{(0,3)} = \dfrac{(a_1 + a_2 + a_3 - t)^2}{8 a_1 a_2 a_3}, & a_1 + a_2 - a_3 \le t \le a_1 + a_2 + a_3.
\end{cases}
\]
\end{appendixformulas}

\subsubsection{Slabs}
\label{appendix:3dslab}

\begin{appendixformulas}
\[
\slabinf =
\begin{cases}
g_1^{(0,3)} = \dfrac{t}{a_1}, & 0 \le t \le a_1 - (a_2 + a_3), \\[8pt]
g_2^{(0,3)} = \dfrac{t}{a_1} - \dfrac{t^3 + 3(a_2 + a_3 - a_1)^2 t}{12 a_1 a_2 a_3}, & 0 \le t \le (a_2 + a_3) - a_1, \\[8pt]
g_3^{(0,3)} = \dfrac{t}{a_1} - \dfrac{(t + a_2 + a_3 - a_1)^3}{24 a_1 a_2 a_3}, & |a_1 - (a_2 + a_3)| \le t \le a_1 - a_2 + a_3, \\[8pt]
g_4^{(0,3)} = 1 - \dfrac{a_3^2}{12 a_1 a_2} - \dfrac{((a_1 + a_2) - t)^2}{4 a_1 a_2}, & a_1 - a_2 + a_3 \le t \le a_1 + a_2 - a_3, \\[8pt]
g_5^{(0,3)} = 1 - \dfrac{(a_1 + a_2 + a_3 - t)^3}{24 a_1 a_2 a_3}, & a_1 + a_2 - a_3 \le t \le a_1 + a_2 + a_3.
\end{cases}
\]
\end{appendixformulas}

%-------------------------------------------------------------

\subsection{Volume formulas: dimension four}
\label{appendix:4dvol}

%-------------------------------------------------------------

The ranges in which these rational functions are valid are indicated in \Cref{table:4d_kkchamber}.

\subsubsection{Slices}
\label{appendix:4dslice}

\begin{appendixformulas}
\begin{align*}
f_1^{(0,4)} &= \frac{1}{a_1},\\[6pt]
f_2^{(0,4)} &= \frac{1}{2}\,\frac{a_1 + a_2 - t}{a_1 a_2},\\[6pt]
f_3^{(0,4)} &= \frac{(a_1 + a_2 + a_3 + a_4 - t)^3}{48\,a_1 a_2 a_3 a_4}\\
f_4^{(0,4)} &= \frac{1}{a_3}
      - \frac{t^2 + (a_3 - a_1 - a_2)^2 + \tfrac{1}{3}a_4^2}{4\,a_1 a_2 a_3},\\[6pt]
f_5^{(0,4)} &= \frac{(t - (a_1 + a_2 + a_3))^2}{8\,a_1 a_2 a_3}
      + \frac{a_4^2}{24\,a_1 a_2 a_3},\\[6pt]
f_6^{(0,4)} &= \frac{1}{a_1}
      - \frac{(t - (a_1 - a_2 - a_3))^2}{8\,a_1 a_2 a_3}
      - \frac{a_4^2}{24\,a_1 a_2 a_3},\\[6pt]
f_7^{(0,4)} &= \frac{(t - a_1)^3}{24\,a_1 a_2 a_3 a_4}
      + \frac{\bigl((a_2 + a_3 - a_4)^2 - 4a_2 a_3\bigr)(t - a_1)}{8\,a_1 a_2 a_3 a_4}
      + \frac{1}{2a_1},\\[6pt]
f_8^{(0,4)} &= \frac{(t - (a_3 + a_4))^3}{24\,a_1 a_2 a_3 a_4}
      + \frac{\bigl((a_1 - a_2)^2 - 4a_3 a_4\bigr)(t - (a_3 + a_4))}{8\,a_1 a_2 a_3 a_4}
      + \frac{a_1 + a_2 - a_3 - a_4}{2\,a_1 a_2},\\[6pt]
f_9^{(0,4)} &= \frac{(a_1 - a_2 - a_3 - a_4)t^2}{8\,a_1 a_2 a_3 a_4}
      + \frac{(a_1 - a_2 - a_3 - a_4)^3}{24\,a_1 a_2 a_3 a_4}
      + \frac{1}{a_1},\\[6pt]
f_{10}^{(0,4)} &= -\frac{(t - (a_1 - a_2 - a_3 - a_4))^3}{48\,a_1 a_2 a_3 a_4}
      + \frac{1}{a_1},\\[6pt]
f_{11}^{(0,4)} &= \frac{(t - (a_1 - a_2 + a_3 + a_4))^3}{48\,a_1 a_2 a_3 a_4}
      - \frac{t - (a_1 - a_2 + a_3 + a_4)}{2\,a_1 a_2}
      + \frac{2a_2 - a_3 - a_4}{2\,a_1 a_2},\\[6pt]
f_{12}^{(0,4)} &= \frac{(t - (a_1 + a_2 - a_3 - a_4))^3}{48\,a_1 a_2 a_3 a_4}
      + \frac{a_1 + a_2 - t}{2\,a_1 a_2},\\[6pt]
f_{13}^{(0,4)} &= \frac{(t - (-a_1 + a_2 + a_3 + 3a_4))^3}{48\,a_1 a_2 a_3 a_4}
      + \frac{(a_1 - a_2 - a_3 - a_4)(t - (-a_1 + a_2 + a_3 + 3a_4))}{2\,a_1 a_2 a_3}\\
      &\quad + \frac{1}{a_2} + \frac{1}{a_3}
      + \frac{-a_1^2 + 3a_1 a_4 - a_2^2 - 3a_2 a_4 - a_3^2 - 3a_3 a_4 - 2a_4^2}{2\,a_1 a_2 a_3},\\[6pt]
f_{14}^{(0,4)} &= \frac{(t - \tfrac{a_1 + a_2 + a_3 + a_4}{3})^3}{16\,a_1 a_2 a_3 a_4}
      + \frac{(a_1^2 - a_1 a_2 - a_1 a_3 - a_1 a_4 + a_2^2 - a_2 a_3 - a_2 a_4 + a_3^2 - a_3 a_4 + a_4^2)
              (t - \tfrac{a_1 + a_2 + a_3 + a_4}{3})}{6\,a_1 a_2 a_3 a_4}\\
      &\quad + \frac{2}{9}\!\left(\frac{1}{a_1} + \frac{1}{a_2} + \frac{1}{a_3} + \frac{1}{a_4}\right)
      + \frac{\frac{a_1^3 + a_2^3 + a_3^3 + a_4^3}{27}
             - \frac{a_1^2(a_2 + a_3 + a_4) + a_2^2(a_1 + a_3 + a_4)
             + a_3^2(a_1 + a_2 + a_4) + a_4^2(a_1 + a_2 + a_3)}{18}}{a_1 a_2 a_3 a_4}.
\end{align*}
\end{appendixformulas}

\subsubsection{Slabs}
\label{appendix:4dslab}

\begin{appendixformulas}
\begin{align*}
g_1^{(0,4)} &= \frac{t}{a_1},\\
g_2^{(0,4)} &= 1 - \frac{a_3^{2} + a_4^{2}}{12\,a_1 a_2}
  - \frac{\bigl(t - a_1 - a_2\bigr)^{2}}{4\,a_1 a_2},\\
g_3^{(0,4)} &= \frac{t}{a_1} - \frac{(t^2 + a_4^{2})t}{12\,a_1 a_2 a_3}
  - \frac{\bigl(a_1 - a_2 - a_3\bigr)^{2}t}{4\,a_1 a_2 a_3},\\
g_4^{(0,4)} &= \frac{1}{24 a_1 a_2 a_3}\Bigl(
    a_1^3 - 3a_1^2 a_2 + 3a_1 a_2^2 - a_2^3 - 3a_1^2 a_3 + 6a_1 a_2 a_3 - 3a_2^2 a_3 + 3a_1 a_3^2 - 3a_2 a_3^2 - a_3^3 + a_1 a_4^2 - a_2 a_4^2 - a_3 a_4^2 \\
    &\qquad
    - 3a_1^2 t + 6a_1 a_2 t - 3a_2^2 t + 6a_1 a_3 t 
    + 18a_2 a_3 t - 3a_3^2 t - a_4^2 t + 3a_1 t^2 - 3a_2 t^2 - 3a_3 t^2 - t^3
  \Bigr),\\
g_5^{(0,4)} &= -\frac{1}{24 a_1 a_2 a_3}\Bigl(
    a_1^3 + 3a_1^2 a_2 + 3a_1 a_2^2 + a_2^3 + 3a_1^2 a_3 - 18a_1 a_2 a_3 
    + 3a_2^2 a_3 + 3a_1 a_3^2 + 3a_2 a_3^2 + a_3^3 + a_1 a_4^2 + a_2 a_4^2\\
&\qquad
     + a_3 a_4^2 - 3a_1^2 t - 6a_1 a_2 t - 3a_2^2 t - 6a_1 a_3 t - 6a_2 a_3 t - 3a_3^2 t - a_4^2 t + 3a_1 t^2 + 3a_2 t^2 + 3a_3 t^2 - t^3
  \Bigr),\\   
g_6^{(0,4)} &= \frac{1}{24 a_1 a_2 a_3 a_4}\Bigl(
    a_1^3 - 3a_1^2 a_2 + 3a_1 a_2^2 - a_2^3 - 3a_1^2 a_3 + 6a_1 a_2 a_3 - 3a_2^2 a_3 + 3a_1 a_3^2 - 3a_2 a_3^2 - a_3^3 - 3a_1^2 a_4 + 6a_1 a_2 a_4\\
&\qquad
     - 3a_2^2 a_4 + 6a_1 a_3 a_4 + 18a_2 a_3 a_4 - 3a_3^2 a_4 + 3a_1 a_4^2 - 3a_2 a_4^2 - 3a_3 a_4^2 - a_4^3 + a_1 t^2 - a_2 t^2 - a_3 t^2 - a_4 t^2
  \Bigr) t,\\
g_7^{(0,4)} &= \frac{1}{96 a_1 a_2 a_3 a_4}\Bigl(
        a_1^4 - 4 a_1^3 a_2 + 6 a_1^2 a_2^2 - 4 a_1 a_2^3 + a_2^4 + 6 a_1^2 a_3^2 - 12 a_1 a_2 a_3^2 + 6 a_2^2 a_3^2 + a_3^4 - 12 a_1^2 a_3 a_4 + 24 a_1 a_2 a_3 a_4  \\
    &\qquad
        - 12 a_2^2 a_3 a_4 - 4 a_3^3 a_4
        + 6 a_1^2 a_4^2 - 12 a_1 a_2 a_4^2 + 6 a_2^2 a_4^2 
        + 6 a_3^2 a_4^2 - 4 a_3 a_4^3 + a_4^4
        - 12 a_1^2 a_3 t + 24 a_1 a_2 a_3 t - 12 a_2^2 a_3 t \\
    &\qquad
        - 4 a_3^3 t - 12 a_1^2 a_4 t + 24 a_1 a_2 a_4 t
        - 12 a_2^2 a_4 t + 48 a_1 a_3 a_4 t + 48 a_2 a_3 a_4 t
        - 12 a_3^2 a_4 t - 12 a_3 a_4^2 t - 4 a_4^3 t
        + 6 a_1^2 t^2\\
    &\qquad
         - 12 a_1 a_2 t^2 + 6 a_2^2 t^2 
         + 6 a_3^2 t^2 - 12 a_3 a_4 t^2 + 6 a_4^2 t^2
        - 4 a_3 t^3 - 4 a_4 t^3 + t^4
      \Bigr),\\[4pt]
g_8^{(0,4)} &= \frac{1}{96 a_1 a_2 a_3 a_4}\Bigl(
        a_1^4 + 6 a_1^2 a_2^2 + a_2^4
        - 12 a_1^2 a_2 a_3 - 4 a_2^3 a_3
        + 6 a_1^2 a_3^2 + 6 a_2^2 a_3^2 - 4 a_2 a_3^3 + a_3^4
        - 12 a_1^2 a_2 a_4 - 4 a_2^3 a_4 \\
    &\qquad
        - 12 a_1^2 a_3 a_4 + 48 a_1 a_2 a_3 a_4
        - 12 a_2^2 a_3 a_4 - 12 a_2 a_3^2 a_4 - 4 a_3^3 a_4
        + 6 a_1^2 a_4^2 + 6 a_2^2 a_4^2 - 12 a_2 a_3 a_4^2
        + 6 a_3^2 a_4^2 - 4 a_2 a_4^3\\
    &\qquad
        - 4 a_1^3 t - 12 a_1 a_2^2 t + 24 a_1 a_2 a_3 t - 4 a_3 a_4^3 + a_4^4 
        - 12 a_1 a_3^2 t + 24 a_1 a_2 a_4 t
        + 24 a_1 a_3 a_4 t + 48 a_2 a_3 a_4 t
        - 12 a_1 a_4^2 t\\
    &\qquad
        - 12 a_2 a_3 t^2 + 6 a_3^2 t^2
        - 12 a_2 a_4 t^2 - 12 a_3 a_4 t^2 + 6 a_4^2 t^2 + 6 a_1^2 t^2 + 6 a_2^2 t^2 
        - 4 a_1 t^3 + t^4
      \Bigr),\\[4pt]
g_9^{(0,4)} &= -\frac{1}{192 a_1 a_2 a_3 a_4}\Bigl(
        a_1^4 - 4 a_1^3 a_2 + 6 a_1^2 a_2^2 - 4 a_1 a_2^3 + a_2^4
        - 4 a_1^3 a_3 + 12 a_1^2 a_2 a_3 - 12 a_1 a_2^2 a_3
        + 4 a_2^3 a_3 + 6 a_1^2 a_3^2 \\
    &\qquad
        - 12 a_1 a_2 a_3^2 + 6 a_2^2 a_3^2
        - 4 a_1 a_3^3 + 4 a_2 a_3^3 + a_3^4
        - 4 a_1^3 a_4 + 12 a_1^2 a_2 a_4 - 12 a_1 a_2^2 a_4
        + 4 a_2^3 a_4 + 12 a_1^2 a_3 a_4 \\
    &\qquad
        - 24 a_1 a_2 a_3 a_4 + 12 a_2^2 a_3 a_4
        - 12 a_1 a_3^2 a_4 + 12 a_2 a_3^2 a_4 + 4 a_3^3 a_4
        + 6 a_1^2 a_4^2 - 12 a_1 a_2 a_4^2 + 6 a_2^2 a_4^2
        - 12 a_1 a_3 a_4^2 + 12 a_2 a_3 a_4^2 \\
    &\qquad
        + 6 a_3^2 a_4^2 - 4 a_1 a_4^3 + 4 a_2 a_4^3 + 4 a_3 a_4^3 + a_4^4 
        - 4 a_1^3 t + 12 a_1^2 a_2 t - 12 a_1 a_2^2 t + 4 a_2^3 t
        + 12 a_1^2 a_3 t - 24 a_1 a_2 a_3 t \\
    &\qquad
        + 12 a_2^2 a_3 t - 12 a_1 a_3^2 t + 12 a_2 a_3^2 t + 4 a_3^3 t 
        + 12 a_1^2 a_4 t - 24 a_1 a_2 a_4 t + 12 a_2^2 a_4 t
        - 24 a_1 a_3 a_4 t - 168 a_2 a_3 a_4 t \\
    &\qquad
        + 12 a_3^2 a_4 t - 12 a_1 a_4^2 t + 12 a_2 a_4^2 t
        + 12 a_3 a_4^2 t + 4 a_4^3 t 
        + 6 a_1^2 t^2 - 12 a_1 a_2 t^2 + 6 a_2^2 t^2
        - 12 a_1 a_3 t^2 + 12 a_2 a_3 t^2 + 6 a_3^2 t^2 \\
    &\qquad
        - 12 a_1 a_4 t^2 + 12 a_2 a_4 t^2 + 12 a_3 a_4 t^2 + 6 a_4^2 t^2 
        - 4 a_1 t^3 + 4 a_2 t^3 + 4 a_3 t^3 + 4 a_4 t^3 + t^4
      \Bigr),\\
g_{10}^{(0,4)} &= -\frac{1}{192 a_1 a_2 a_3 a_4}\Bigl(
        a_1^4 + 4 a_1^3 a_2 + 6 a_1^2 a_2^2 + 4 a_1 a_2^3 + a_2^4
        + 4 a_1^3 a_3 + 12 a_1^2 a_2 a_3 + 12 a_1 a_2^2 a_3
        + 4 a_2^3 a_3 + 6 a_1^2 a_3^2 \\
    &\qquad
        + 12 a_1 a_2 a_3^2 + 6 a_2^2 a_3^2
        + 4 a_1 a_3^3 + 4 a_2 a_3^3 + a_3^4 
        + 4 a_1^3 a_4 + 12 a_1^2 a_2 a_4 + 12 a_1 a_2^2 a_4
        + 4 a_2^3 a_4 + 12 a_1^2 a_3 a_4 \\
    &\qquad
        - 168 a_1 a_2 a_3 a_4 + 12 a_2^2 a_3 a_4
        + 12 a_1 a_3^2 a_4 + 12 a_2 a_3^2 a_4 + 4 a_3^3 a_4 
        + 6 a_1^2 a_4^2 + 12 a_1 a_2 a_4^2 + 6 a_2^2 a_4^2
        + 12 a_1 a_3 a_4^2 + 12 a_2 a_3 a_4^2 \\
    &\qquad
        + 6 a_3^2 a_4^2 + 4 a_1 a_4^3 + 4 a_2 a_4^3 + 4 a_3 a_4^3 + a_4^4 
        - 4 a_1^3 t - 12 a_1^2 a_2 t - 12 a_1 a_2^2 t - 4 a_2^3 t
        - 12 a_1^2 a_3 t - 24 a_1 a_2 a_3 t \\
    &\qquad
        - 12 a_2^2 a_3 t - 12 a_1 a_3^2 t - 12 a_2 a_3^2 t - 4 a_3^3 t 
        - 12 a_1^2 a_4 t - 24 a_1 a_2 a_4 t - 12 a_2^2 a_4 t
        - 24 a_1 a_3 a_4 t - 24 a_2 a_3 a_4 t \\
    &\qquad
        - 12 a_3^2 a_4 t - 12 a_1 a_4^2 t - 12 a_2 a_4^2 t
        - 12 a_3 a_4^2 t - 4 a_4^3 t 
        + 6 a_1^2 t^2 + 12 a_1 a_2 t^2 + 6 a_2^2 t^2
        + 12 a_1 a_3 t^2 + 12 a_2 a_3 t^2 + 6 a_3^2 t^2 \\
    &\qquad
        + 12 a_1 a_4 t^2 + 12 a_2 a_4 t^2 + 12 a_3 a_4 t^2 + 6 a_4^2 t^2 
        - 4 a_1 t^3 - 4 a_2 t^3 - 4 a_3 t^3 - 4 a_4 t^3 + t^4
      \Bigr),\\[4pt]
g_{11}^{(0,4)} &= \frac{1}{192 a_1 a_2 a_3 a_4}\Bigl(
        a_1^4 - 4 a_1^3 a_2 + 6 a_1^2 a_2^2 - 4 a_1 a_2^3 + a_2^4 
        - 4 a_1^3 a_3 + 12 a_1^2 a_2 a_3 - 12 a_1 a_2^2 a_3
        + 4 a_2^3 a_3 + 6 a_1^2 a_3^2 \\
    &\qquad
        - 12 a_1 a_2 a_3^2 + 6 a_2^2 a_3^2
        - 4 a_1 a_3^3 + 4 a_2 a_3^3 + a_3^4 
        + 4 a_1^3 a_4 - 12 a_1^2 a_2 a_4 + 12 a_1 a_2^2 a_4
        - 4 a_2^3 a_4 - 12 a_1^2 a_3 a_4 \\
    &\qquad
        + 24 a_1 a_2 a_3 a_4 - 12 a_2^2 a_3 a_4
        + 12 a_1 a_3^2 a_4 - 12 a_2 a_3^2 a_4 - 4 a_3^3 a_4 
        + 6 a_1^2 a_4^2 - 12 a_1 a_2 a_4^2 + 6 a_2^2 a_4^2
        - 12 a_1 a_3 a_4^2 + 12 a_2 a_3 a_4^2 \\
    &\qquad
        + 6 a_3^2 a_4^2 + 4 a_1 a_4^3 - 4 a_2 a_4^3 - 4 a_3 a_4^3 + a_4^4 
        + 4 a_1^3 t - 12 a_1^2 a_2 t + 12 a_1 a_2^2 t - 4 a_2^3 t
        - 12 a_1^2 a_3 t + 24 a_1 a_2 a_3 t \\
    &\qquad
        - 12 a_2^2 a_3 t + 12 a_1 a_3^2 t - 12 a_2 a_3^2 t - 4 a_3^3 t 
        - 36 a_1^2 a_4 t + 72 a_1 a_2 a_4 t - 36 a_2^2 a_4 t
        + 72 a_1 a_3 a_4 t + 120 a_2 a_3 a_4 t \\
    &\qquad
        - 36 a_3^2 a_4 t + 12 a_1 a_4^2 t - 12 a_2 a_4^2 t
        - 12 a_3 a_4^2 t - 12 a_4^3 t 
        + 6 a_1^2 t^2 - 12 a_1 a_2 t^2 + 6 a_2^2 t^2
        - 12 a_1 a_3 t^2 + 12 a_2 a_3 t^2 + 6 a_3^2 t^2 \\
    &\qquad
        + 12 a_1 a_4 t^2 - 12 a_2 a_4 t^2 - 12 a_3 a_4 t^2 + 6 a_4^2 t^2 
        + 4 a_1 t^3 - 4 a_2 t^3 - 4 a_3 t^3 - 12 a_4 t^3 + t^4
      \Bigr),\\[4pt]
g_{12}^{(0,4)} &= \frac{1}{192 a_1 a_2 a_3 a_4}\Bigl(
        a_1^4 - 4 a_1^3 a_2 + 6 a_1^2 a_2^2 - 4 a_1 a_2^3 + a_2^4 
        + 4 a_1^3 a_3 - 12 a_1^2 a_2 a_3 + 12 a_1 a_2^2 a_3
        - 4 a_2^3 a_3 + 6 a_1^2 a_3^2 \\
    &\qquad
        - 12 a_1 a_2 a_3^2 + 6 a_2^2 a_3^2
        + 4 a_1 a_3^3 - 4 a_2 a_3^3 + a_3^4
        + 4 a_1^3 a_4 - 12 a_1^2 a_2 a_4 + 12 a_1 a_2^2 a_4
        - 4 a_2^3 a_4 - 36 a_1^2 a_3 a_4 \\
    &\qquad
        + 72 a_1 a_2 a_3 a_4 - 36 a_2^2 a_3 a_4
        + 12 a_1 a_3^2 a_4 - 12 a_2 a_3^2 a_4 - 12 a_3^3 a_4 
        + 6 a_1^2 a_4^2 - 12 a_1 a_2 a_4^2 + 6 a_2^2 a_4^2
        + 12 a_1 a_3 a_4^2 - 12 a_2 a_3 a_4^2 \\
    &\qquad
        + 6 a_3^2 a_4^2 + 4 a_1 a_4^3 - 4 a_2 a_4^3 - 12 a_3 a_4^3 + a_4^4 
        - 4 a_1^3 t + 12 a_1^2 a_2 t - 12 a_1 a_2^2 t + 4 a_2^3 t
        - 12 a_1^2 a_3 t + 24 a_1 a_2 a_3 t \\
    &\qquad
        - 12 a_2^2 a_3 t - 12 a_1 a_3^2 t + 12 a_2 a_3^2 t - 4 a_3^3 t 
        - 12 a_1^2 a_4 t + 24 a_1 a_2 a_4 t - 12 a_2^2 a_4 t
        + 72 a_1 a_3 a_4 t + 120 a_2 a_3 a_4 t \\
    &\qquad
        - 12 a_3^2 a_4 t - 12 a_1 a_4^2 t + 12 a_2 a_4^2 t
        - 12 a_3 a_4^2 t - 4 a_4^3 t 
        + 6 a_1^2 t^2 - 12 a_1 a_2 t^2 + 6 a_2^2 t^2
        + 12 a_1 a_3 t^2 - 12 a_2 a_3 t^2 + 6 a_3^2 t^2 \\
    &\qquad
        + 12 a_1 a_4 t^2 - 12 a_2 a_4 t^2 - 36 a_3 a_4 t^2 + 6 a_4^2 t^2
        - 4 a_1 t^3 + 4 a_2 t^3 - 4 a_3 t^3 - 4 a_4 t^3 + t^4
      \Bigr),\\[4pt]
g_{13}^{(0,4)} &= \frac{1}{192 a_1 a_2 a_3 a_4}\Bigl(
        a_1^4 + 4 a_1^3 a_2 + 6 a_1^2 a_2^2 + 4 a_1 a_2^3 + a_2^4 
        - 4 a_1^3 a_3 - 12 a_1^2 a_2 a_3 - 12 a_1 a_2^2 a_3
        - 4 a_2^3 a_3 + 6 a_1^2 a_3^2 \\
    &\qquad
        + 12 a_1 a_2 a_3^2 + 6 a_2^2 a_3^2
        - 4 a_1 a_3^3 - 4 a_2 a_3^3 + a_3^4 
        - 4 a_1^3 a_4 - 12 a_1^2 a_2 a_4 - 12 a_1 a_2^2 a_4
        - 4 a_2^3 a_4 - 36 a_1^2 a_3 a_4 \\
    &\qquad
        + 120 a_1 a_2 a_3 a_4 - 36 a_2^2 a_3 a_4
        - 12 a_1 a_3^2 a_4 - 12 a_2 a_3^2 a_4 - 12 a_3^3 a_4 
        + 6 a_1^2 a_4^2 + 12 a_1 a_2 a_4^2 + 6 a_2^2 a_4^2
        - 12 a_1 a_3 a_4^2 - 12 a_2 a_3 a_4^2 \\
    &\qquad
        + 6 a_3^2 a_4^2 - 4 a_1 a_4^3 - 4 a_2 a_4^3 - 12 a_3 a_4^3 + a_4^4 
        - 4 a_1^3 t - 12 a_1^2 a_2 t - 12 a_1 a_2^2 t - 4 a_2^3 t
        + 12 a_1^2 a_3 t + 24 a_1 a_2 a_3 t \\
    &\qquad
        + 12 a_2^2 a_3 t - 12 a_1 a_3^2 t - 12 a_2 a_3^2 t + 4 a_3^3 t 
        + 12 a_1^2 a_4 t + 24 a_1 a_2 a_4 t + 12 a_2^2 a_4 t
        + 72 a_1 a_3 a_4 t + 72 a_2 a_3 a_4 t \\
    &\qquad
        + 12 a_3^2 a_4 t - 12 a_1 a_4^2 t - 12 a_2 a_4^2 t
        + 12 a_3 a_4^2 t + 4 a_4^3 t 
        + 6 a_1^2 t^2 + 12 a_1 a_2 t^2 + 6 a_2^2 t^2
        - 12 a_1 a_3 t^2 - 12 a_2 a_3 t^2 + 6 a_3^2 t^2 \\
    &\qquad
        - 12 a_1 a_4 t^2 - 12 a_2 a_4 t^2 - 36 a_3 a_4 t^2 + 6 a_4^2 t^2 
        - 4 a_1 t^3 - 4 a_2 t^3 + 4 a_3 t^3 + 4 a_4 t^3 + t^4
      \Bigr),\\[4pt]
g_{14}^{(0,4)} &= \frac{1}{192 a_1 a_2 a_3 a_4}\Bigl(
        3 a_1^4 - 4 a_1^3 a_2 + 18 a_1^2 a_2^2 - 4 a_1 a_2^3 + 3 a_2^4 
        - 4 a_1^3 a_3 - 12 a_1^2 a_2 a_3 - 12 a_1 a_2^2 a_3 - 4 a_2^3 a_3 \\
    &\qquad
        + 18 a_1^2 a_3^2 - 12 a_1 a_2 a_3^2 + 18 a_2^2 a_3^2
        - 4 a_1 a_3^3 - 4 a_2 a_3^3 + 3 a_3^4 
        - 4 a_1^3 a_4 - 12 a_1^2 a_2 a_4 - 12 a_1 a_2^2 a_4 - 4 a_2^3 a_4 \\
    &\qquad
        - 12 a_1^2 a_3 a_4 + 72 a_1 a_2 a_3 a_4 - 12 a_2^2 a_3 a_4
        - 12 a_1 a_3^2 a_4 - 12 a_2 a_3^2 a_4 - 4 a_3^3 a_4 
        + 18 a_1^2 a_4^2 - 12 a_1 a_2 a_4^2 + 18 a_2^2 a_4^2 \\
    &\qquad
        - 4 a_1 a_4^3 - 4 a_2 a_4^3 - 4 a_3 a_4^3 + 3 a_4^4 
        - 4 a_1^3 t - 12 a_1^2 a_2 t - 12 a_1 a_2^2 t - 4 a_2^3 t 
        - 12 a_1 a_3 a_4^2 - 12 a_2 a_3 a_4^2 + 18 a_3^2 a_4^2 \\
    &\qquad
        - 12 a_1^2 a_3 t + 72 a_1 a_2 a_3 t - 12 a_2^2 a_3 t
        - 12 a_1 a_3^2 t - 12 a_2 a_3^2 t - 4 a_3^3 t 
        - 12 a_1^2 a_4 t + 72 a_1 a_2 a_4 t - 12 a_2^2 a_4 t
        + 72 a_1 a_3 a_4 t \\
    &\qquad
        - 12 a_3^2 a_4 t - 12 a_1 a_4^2 t - 12 a_2 a_4^2 t
        - 12 a_3 a_4^2 t - 4 a_4^3 t 
        + 18 a_1^2 t^2 - 12 a_1 a_2 t^2 + 18 a_2^2 t^2
        - 12 a_1 a_3 t^2 - 12 a_2 a_3 t^2 + 18 a_3^2 t^2 \\
    &\qquad
        - 12 a_1 a_4 t^2 - 12 a_2 a_4 t^2 - 12 a_3 a_4 t^2 + 18 a_4^2 t^2 
        - 4 a_1 t^3 - 4 a_2 t^3 - 4 a_3 t^3 - 4 a_4 t^3 + 72 a_2 a_3 a_4 t  + 3 t^4
      \Bigr).
\end{align*}
\end{appendixformulas}

%===================================================================

\subsection{Moment formulas: dimension two}
\label{appendix:2Dmoment}
%---------------------------------------------------------

The ranges of $(a,t)$ for which the moment formulas $f_k^{(M,d)}, g_k^{(M,d)}$ are valid agree with the ranges for which the corresponding volume formulas 
$f_k^{(0,d)}, g_k^{(0,d)}$ are valid. These ranges are given in \Cref{table:2d_slice_formulas}.

\subsubsection{Slices}
\label{appendix:2Dslicemoment}

\noindent $M=1$:
\begin{appendixformulas}
\begin{align*}
f_1^{(1,2)} &= \frac{t}{2a_1^2}, \\
f_2^{(1,2)} &= -\frac{(a_1 + a_2 - t)(a_1 - a_2 - t)}{8a_1 a_2^2}
      + \frac{(a_1 + a_2 - t)(a_1 - a_2 + t)}{8a_1^2 a_2}.
\end{align*}
\end{appendixformulas}

\medskip
\noindent $M=2$:
\begin{appendixformulas}
\begin{align*}
f_1^{(2,2)} &= \frac{1}{12a_1} + \frac{a_2^2 + 3t^2}{12a_1^3}, \\
f_2^{(2,2)} &= \frac{(a_1^2 - a_1 a_2 + a_2^2 - 2a_1 t + a_2 t + t^2)(a_1 + a_2 - t)}{24a_1 a_2^3}
     + \frac{(a_1^2 - a_1 a_2 + a_2^2 + a_1 t - 2a_2 t + t^2)(a_1 + a_2 - t)}{24a_1^3 a_2}.
\end{align*}
\end{appendixformulas}

\medskip
\noindent $M=3$:
\begin{appendixformulas}
 \begin{align*}
f_1^{(3,2)} &= \frac{(a_2^2 + t^2)t}{8a_1^4}, \\
f_2^{(3,2)} &= -\frac{(a_1^2 + a_2^2 - 2a_1 t + t^2)(a_1 + a_2 - t)(a_1 - a_2 - t)}{64a_1 a_2^4}
      + \frac{(a_1^2 + a_2^2 - 2a_2 t + t^2)(a_1 + a_2 - t)(a_1 - a_2 + t)}{64a_1^4 a_2}.
\end{align*}
\end{appendixformulas}

\medskip
\noindent $M=4$:
\begin{appendixformulas}
\begin{align*}
f_1^{(4,2)} &= 
    \frac{a_1^4+a_2^4 + 10a_2^2 t^2 + 5t^4}{80a_1^5}, \\[6pt]
f_2^{(4,2)} &= \frac{(a_1 + a_2 - t)}{160}
      \left(
      \frac{
        \begin{aligned}
        &(a_1^4 - a_1^3 a_2 + a_1^2 a_2^2 - a_1 a_2^3 + a_2^4 
        - 4a_1^3 t + 3a_1^2 a_2 t - 2a_1 a_2^2 t + a_2^3 t \\
        &\quad + 6a_1^2 t^2 - 3a_1 a_2 t^2 + a_2^2 t^2 
        - 4a_1 t^3 + a_2 t^3 + t^4)
        \end{aligned}
      }{a_1 a_2^5} \right.
      \\[6pt]
      &\quad +
      \left. \frac{
        \begin{aligned}
        &(a_1^4 - a_1^3 a_2 + a_1^2 a_2^2 - a_1 a_2^3 + a_2^4
        + a_1^3 t - 2a_1^2 a_2 t + 3a_1 a_2^2 t - 4a_2^3 t \\
        &\quad + a_1^2 t^2 - 3a_1 a_2 t^2 + 6a_2^2 t^2 
        + a_1 t^3 - 4a_2 t^3 + t^4)
        \end{aligned}
      }{a_1^5 a_2}
      \right).
\end{align*}
\end{appendixformulas}

\subsubsection{Slabs}
\label{appendix:2Dslabmoment}
Observe that since the slabs are centrally symmetric, all odd moments of slabs are $0$.

\medskip
\noindent $M=1$: 
$\quad 0$.

\medskip
\noindent $M=2$:
\begin{appendixformulas}
\begin{align*}
g_1^{(2,2)} &= 
    \frac{t}{12a_1}
    + \frac{(3a_2^2 + t^2)t}{48a_1^3}
    + \frac{(a_2^2 + 3t^2)t}{48a_1^3}, \\[6pt]
g_2^{(2,2)} &= 
    \frac{a_1 - a_2 + t}{32a_1}
    - \frac{a_1 - a_2 - t}{32a_2} \\[4pt]
&\quad
    + \frac{(a_1^2 + a_1 a_2 + a_2^2 - a_1 t - 2a_2 t + t^2)(a_1 - a_2 + t)}{96a_1^3}
    - \frac{(a_1^2 + a_1 a_2 + a_2^2 - 2a_1 t - a_2 t + t^2)(a_1 - a_2 - t)}{96a_2^3} \\[4pt]
&\quad
    + \frac{(a_1^2 - a_1 a_2 + a_2^2 - 2a_1 t + a_2 t + t^2)(a_1 + a_2 - t)t}{96a_1 a_2^3}
    + \frac{(a_1^2 - a_1 a_2 + a_2^2 + a_1 t - 2a_2 t + t^2)(a_1 + a_2 - t)t}{96a_1^3 a_2}.
\end{align*}
\end{appendixformulas}

\medskip
\noindent $M=3$:
$\quad 0$.

\medskip
\noindent $M=4$:
\begin{appendixformulas}
\begin{align*}
g_1^{(4,2)} &= 
    \frac{t}{80a_1}
    + \frac{(5a_2^4 + 10a_2^2 t^2 + t^4)t}{480a_1^5}
    + \frac{(a_2^4 + 10a_2^2 t^2 + 5t^4)t}{480a_1^5}, \\[6pt]
g_2^{(4,2)} &= 
    \frac{a_1 - a_2 + t}{192a_1}
    - \frac{a_1 - a_2 - t}{192a_2} \\[2pt]
&\quad
    + \frac{
      \begin{aligned}
      &(a_1^4 + a_1^3 a_2 + a_1^2 a_2^2 + a_1 a_2^3 + a_2^4 
      - a_1^3 t - 2a_1^2 a_2 t - 3a_1 a_2^2 t - 4a_2^3 t \\
      &\quad + a_1^2 t^2 + 3a_1 a_2 t^2 + 6a_2^2 t^2 
      - a_1 t^3 - 4a_2 t^3 + t^4)(a_1 - a_2 + t)
      \end{aligned}
    }{960a_1^5} \\[2pt]
&\quad
    - \frac{
      \begin{aligned}
      &(a_1^4 + a_1^3 a_2 + a_1^2 a_2^2 + a_1 a_2^3 + a_2^4 
      - 4a_1^3 t - 3a_1^2 a_2 t - 2a_1 a_2^2 t - a_2^3 t \\
      &\quad + 6a_1^2 t^2 + 3a_1 a_2 t^2 + a_2^2 t^2 
      - 4a_1 t^3 - a_2 t^3 + t^4)(a_1 - a_2 - t)
      \end{aligned}
    }{960a_2^5} \\[2pt]
&\quad
    + \frac{
      \begin{aligned}
      &(a_1^4 - a_1^3 a_2 + a_1^2 a_2^2 - a_1 a_2^3 + a_2^4 
      - 4a_1^3 t + 3a_1^2 a_2 t - 2a_1 a_2^2 t + a_2^3 t \\
      &\quad + 6a_1^2 t^2 - 3a_1 a_2 t^2 + a_2^2 t^2 
      - 4a_1 t^3 + a_2 t^3 + t^4)(a_1 + a_2 - t)t
      \end{aligned}
    }{960a_1 a_2^5} \\[2pt]
&\quad
    + \frac{
      \begin{aligned}
      &(a_1^4 - a_1^3 a_2 + a_1^2 a_2^2 - a_1 a_2^3 + a_2^4 
      + a_1^3 t - 2a_1^2 a_2 t + 3a_1 a_2^2 t - 4a_2^3 t \\
      &\quad + a_1^2 t^2 - 3a_1 a_2 t^2 + 6a_2^2 t^2 
      + a_1 t^3 - 4a_2 t^3 + t^4)(a_1 + a_2 - t)t
      \end{aligned}
    }{960a_1^5 a_2}.
\end{align*}
\end{appendixformulas}

%--------------------------------------------------------------------
\subsection{Moment formulas: dimension three}
\label{appendix:3Dmoment}

\subsubsection{Slice}
\label{appendix:3dmomentslice}

\noindent $M=1$:
\begin{appendixformulas}
\begin{align*}
f_1^{(1,3)} &= \frac{t}{2 a_1^{2}},
\\[1.2ex]
f_2^{(1,3)}
&= -\frac{(a_1 a_2 + (a_1 + a_2)a_3)\,t^{3}}{24\,a_1^{2}a_2^{2}a_3^{2}},
\\[-0.3ex]
&\phantom{={}}
   -\frac{
      3\bigl(
        a_1^{3} a_2
        - 2a_1^{2} a_2^{2}
        + a_1 a_2^{3}
        + (a_1 + a_2)a_3^{3}
        - (2a_1^{2}+a_1 a_2+2a_2^{2})a_3^{2}
        + (a_1^{3}-a_1^{2} a_2-a_1 a_2^{2}+a_2^{3})a_3
      \bigr)t
}{24\,a_1^{2}a_2^{2}a_3^{2}},
\\[1.3ex]
f_3^{(1,3)}
&= \frac{
      a_1^{4} a_2
      - 3a_1^{3} a_2^{2}
      + 3a_1^{2} a_2^{3}
      - a_1 a_2^{4}
      - (a_1+a_2)a_3^{4}
      + (3a_1^{2}+2a_1 a_2-3a_2^{2})a_3^{3}
}{48\,a_1^{2}a_2^{2}a_3^{2}}
\\[-0.3ex]
&\phantom{={}}
   -\frac{
      (a_1 a_2 + (a_1+a_2)a_3)t^{3}
}{48\,a_1^{2}a_2^{2}a_3^{2}}
\\[-0.3ex]
&\phantom{={}}
   -\frac{
      3(a_1^{3}-2a_1 a_2^{2}+a_2^{3})a_3^{2}
      + (a_1^{4}-2a_1^{3}a_2+2a_1 a_2^{3}-a_2^{4})a_3
}{48\,a_1^{2}a_2^{2}a_3^{2}}
\\[-0.3ex]
&\phantom{={}}
   +\frac{
      3\bigl(
        a_1^{2}a_2
        - a_1 a_2^{2}
        - (a_1+a_2)a_3^{2}
        + (a_1^{2}-a_2^{2})a_3
      \bigr)t^{2}
}{48\,a_1^{2}a_2^{2}a_3^{2}}
\\[-0.3ex]
&\phantom{={}}
   -\frac{
      3\bigl(
        a_1^{3} a_2
        - 2a_1^{2}a_2^{2}
        + a_1 a_2^{3}
        + (a_1+a_2)a_3^{3}
        - (2a_1^{2}+a_1 a_2+6a_2^{2})a_3^{2}
        + (a_1^{3}-a_1^{2}a_2-a_1 a_2^{2}+a_2^{3})a_3
      \bigr)t
}{48\,a_1^{2}a_2^{2}a_3^{2}},
\\[1.3ex]
f_4^{(1,3)}
&= -\frac{
      3a_1^{3}
      - 3a_1^{2} a_2
      - 3a_1 a_2^{2}
      + 3a_2^{3}
      - 2a_1 a_2 a_3
      + (a_1+a_2)a_3^{2}
}{24\,a_1^{2}a_2^{2}}
\\[-0.3ex]
&\phantom{={}}
   -\frac{
      3(a_1+a_2)t^{2}
      - 6(a_1^{2}+a_2^{2})t
}{24\,a_1^{2}a_2^{2}},
\\[1.3ex]
f_5^{(1,3)}
&= -\frac{
      a_1^{4}a_2
      + 3a_1^{3}a_2^{2}
      + 3a_1^{2}a_2^{3}
      + a_1 a_2^{4}
      + (a_1+a_2)a_3^{4}
      + (3a_1^{2}-2a_1 a_2+3a_2^{2})a_3^{3}
}{48\,a_1^{2}a_2^{2}a_3^{2}}
\\[-0.3ex]
&\phantom{={}}
   -\frac{
      (a_1 a_2+(a_1+a_2)a_3)t^{3}
}{48\,a_1^{2}a_2^{2}a_3^{2}}
\\[-0.3ex]
&\phantom{={}}
   -\frac{
      3(a_1^{3}-2a_1^{2}a_2-2a_1 a_2^{2}+a_2^{3})a_3^{2}
      + (a_1^{4}-2a_1^{3}a_2-6a_1^{2}a_2^{2}-2a_1 a_2^{3}+a_2^{4})a_3
}{48\,a_1^{2}a_2^{2}a_3^{2}}
\\[-0.3ex]
&\phantom{={}}
   -\frac{
      3\bigl(
        a_1^{2}a_2
        + a_1 a_2^{2}
        + (a_1+a_2)a_3^{2}
        + (a_1^{2}+a_2^{2})a_3
      \bigr)t^{2}
}{48\,a_1^{2}a_2^{2}a_3^{2}}
\\[-0.3ex]
&\phantom{={}}
   +\frac{
      3\bigl(
        a_1^{3} a_2
        + 2a_1^{2}a_2^{2}
        + a_1 a_2^{3}
        + (a_1+a_2)a_3^{3}
        + (2a_1^{2}-a_1 a_2+2a_2^{2})a_3^{2}
        + (a_1^{3}-a_1^{2}a_2-a_1 a_2^{2}+a_2^{3})a_3
      \bigr)t
}{48\,a_1^{2}a_2^{2}a_3^{2}}.
\end{align*}
\end{appendixformulas}

\medskip
\noindent $M=2$:
\begin{appendixformulas}
\begin{align*}
f_1^{(2,3)}
  &= \frac{2 a_1^{2} + a_2^{2} + a_3^{2} + 3 t^{2}}
          {12 a_1^{3}},
          \\[-0.3ex]
%
%======================  f2  ======================%
f_2^{(2,3)}
&= -\frac{
      a_1^{6} a_2^{2}
      - 4 a_1^{5} a_2^{3}
      + 6 a_1^{4} a_2^{4}
      - 4 a_1^{3} a_2^{5}
      + a_1^{2} a_2^{6}
      + (a_1^{2} + a_2^{2}) a_3^{6}
      - 4 (a_1^{3} + a_2^{3}) a_3^{5}
}{96\,a_1^{3} a_2^{3} a_3^{3}}
\\[-0.3ex]
&\phantom{={}}
   -\frac{
      3 (2 a_1^{4} + a_1^{2} a_2^{2} + 2 a_2^{4}) a_3^{4}
      - 4 (a_1^{5} + 2 a_1^{3} a_2^{2}
           + 2 a_1^{2} a_2^{3} + a_2^{5}) a_3^{3}
}{96\,a_1^{3} a_2^{3} a_3^{3}}
\\[-0.3ex]
&\phantom{={}}
   -\frac{
      (a_1^{6} + 3 a_1^{4} a_2^{2} - 8 a_1^{3} a_2^{3}
       + 3 a_1^{2} a_2^{4} + a_2^{6}) a_3^{2}+(a_1^{2} a_2^{2}
       + (a_1^{2} + a_2^{2}) a_3^{2}) t^{4}
}{96\,a_1^{3} a_2^{3} a_3^{3}}
\\[-0.3ex]
&\phantom{={}}
   -\frac{
      6\bigl(
        a_1^{4} a_2^{2}
        - 2 a_1^{3} a_2^{3}
        + a_1^{2} a_2^{4}
        + (a_1^{2} + a_2^{2}) a_3^{4}
        - 2 (a_1^{3} + a_2^{3}) a_3^{3}
        + (a_1^{4} + a_2^{4}) a_3^{2}
      \bigr) t^{2}
}{96\,a_1^{3} a_2^{3} a_3^{3}},
\\[2.0ex]
%
%======================  f3  ======================%
f_3^{(2,3)}
&= -\frac{
      a_1^{6} a_2^{2}
      - 4 a_1^{5} a_2^{3}
      + 6 a_1^{4} a_2^{4}
      - 4 a_1^{3} a_2^{5}
      + a_1^{2} a_2^{6}
      + (a_1^{2} + a_2^{2}) a_3^{6}
      - 4 (a_1^{3} + 3 a_2^{3}) a_3^{5}
      + (a_1^{2} a_2^{2} + (a_1^{2} + a_2^{2}) a_3^{2}) t^{4}
}{192\,a_1^{3} a_2^{3} a_3^{3}}
\\[-0.3ex]
&\phantom{={}}
   -\frac{
      3 (2 a_1^{4} + a_1^{2} a_2^{2} + 2 a_2^{4}) a_3^{4}
      - 4 (a_1^{5} + 2 a_1^{3} a_2^{2}
      + 6 a_1^{2} a_2^{3} + 3 a_2^{5}) a_3^{3}
    (a_1^{6} + 3 a_1^{4} a_2^{2} - 8 a_1^{3} a_2^{3}
       + 3 a_1^{2} a_2^{4} + a_2^{6}) a_3^{2}
}{192\,a_1^{3} a_2^{3} a_3^{3}}
\\[-0.3ex]
&\phantom{={}}
   -\frac{
      4\bigl(
        a_1^{3} a_2^{2}
        - a_1^{2} a_2^{3}
        - (a_1^{2} + a_2^{2}) a_3^{3}
        + (a_1^{3} - a_2^{3}) a_3^{2}
      \bigr) t^{3} + 4 t (a_1^{5} + a_1^{3} a_2^{2}
           - a_1^{2} a_2^{3}
           - a_2^{5}) a_3^{2}
}{192\,a_1^{3} a_2^{3} a_3^{3}}
\\[-0.3ex]
&\phantom{={}}
   -\frac{
      6\bigl(
        a_1^{4} a_2^{2}
        - 2 a_1^{3} a_2^{3}
        + a_1^{2} a_2^{4}
        + (a_1^{2} + a_2^{2}) a_3^{4}
        - 2 (a_1^{3} + 3 a_2^{3}) a_3^{3}
        + (a_1^{4} + a_2^{4}) a_3^{2}
      \bigr) t^{2}
}{192\,a_1^{3} a_2^{3} a_3^{3}}
\\[-0.3ex]
&\phantom{={}}
   -\frac{
      4\bigl(
        a_1^{5} a_2^{2}
        - 3 a_1^{4} a_2^{3}
        + 3 a_1^{3} a_2^{4}
        - a_1^{2} a_2^{5}
        - (a_1^{2} + a_2^{2}) a_3^{5}
        + 3 (a_1^{3} - a_2^{3}) a_3^{4}
        - (3 a_1^{4} + a_1^{2} a_2^{2} + 3 a_2^{4}) a_3^{3}
      \bigr) t
}{192\,a_1^{3} a_2^{3} a_3^{3}},
\\[2.0ex]
%
%======================  f4  ======================%
f_4^{(2,3)}
&= \frac{
      a_1^{5}
      + 2 a_1^{3} a_2^{2}
      + 2 a_1^{2} a_2^{3}
      + a_2^{5}
      - (3 a_1^{4} + a_1^{2} a_2^{2}
       + 3 a_2^{4} + (a_1^{2} + a_2^{2}) a_3^{2}) t
      - (a_1^{2} + a_2^{2}) t^{3}
      + 3 (a_1^{3} + a_2^{3}) t^{2}
}{24\,a_1^{3} a_2^{3}},
\\[2.0ex]
%
%======================  f5  ======================%
f_5^{(2,3)}
&= \frac{
      a_1^{6} a_2^{2}
      + 4 a_1^{5} a_2^{3}
      + 6 a_1^{4} a_2^{4}
      + 4 a_1^{3} a_2^{5}
      + a_1^{2} a_2^{6}
      + 3 (2 a_1^{4} + a_1^{2} a_2^{2} + 2 a_2^{4}) a_3^{4}
      + 4 (a_1^{5} + 2 a_1^{3} a_2^{2}
           + 2 a_1^{2} a_2^{3} + a_2^{5}) a_3^{3}
}{192\,a_1^{3} a_2^{3} a_3^{3}},
\\[-0.3ex]
&\phantom{={}}
   +\frac{
      (a_1^{2} + a_2^{2}) a_3^{6}
      + 4 (a_1^{3} + a_2^{3}) a_3^{5}
      + (a_1^{6} + 3 a_1^{4} a_2^{2}
       + 8 a_1^{3} a_2^{3}
       + 3 a_1^{2} a_2^{4}
       + a_2^{6}) a_3^{2}
       + (a_1^{2} a_2^{2}
       + (a_1^{2} + a_2^{2}) a_3^{2}) t^{4}
}{192\,a_1^{3} a_2^{3} a_3^{3}}
\\[-0.3ex]
&\phantom{={}}
   -\frac{
      4 (a_1^{3} a_2^{2}
         + a_1^{2} a_2^{3}
         + (a_1^{2} + a_2^{2}) a_3^{3}
         + (a_1^{3} + a_2^{3}) a_3^{2}) t^{3}
         + 4 t (a_1^{5} + a_1^{3} a_2^{2}
           + a_1^{2} a_2^{3}
           + a_2^{5}) a_3^{2}
}{192\,a_1^{3} a_2^{3} a_3^{3}}
\\[-0.3ex]
&\phantom{={}}
   +\frac{
      6\bigl(
        a_1^{4} a_2^{2}
        + 2 a_1^{3} a_2^{3}
        + a_1^{2} a_2^{4}
        + (a_1^{2} + a_2^{2}) a_3^{4}
        + 2 (a_1^{3} + a_2^{3}) a_3^{3}
        + (a_1^{4} + a_2^{4}) a_3^{2}
      \bigr) t^{2}
}{192\,a_1^{3} a_2^{3} a_3^{3}}
\\[-0.3ex]
&\phantom{={}}
   -\frac{
      4\bigl(
        a_1^{5} a_2^{2}
        + 3 a_1^{4} a_2^{3}
        + 3 a_1^{3} a_2^{4}
        + a_1^{2} a_2^{5}
        + (a_1^{2} + a_2^{2}) a_3^{5}
        + 3 (a_1^{3} + a_2^{3}) a_3^{4}
        + (3 a_1^{4} + a_1^{2} a_2^{2} + 3 a_2^{4}) a_3^{3}
      \bigr) t
}{192\,a_1^{3} a_2^{3} a_3^{3}}.
\end{align*}
\end{appendixformulas}

%%%%%%%%%%%%%%%%%%%%%%%%%%%%%%%%%%%%
% Moment = 3
%%%%%%%%%%%%%%%%%%%%%%%%%%%%%%%%%%%%

\medskip
\noindent $M=3$:
\begin{appendixformulas}
\begin{align*}
f_1^{(3,3)}
&= \frac{t^{3} + (a_2^{2} + a_3^{2})\,t}{8\,a_1^{4}},
\\[-0.2ex]
f_2^{(3,3)}
&= -\frac{1}{320\,a_1^{4} a_2^{4} a_3^{4}}\Bigl[
      \bigl(a_1^{3} a_2^{3} + (a_1^{3} + a_2^{3})a_3^{3}\bigr)\,t^{5} + 10\bigl(
        a_1^{5} a_2^{3}
        - 2a_1^{4} a_2^{4}
        + a_1^{3} a_2^{5}
        + (a_1^{3} + a_2^{3})a_3^{5}
\\[-0.2ex]
&\phantom{={}-\frac{1}{320\,a_1^{4} a_2^{4} a_3^{4}}\Bigl[\ + 10\bigl(}
        - 2(a_1^{4} + a_2^{4})a_3^{4}
        + (a_1^{5} + a_2^{5})a_3^{3}
      \bigr)t^{3} + 5\bigl(
        a_1^{7} a_2^{3}
        - 4a_1^{6} a_2^{4}
        + 6a_1^{5} a_2^{5}
        - 4a_1^{4} a_2^{6}
        + a_1^{3} a_2^{7}
\\[-0.2ex]
&\phantom{={}-\frac{1}{320\,a_1^{4} a_2^{4} a_3^{4}}\Bigl[\ + 5\bigl(}
        + (a_1^{3} + a_2^{3})a_3^{7}
        - 4(a_1^{4} + a_2^{4})a_3^{6}
        + 6(a_1^{5} + a_2^{5})a_3^{5}
\\[-0.2ex]
&\phantom{={}-\frac{1}{320\,a_1^{4} a_2^{4} a_3^{4}}\Bigl[\ + 5\bigl(}
        - \bigl(4a_1^{6} + a_1^{3} a_2^{3} + 4a_2^{6}\bigr)a_3^{4}
        + (a_1^{7} - a_1^{4} a_2^{3} - a_1^{3} a_2^{4} + a_2^{7})a_3^{3}
      \bigr)t
   \Bigr],
\\[-0.2ex]
f_3^{(3,3)}
&= \frac{1}{640\,a_1^{4} a_2^{4} a_3^{4}}\Bigl[
      a_1^{8} a_2^{3}
      - 5a_1^{7} a_2^{4}
      + 10a_1^{6} a_2^{5}
      - 10a_1^{5} a_2^{6}
      + 5a_1^{4} a_2^{7}
      - a_1^{3} a_2^{8}
        - (a_1^{3} + a_2^{3})a_3^{8}
      + 5(a_1^{4} - a_2^{4})a_3^{7}
\\[-0.2ex]
&\phantom{=\frac{1}{640\,a_1^{4} a_2^{4} a_3^{4}}\Bigl[}
      + 2\bigl(5a_1^{6} + 2a_1^{3} a_2^{3} - 5a_2^{6}\bigr)a_3^{5}
      - \bigl(a_1^{3} a_2^{3} + (a_1^{3} + a_2^{3})a_3^{3}\bigr)t^{5}
    - 5\bigl(a_1^{7} - 2a_1^{3} a_2^{4} + a_2^{7}\bigr)a_3^{4}
\\[-0.2ex]
&\phantom{=\frac{1}{640\,a_1^{4} a_2^{4} a_3^{4}}\Bigl[}
      + 5\bigl(
        a_1^{4} a_2^{3}
        - a_1^{3} a_2^{4}
        - (a_1^{3} + a_2^{3})a_3^{4}
        + (a_1^{4} - a_2^{4})a_3^{3}
      \bigr)t^{4}
        + (a_1^{8} - 4a_1^{5} a_2^{3} + 4a_1^{3} a_2^{5} - a_2^{8})a_3^{3}
\\[-0.2ex]
&\phantom{=\frac{1}{640\,a_1^{4} a_2^{4} a_3^{4}}\Bigl[}
      - 10\bigl(
        a_1^{5} a_2^{3}
        - 2a_1^{4} a_2^{4}
        + a_1^{3} a_2^{5}
        + (a_1^{3} + a_2^{3})a_3^{5}
        - 2(a_1^{4} + 3a_2^{4})a_3^{4}
        + (a_1^{5} + a_2^{5})a_3^{3}
      \bigr)t^{3}
\\[-0.2ex]
&\phantom{=\frac{1}{640\,a_1^{4} a_2^{4} a_3^{4}}\Bigl[}
      + 10\bigl(
        a_1^{6} a_2^{3}
        - 3a_1^{5} a_2^{4}
        + 3a_1^{4} a_2^{5}
        - a_1^{3} a_2^{6}
        - (a_1^{3} + a_2^{3})a_3^{6}
        + 3(a_1^{4} - a_2^{4})a_3^{5}
        - 3(a_1^{5} + a_2^{5})a_3^{4}
\\[-0.2ex]
&\phantom{=\frac{1}{640\,a_1^{4} a_2^{4} a_3^{4}}\Bigl[\ + 10\bigl(}
        + (a_1^{6} - a_2^{6})a_3^{3}
      \bigr)t^{2}
        - 5\bigl(
        a_1^{7} a_2^{3}
        - 4a_1^{6} a_2^{4}
        + 6a_1^{5} a_2^{5}
        - 4a_1^{4} a_2^{6}
        + a_1^{3} a_2^{7}
\\[-0.2ex]
&\phantom{=\frac{1}{640\,a_1^{4} a_2^{4} a_3^{4}}\Bigl[\ - 5\bigl(}
        + (a_1^{3} + a_2^{3})a_3^{7}
        - 4(a_1^{4} + 3a_2^{4})a_3^{6}
        + 6(a_1^{5} + a_2^{5})a_3^{5}
    - \bigl(4a_1^{6} + a_1^{3} a_2^{3} + 12a_2^{6}\bigr)a_3^{4}
\\[-0.2ex]
&\phantom{=\frac{1}{640\,a_1^{4} a_2^{4} a_3^{4}}\Bigl[\ - 5\bigl(}
    + (a_1^{7} - a_1^{4} a_2^{3} - a_1^{3} a_2^{4} + a_2^{7})a_3^{3}
      \bigr)t
   \Bigr],
\\[-0.2ex]
f_4^{(3,3)}
&= -\frac{1}{320\,a_1^{4} a_2^{4}}\Bigl[
      5a_1^{7}
      - 5a_1^{4} a_2^{3}
      - 5a_1^{3} a_2^{4}
      + 5a_2^{7}
      - 4a_1^{3} a_2^{3} a_3
         + (a_1^{3} + a_2^{3})a_3^{4}
      + 5(a_1^{3} + a_2^{3})t^{4}
      - 20(a_1^{4} + a_2^{4})t^{3}
\\[-0.2ex]
&\phantom{=-\frac{1}{320\,a_1^{4} a_2^{4}}\Bigl[}
      + 10(a_1^{5} + a_2^{5})a_3^{2}
      + 10\bigl(3a_1^{5} + 3a_2^{5} + (a_1^{3} + a_2^{3})a_3^{2}\bigr)t^{2}
     - 20\bigl(a_1^{6} + a_2^{6} + (a_1^{4} + a_2^{4})a_3^{2}\bigr)t
   \Bigr]
,
\\[-0.2ex]
f_5^{(3,3)}
&= -\frac{1}{640\,a_1^{4} a_2^{4} a_3^{4}}\Bigl[
      a_1^{8} a_2^{3}
      + 5a_1^{7} a_2^{4}
      + 10a_1^{6} a_2^{5}
      + 10a_1^{5} a_2^{6}
      + 5a_1^{4} a_2^{7}
      + a_1^{3} a_2^{8}
    + (a_1^{3} + a_2^{3})a_3^{8}
      + 5(a_1^{4} + a_2^{4})a_3^{7}
\\[-0.2ex]
&\phantom{=-\frac{1}{640\,a_1^{4} a_2^{4} a_3^{4}}\Bigl[}
    + 10(a_1^{5} + a_2^{5})a_3^{6}
      + 2\bigl(5a_1^{6} - 2a_1^{3} a_2^{3} + 5a_2^{6}\bigr)a_3^{5}
      - \bigl(a_1^{3} a_2^{3} + (a_1^{3} + a_2^{3})a_3^{3}\bigr)t^{5}
\\[-0.2ex]
&\phantom{=-\frac{1}{640\,a_1^{4} a_2^{4} a_3^{4}}\Bigl[}
      + 5\bigl(
        a_1^{4} a_2^{3}
        + a_1^{3} a_2^{4}
        + (a_1^{3} + a_2^{3})a_3^{4}
        + (a_1^{4} + a_2^{4})a_3^{3}
      \bigr)t^{4}
      + (a_1^{8} - 4a_1^{5} a_2^{3} - 10a_1^{4} a_2^{4} - 4a_1^{3} a_2^{5} + a_2^{8})a_3^{3}
\\[-0.2ex]
&\phantom{=-\frac{1}{640\,a_1^{4} a_2^{4} a_3^{4}}\Bigl[}
      - 10\bigl(
        a_1^{5} a_2^{3}
        + 2a_1^{4} a_2^{4}
        + a_1^{3} a_2^{5}
        + (a_1^{3} + a_2^{3})a_3^{5}
    + 2(a_1^{4} + a_2^{4})a_3^{4}
        + (a_1^{5} + a_2^{5})a_3^{3}
      \bigr)t^{3}
\\[-0.2ex]
&\phantom{=-\frac{1}{640\,a_1^{4} a_2^{4} a_3^{4}}\Bigl[}
      + 10\bigl(
        a_1^{6} a_2^{3}
        + 3a_1^{5} a_2^{4}
        + 3a_1^{4} a_2^{5}
        + a_1^{3} a_2^{6}
    + (a_1^{3} + a_2^{3})a_3^{6}
        + 3(a_1^{4} + a_2^{4})a_3^{5}
        + 3(a_1^{5} + a_2^{5})a_3^{4}
      \bigr)t^{2}
\\[-0.2ex]
&\phantom{=-\frac{1}{640\,a_1^{4} a_2^{4} a_3^{4}}\Bigl[}
      + 10\bigl(
    + (a_1^{6} + a_2^{6})a_3^{3}
      \bigr)t^{2}
       + 5\bigl(a_1^{7} - 2a_1^{4} a_2^{3} - 2a_1^{3} a_2^{4} + a_2^{7}\bigr)a_3^{4}
\\[-0.2ex]
&\phantom{=-\frac{1}{640\,a_1^{4} a_2^{4} a_3^{4}}\Bigl[}
      - 5\bigl(
        a_1^{7} a_2^{3}
        + 4a_1^{6} a_2^{4}
        + 6a_1^{5} a_2^{5}
        + 4a_1^{4} a_2^{6}
        + a_1^{3} a_2^{7}
    + (a_1^{3} + a_2^{3})a_3^{7}
        + 4(a_1^{4} + a_2^{4})a_3^{6}
        + 6(a_1^{5} + a_2^{5})a_3^{5}
\\[-0.2ex]
&\phantom{=-\frac{1}{640\,a_1^{4} a_2^{4} a_3^{4}}\Bigl[\ - 5\bigl(}
        + \bigl(4a_1^{6} - a_1^{3} a_2^{3} + 4a_2^{6}\bigr)a_3^{4}
        + (a_1^{7} - a_1^{4} a_2^{3} - a_1^{3} a_2^{4} + a_2^{7})a_3^{3}
      \bigr)t
   \Bigr].
\end{align*}
\end{appendixformulas}

%%%%%%%%%%%%%%%%%%%%%%%%%%%%%%%%%%%%
% Moment = 4
%%%%%%%%%%%%%%%%%%%%%%%%%%%%%%%%%%%%

\medskip
\noindent $M=4$:
\begin{appendixformulas}
\begin{align*}
f_1^{(4,3)}
&= \frac{1}{240\,a_1^{5}}
\Bigl(
   6a_1^{4}
   + 3a_2^{4}
   + 10a_2^{2}a_3^{2}
   + 3a_3^{4}
   + 15t^{4}
   + 30(a_2^{2}+a_3^{2})t^{2}
\Bigr)
,
\\[-0.2ex]
f_2^{(4,3)}
&= -\frac{1}{960\,a_1^{5}a_2^{5}a_3^{5}}
\Bigl(
   a_1^{10}a_2^{4}
   - 6a_1^{9}a_2^{5}
   + 15a_1^{8}a_2^{6}
   - 20a_1^{7}a_2^{7}
   + 15a_1^{6}a_2^{8}
   - 6a_1^{5}a_2^{9}
   + a_1^{4}a_2^{10}
\\ &\quad
   + (a_1^{4}+a_2^{4})a_3^{10}
   - 6(a_1^{5}+a_2^{5})a_3^{9}
   + 15(a_1^{6}+a_2^{6})a_3^{8}
   - 20(a_1^{7}+a_2^{7})a_3^{7}
   + 5(3a_1^{8}+a_1^{4}a_2^{4}+3a_2^{8})a_3^{6}
\\ &\quad
   + (a_1^{4}a_2^{4}+(a_1^{4}+a_2^{4})a_3^{4})t^{6}
   - 6(a_1^{9}+2a_1^{5}a_2^{4}+2a_1^{4}a_2^{5}+a_2^{9})a_3^{5}+(a_1^{10}+5a_1^{6}a_2^{4}-12a_1^{5}a_2^{5}+5a_1^{4}a_2^{6}+a_2^{10})a_3^{4}
\\ &\quad
   + 15\bigl(
        a_1^{6}a_2^{4}
        - 2a_1^{5}a_2^{5}
        + a_1^{4}a_2^{6}
        + (a_1^{4}+a_2^{4})a_3^{6}
        - 2(a_1^{5}+a_2^{5})a_3^{5}
        + (a_1^{6}+a_2^{6})a_3^{4}
     \bigr)t^{4}
\\ &\quad
   + 15\bigl(
        a_1^{8}a_2^{4}
        - 4a_1^{7}a_2^{5}
        + 6a_1^{6}a_2^{6}
        - 4a_1^{5}a_2^{7}
        + a_1^{4}a_2^{8}
        + (a_1^{4}+a_2^{4})a_3^{8}
        - 4(a_1^{5}+a_2^{5})a_3^{7}
\\ &\qquad
        + 6(a_1^{6}+a_2^{6})a_3^{6}
        - 4(a_1^{7}+a_2^{7})a_3^{5}
        + (a_1^{8}+a_2^{8})a_3^{4}
     \bigr)t^{2}
\Bigr),
\\[-0.2ex]
f_3^{(4,3)}
&= -\frac{1}{1920\,a_1^{5}a_2^{5}a_3^{5}}
\Bigl[
   a_1^{10}a_2^{4}
   - 6a_1^{9}a_2^{5}
   + 15a_1^{8}a_2^{6}
   - 20a_1^{7}a_2^{7}
   + 15a_1^{6}a_2^{8}
   - 6a_1^{5}a_2^{9}
   + a_1^{4}a_2^{10}
\\ &\quad
   + (a_1^{4}+a_2^{4})a_3^{10}
   - 6(a_1^{5}+3a_2^{5})a_3^{9}
   + 15(a_1^{6}+a_2^{6})a_3^{8}
   - 20(a_1^{7}+3a_2^{7})a_3^{7}
\\ &\quad
   + 5(3a_1^{8}+a_1^{4}a_2^{4}+3a_2^{8})a_3^{6}
   + (a_1^{4}a_2^{4}+(a_1^{4}+a_2^{4})a_3^{4})t^{6}
   - 6(a_1^{5}a_2^{4}-a_1^{4}a_2^{5}-(a_1^{4}+a_2^{4})a_3^{5}+(a_1^{5}-a_2^{5})a_3^{4})t^{5}
\\ &\quad
   - 6(a_1^{9}+2a_1^{5}a_2^{4}+6a_1^{4}a_2^{5}+3a_2^{9})a_3^{5}
   + (a_1^{10}+5a_1^{6}a_2^{4}-12a_1^{5}a_2^{5}+5a_1^{4}a_2^{6}+a_2^{10})a_3^{4}
\\ &\quad
   + 15\bigl(
        a_1^{6}a_2^{4}
        - 2a_1^{5}a_2^{5}
        + a_1^{4}a_2^{6}
        + (a_1^{4}+a_2^{4})a_3^{6}
        - 2(a_1^{5}+3a_2^{5})a_3^{5}
        + (a_1^{6}+a_2^{6})a_3^{4}
     \bigr)t^{4}
\\ &\quad
   - 20\bigl(
        a_1^{7}a_2^{4}
        - 3a_1^{6}a_2^{5}
        + 3a_1^{5}a_2^{6}
        - a_1^{4}a_2^{7}
        - (a_1^{4}+a_2^{4})a_3^{7}
        + 3(a_1^{5}-a_2^{5})a_3^{6}
        - 3(a_1^{6}+a_2^{6})a_3^{5}
        + (a_1^{7}-a_2^{7})a_3^{4}
     \bigr)t^{3}
\\ &\quad
   + 15\bigl(
        a_1^{8}a_2^{4}
        - 4a_1^{7}a_2^{5}
        + 6a_1^{6}a_2^{6}
        - 4a_1^{5}a_2^{7}
        + a_1^{4}a_2^{8}
        + (a_1^{4}+a_2^{4})a_3^{8}
        - 4(a_1^{5}+3a_2^{5})a_3^{7}
\\ &\qquad
        + 6(a_1^{6}+a_2^{6})a_3^{6}
        - 4(a_1^{7}+3a_2^{7})a_3^{5}
        + (a_1^{8}+a_2^{8})a_3^{4}
     \bigr)t^{2}
   - 6\bigl(
        a_1^{9}a_2^{4}
        - 5a_1^{8}a_2^{5}
        + 10a_1^{7}a_2^{6}
        - 10a_1^{6}a_2^{7}
        + 5a_1^{5}a_2^{8}
        - a_1^{4}a_2^{9}
\\ &\qquad
        - (a_1^{4}+a_2^{4})a_3^{9}
        + 5(a_1^{5}-a_2^{5})a_3^{8}
        - 10(a_1^{6}+a_2^{6})a_3^{7}
        + 10(a_1^{7}-a_2^{7})a_3^{6}
        - (5a_1^{8}+a_1^{4}a_2^{4}+5a_2^{8})a_3^{5}
     \bigr)t
\\ &\qquad
     + (a_1^{9}+a_1^{5}a_2^{4}-a_1^{4}a_2^{5}-a_2^{9})a_3^{4}t
\Bigr],
\\[-0.2ex]
f_4^{(4,3)}
&= \frac{1}{480\,a_1^{5}a_2^{5}}
\Bigl(
   3a_1^{9}
   + 6a_1^{5}a_2^{4}
   + 6a_1^{4}a_2^{5}
   + 3a_2^{9}
\\ &\quad
   + 3(a_1^{5}+a_2^{5})a_3^{4}
   - 3(a_1^{4}+a_2^{4})t^{5}
   + 15(a_1^{5}+a_2^{5})t^{4}
\\ &\quad
   - 10\bigl(3a_1^{6}+3a_2^{6}+(a_1^{4}+a_2^{4})a_3^{2}\bigr)t^{3}
\\ &\quad
   + 10(a_1^{7}+a_2^{7})a_3^{2}
   + 30(a_1^{7}+a_2^{7}+(a_1^{5}+a_2^{5})a_3^{2})t^{2}
\\ &\quad
   - 3\bigl(
        5a_1^{8}
        + a_1^{4}a_2^{4}
        + 5a_2^{8}
        + (a_1^{4}+a_2^{4})a_3^{4}
        + 10(a_1^{6}+a_2^{6})a_3^{2}
     \bigr)t
\Bigr),
\\[-0.2ex]
f_5^{(4,3)}
&= \frac{1}{1920\,a_1^{5} a_2^{5} a_3^{5}}\Bigl(
   a_1^{10} a_2^{4}
   + 6a_1^{9} a_2^{5}
   + 15a_1^{8} a_2^{6}
   + 20a_1^{7} a_2^{7}
   + 15a_1^{6} a_2^{8}
   + 6a_1^{5} a_2^{9}
   + a_1^{4} a_2^{10}
\\ &\quad
   + (a_1^{4}+a_2^{4}) a_3^{10}
   + 6(a_1^{5}+a_2^{5}) a_3^{9}
   + 15(a_1^{6}+a_2^{6}) a_3^{8}
   + 20(a_1^{7}+a_2^{7}) a_3^{7}
   + 5(3a_1^{8}+a_1^{4}a_2^{4}+3a_2^{8}) a_3^{6}
\\ &\quad
   + (a_1^{4}a_2^{4} + (a_1^{4}+a_2^{4}) a_3^{4}) t^{6}
   + 6(a_1^{9} + 2a_1^{5}a_2^{4} + 2a_1^{4}a_2^{5} + a_2^{9}) a_3^{5}- 6\bigl(
       a_1^{5}a_2^{4}
       + a_1^{4}a_2^{5}
       + (a_1^{4}+a_2^{4}) a_3^{5}
       + (a_1^{5}+a_2^{5}) a_3^{4}
     \bigr) t^{5} 
\\ &\quad
   + (a_1^{10} + 5a_1^{6}a_2^{4} + 12a_1^{5}a_2^{5} + 5a_1^{4}a_2^{6} + a_2^{10}) a_3^{4} + 15 t^2 + (a_1^{8}+a_2^{8}) a_3^{4}
\\ &\quad
   + 15\bigl(
       a_1^{6}a_2^{4}
       + 2a_1^{5}a_2^{5}
       + a_1^{4}a_2^{6}
       + (a_1^{4}+a_2^{4}) a_3^{6}
       + 2(a_1^{5}+a_2^{5}) a_3^{5}
       + (a_1^{6}+a_2^{6}) a_3^{4}
     \bigr) t^{4}
\\ &\quad
   - 20\bigl(
       a_1^{7}a_2^{4}
       + 3a_1^{6}a_2^{5}
       + 3a_1^{5}a_2^{6}
       + a_1^{4}a_2^{7}
       + (a_1^{4}+a_2^{4}) a_3^{7}
       + 3(a_1^{5}+a_2^{5}) a_3^{6}
       + 3(a_1^{6}+a_2^{6}) a_3^{5}
       + (a_1^{7}+a_2^{7}) a_3^{4}
     \bigr) t^{3}
\\ &\quad
   + 15\bigl(
       a_1^{8}a_2^{4}
       + 4a_1^{7}a_2^{5}
       + 6a_1^{6}a_2^{6}
       + 4a_1^{5}a_2^{7}
       + a_1^{4}a_2^{8}
       + (a_1^{4}+a_2^{4}) a_3^{8}
       + 4(a_1^{5}+a_2^{5}) a_3^{7}
       + 6(a_1^{6}+a_2^{6}) a_3^{6}
       + 4(a_1^{7}+a_2^{7}) a_3^{5}
     \bigr) t^{2}
\\ &\quad
   - 6\bigl(
       a_1^{9}a_2^{4}
       + 5a_1^{8}a_2^{5}
       + 10a_1^{7}a_2^{6}
       + 10a_1^{6}a_2^{7}
       + 5a_1^{5}a_2^{8}
       + a_1^{4}a_2^{9}
       + (a_1^{4}+a_2^{4}) a_3^{9}
       + 5(a_1^{5}+a_2^{5}) a_3^{8}
     \bigr) t
\\ &\quad
   - 6\bigl(
    + 10(a_1^{6}+a_2^{6}) a_3^{7}
       + 10(a_1^{7}+a_2^{7}) a_3^{6}
       + (5a_1^{8}+a_1^{4}a_2^{4}+5a_2^{8}) a_3^{5}
     + (a_1^{9}+a_1^{5}a_2^{4}+a_1^{4}a_2^{5}+a_2^{9}) a_3^{4}
     \bigr) t
\Bigr).
\end{align*}
\end{appendixformulas}

%----------------------------------------------------------

\subsubsection{Slabs}
\label{appendix:3dmomentslab}

\noindent $M=1$:
$\quad 0$.

%%%%%%%%%%%%%%%%%%%%%%%%%%%%%%%%%%%%
% Moment = 2
%%%%%%%%%%%%%%%%%%%%%%%%%%%%%%%%%%%%
\medskip
\noindent $M=2$:
\begin{appendixformulas}
\begin{align*}
g_1^{(2,3)}
&= \frac{t^{3} + \bigl(2a_1^{2} + a_2^{2} + a_3^{2}\bigr)\,t}{12\,a_1^{3}}
,
\\[-0.2ex]
g_2^{(2,3)}
&= -\frac{1}{480\,a_1^{3} a_2^{3} a_3^{3}}\Bigl[
      \bigl(a_1^{2} a_2^{2} + (a_1^{2} + a_2^{2})a_3^{2}\bigr)t^{5}
     + 10\bigl(
        a_1^{4} a_2^{2}
        - 2a_1^{3} a_2^{3}
        + a_1^{2} a_2^{4}
        + (a_1^{2} + a_2^{2})a_3^{4}
\\[-0.2ex]
&\phantom{={}-\frac{1}{480\,a_1^{3} a_2^{3} a_3^{3}}\Bigl[\ + 10\bigl(}
        - 2(a_1^{3} + a_2^{3})a_3^{3}
        + (a_1^{4} + a_2^{4})a_3^{2}
      \bigr)t^{3}
     + 5\bigl(
        a_1^{6} a_2^{2}
        - 4a_1^{5} a_2^{3}
        + 6a_1^{4} a_2^{4}
        - 4a_1^{3} a_2^{5}
        + a_1^{2} a_2^{6}
\\[-0.2ex]
&\phantom{={}-\frac{1}{480\,a_1^{3} a_2^{3} a_3^{3}}\Bigl[\ + 5\bigl(}
        + (a_1^{2} + a_2^{2})a_3^{6}
        - 4(a_1^{3} + a_2^{3})a_3^{5}
        + 3\bigl(2a_1^{4} + a_1^{2} a_2^{2} + 2a_2^{4}\bigr)a_3^{4}
    - 4\bigl(a_1^{5} + 2a_1^{3} a_2^{2}
\\[-0.2ex]
&\phantom{={}-\frac{1}{480\,a_1^{3} a_2^{3} a_3^{3}}\Bigl[\ + 5\bigl(}
     + 2a_1^{2} a_2^{3} + a_2^{5}\bigr)a_3^{3}
        + \bigl(a_1^{6} + 3a_1^{4} a_2^{2} - 8a_1^{3} a_2^{3} + 3a_1^{2} a_2^{4} + a_2^{6}\bigr)a_3^{2}
      \bigr)t
   \Bigr]
,
\\[-0.2ex]
g_3^{(2,3)}
&= \frac{1}{960\,a_1^{3} a_2^{3} a_3^{3}}\Bigl[
      a_1^{7} a_2^{2}
      - 5a_1^{6} a_2^{3}
      + 10a_1^{5} a_2^{4}
      - 10a_1^{4} a_2^{5}
      + 5a_1^{3} a_2^{6}
      - a_1^{2} a_2^{7}
\\[-0.2ex]
&\phantom{=\frac{1}{960\,a_1^{3} a_2^{3} a_3^{3}}\Bigl[}
      - (a_1^{2} + a_2^{2})a_3^{7}
      + 5(a_1^{3} - a_2^{3})a_3^{6}
      - 2\bigl(5a_1^{4} + 3a_1^{2} a_2^{2} + 5a_2^{4}\bigr)a_3^{5}
    - \bigl(a_1^{2} a_2^{2} + (a_1^{2} + a_2^{2})a_3^{2}\bigr)t^{5}
\\[-0.2ex]
&\phantom{=\frac{1}{960\,a_1^{3} a_2^{3} a_3^{3}}\Bigl[}
      + 5\bigl(
        2a_1^{5}
        + 5a_1^{3} a_2^{2}
        - 5a_1^{2} a_2^{3}
        - 2a_2^{5}
      \bigr)a_3^{4}
     + 5\bigl(
        a_1^{3} a_2^{2}
        - a_1^{2} a_2^{3}
        - (a_1^{2} + a_2^{2})a_3^{3}
        + (a_1^{3} - a_2^{3})a_3^{2}
      \bigr)t^{4}
\\[-0.2ex]
&\phantom{=\frac{1}{960\,a_1^{3} a_2^{3} a_3^{3}}\Bigl[}
      - 5\bigl(
        a_1^{6}
        + 5a_1^{4} a_2^{2}
        - 12a_1^{3} a_2^{3}
        + 5a_1^{2} a_2^{4}
        + a_2^{6}
      \bigr)a_3^{3}
- 10\bigl(
        a_1^{4} a_2^{2}
        - 2a_1^{3} a_2^{3}
        + a_1^{2} a_2^{4}
        + (a_1^{2} + a_2^{2})a_3^{4}
\\[-0.2ex]
&\phantom{=\frac{1}{960\,a_1^{3} a_2^{3} a_3^{3}}\Bigl[\ - 10\bigl(}
        - 2(a_1^{3} + 3a_2^{3})a_3^{3}
        + (a_1^{4} + a_2^{4})a_3^{2}
      \bigr)t^{3}
     + \bigl(
        a_1^{7}
        + 6a_1^{5} a_2^{2}
        - 25a_1^{4} a_2^{3}
        + 25a_1^{3} a_2^{4}
        - 6a_1^{2} a_2^{5}
        - a_2^{7}
      \bigr)a_3^{2}
\\[-0.2ex]
&\phantom{=\frac{1}{960\,a_1^{3} a_2^{3} a_3^{3}}\Bigl[}
      + 10\bigl(
        a_1^{5} a_2^{2}
        - 3a_1^{4} a_2^{3}
        + 3a_1^{3} a_2^{4}
        - a_1^{2} a_2^{5}
 - (a_1^{2} + a_2^{2})a_3^{5}
        + 3(a_1^{3} - a_2^{3})a_3^{4}
        - (3a_1^{4} + a_1^{2} a_2^{2} + 3a_2^{4})a_3^{3}
\\[-0.2ex]
&\phantom{=\frac{1}{960\,a_1^{3} a_2^{3} a_3^{3}}\Bigl[\ + 10\bigl(}
        + (a_1^{5} + a_1^{3} a_2^{2} - a_1^{2} a_2^{3} - a_2^{5})a_3^{2}
      \bigr)t^{2}
 - 5\bigl(
        a_1^{6} a_2^{2}
        - 4a_1^{5} a_2^{3}
        + 6a_1^{4} a_2^{4}
        - 4a_1^{3} a_2^{5}
        + a_1^{2} a_2^{6}
\\[-0.2ex]
&\phantom{=\frac{1}{960\,a_1^{3} a_2^{3} a_3^{3}}\Bigl[\ - 5\bigl(}
        + (a_1^{2} + a_2^{2})a_3^{6}
        - 4(a_1^{3} + 3a_2^{3})a_3^{5}
        + 3\bigl(2a_1^{4} + a_1^{2} a_2^{2} + 2a_2^{4}\bigr)a_3^{4}
 - 4\bigl(a_1^{5} + 2a_1^{3} a_2^{2} + 6a_1^{2} a_2^{3} \bigr)a_3^{3}
\\[-0.2ex]
&\phantom{=\frac{1}{960\,a_1^{3} a_2^{3} a_3^{3}}\Bigl[\ - 5\bigl(}
        + \bigl(a_1^{6} + 3a_1^{4} a_2^{2} - 8a_1^{3} a_2^{3} + 3a_1^{2} a_2^{4} + a_2^{6}\bigr)a_3^{2}
      \bigr)t + 3a_2^{5}a_3^3
   \Bigr]
,
\\[-0.2ex]
g_4^{(2,3)}
&= -\frac{1}{480\,a_1^{3} a_2^{3}}\Bigl[
      5a_1^{6}
      + 25a_1^{4} a_2^{2}
      - 60a_1^{3} a_2^{3}
      + 25a_1^{2} a_2^{4}
      + 5a_2^{6}
 + (a_1^{2} + a_2^{2})a_3^{4}
      + 5(a_1^{2} + a_2^{2})t^{4}
      - 20(a_1^{3} + a_2^{3})t^{3}
\\[-0.2ex]
&\phantom{=-\frac{1}{480\,a_1^{3} a_2^{3}}\Bigl[}
      + 2\bigl(5a_1^{4} + 3a_1^{2} a_2^{2} + 5a_2^{4}\bigr)a_3^{2}
+ 10\bigl(
        3a_1^{4}
        + a_1^{2} a_2^{2}
        + 3a_2^{4}
        + (a_1^{2} + a_2^{2})a_3^{2}
      \bigr)t^{2}
\\[-0.2ex]
&\phantom{=-\frac{1}{480\,a_1^{3} a_2^{3}}\Bigl[}
      - 20\bigl(
        a_1^{5}
        + 2a_1^{3} a_2^{2}
        + 2a_1^{2} a_2^{3}
        + a_2^{5}
        + (a_1^{3} + a_2^{3})a_3^{2}
      \bigr)t
   \Bigr]
,
\\[-0.2ex]
g_5^{(2,3)}
&= -\frac{1}{960\,a_1^{3} a_2^{3} a_3^{3}}\Bigl[
      a_1^{7} a_2^{2}
      + 5a_1^{6} a_2^{3}
      + 10a_1^{5} a_2^{4}
      + 10a_1^{4} a_2^{5}
      + 5a_1^{3} a_2^{6}
      + a_1^{2} a_2^{7}
\\[-0.2ex]
&\phantom{=-\frac{1}{960\,a_1^{3} a_2^{3} a_3^{3}}\Bigl[}
      + (a_1^{2} + a_2^{2})a_3^{7}
      + 5(a_1^{3} + a_2^{3})a_3^{6}
      + 2\bigl(5a_1^{4} + 3a_1^{2} a_2^{2} + 5a_2^{4}\bigr)a_3^{5}
    - \bigl(a_1^{2} a_2^{2} + (a_1^{2} + a_2^{2})a_3^{2}\bigr)t^{5}
\\[-0.2ex]
&\phantom{=-\frac{1}{960\,a_1^{3} a_2^{3} a_3^{3}}\Bigl[}
      + 5\bigl(
        2a_1^{5}
        + 5a_1^{3} a_2^{2}
        + 5a_1^{2} a_2^{3}
        + 2a_2^{5}
      \bigr)a_3^{4}
+ 5\bigl(
        a_1^{3} a_2^{2}
        + a_1^{2} a_2^{3}
        + (a_1^{2} + a_2^{2})a_3^{3}
        + (a_1^{3} + a_2^{3})a_3^{2}
      \bigr)t^{4}
\\[-0.2ex]
&\phantom{=-\frac{1}{960\,a_1^{3} a_2^{3} a_3^{3}}\Bigl[}
      + 5\bigl(
        a_1^{6}
        + 5a_1^{4} a_2^{2}
        - 36a_1^{3} a_2^{3}
        + 5a_1^{2} a_2^{4}
        + a_2^{6}
      \bigr)a_3^{3}
 - 10\bigl(
        a_1^{4} a_2^{2}
        + 2a_1^{3} a_2^{3}
        + a_1^{2} a_2^{4}
        + (a_1^{2} + a_2^{2})a_3^{4}
\\[-0.2ex]
&\phantom{=-\frac{1}{960\,a_1^{3} a_2^{3} a_3^{3}}\Bigl[\ - 10\bigl(}
        + 2(a_1^{3} + a_2^{3})a_3^{3}
        + (a_1^{4} + a_2^{4})a_3^{2}
      \bigr)t^{3}
+ \bigl(
        a_1^{7}
        + 6a_1^{5} a_2^{2}
        + 25a_1^{4} a_2^{3}
        + 25a_1^{3} a_2^{4}
        + 6a_1^{2} a_2^{5}
        + a_2^{7}
      \bigr)a_3^{2}
\\[-0.2ex]
&\phantom{=-\frac{1}{960\,a_1^{3} a_2^{3} a_3^{3}}\Bigl[}
      + 10\bigl(
        a_1^{5} a_2^{2}
        + 3a_1^{4} a_2^{3}
        + 3a_1^{3} a_2^{4}
        + a_1^{2} a_2^{5}
 + (a_1^{2} + a_2^{2})a_3^{5}
        + 3(a_1^{3} + a_2^{3})a_3^{4}
        + (3a_1^{4} + a_1^{2} a_2^{2} + 3a_2^{4})a_3^{3}
\\[-0.2ex]
&\phantom{=-\frac{1}{960\,a_1^{3} a_2^{3} a_3^{3}}\Bigl[\ + 10\bigl(}
        + (a_1^{5} + a_1^{3} a_2^{2} + a_1^{2} a_2^{3} + a_2^{5})a_3^{2}
      \bigr)t^{2}
- 5\bigl(
        a_1^{6} a_2^{2}
        + 4a_1^{5} a_2^{3}
        + 6a_1^{4} a_2^{4}
        + 4a_1^{3} a_2^{5}
        + a_1^{2} a_2^{6}
\\[-0.2ex]
&\phantom{=-\frac{1}{960\,a_1^{3} a_2^{3} a_3^{3}}\Bigl[\ - 5\bigl(}
        + (a_1^{2} + a_2^{2})a_3^{6}
        + 4(a_1^{3} + a_2^{3})a_3^{5}
        + 3\bigl(2a_1^{4} + a_1^{2} a_2^{2} + 2a_2^{4}\bigr)a_3^{4}
+ 4\bigl(a_1^{5} + 2a_1^{3} a_2^{2} + 2a_1^{2} a_2^{3} + a_2^{5}\bigr)a_3^{3}
\\[-0.2ex]
&\phantom{=-\frac{1}{960\,a_1^{3} a_2^{3} a_3^{3}}\Bigl[\ - 5\bigl(}
        + \bigl(a_1^{6} + 3a_1^{4} a_2^{2} + 8a_1^{3} a_2^{3} + 3a_1^{2} a_2^{4} + a_2^{6}\bigr)a_3^{2}
      \bigr)t
   \Bigr].
\end{align*}
\end{appendixformulas}

\medskip
\noindent $M=3$:
$\quad 0$.

%%%%%%%%%%%%%%%%%%%%%%%%%%%%%%%%%%%%
% Moment = 4
%%%%%%%%%%%%%%%%%%%%%%%%%%%%%%%%%%%%
\medskip
\noindent $M=4$:
\begin{appendixformulas}
\begin{align*}
g_1^{(4,3)}
&= \frac{
      3t^{5}
      + 10(a_2^{2} + a_3^{2})t^{3}
      + \bigl(6a_1^{4} + 3a_2^{4} + 10a_2^{2} a_3^{2} + 3a_3^{4}\bigr)t
   }{240\,a_1^{5}},
\\[-0.2ex]
g_2^{(4,3)}
&= -\frac{1}{6720\,a_1^{5} a_2^{5} a_3^{5}}\Bigl[
      \bigl(a_1^{4} a_2^{4} + (a_1^{4} + a_2^{4})a_3^{4}\bigr)t^{7}
      + 21\bigl(
        a_1^{6} a_2^{4}
        - 2a_1^{5} a_2^{5}
        + a_1^{4} a_2^{6}
 + (a_1^{4} + a_2^{4})a_3^{6}
        - 2(a_1^{5} + a_2^{5})a_3^{5}
      \bigr)t^{5}
\\[-0.2ex]
&\phantom{={}-\frac{1}{6720\,a_1^{5} a_2^{5} a_3^{5}}\Bigl[}
      + 35\bigl(
        a_1^{8} a_2^{4}
        - 4a_1^{7} a_2^{5}
        + 6a_1^{6} a_2^{6}
        - 4a_1^{5} a_2^{7}
        + a_1^{4} a_2^{8}
+ (a_1^{4} + a_2^{4})a_3^{8}
        - 4(a_1^{5} + a_2^{5})a_3^{7}
        + 6(a_1^{6} + a_2^{6})a_3^{6}
\\[-0.2ex]
&\phantom{={}-\frac{1}{6720\,a_1^{5} a_2^{5} a_3^{5}}\Bigl[\ + 35\bigl(}
        - 4(a_1^{7} + a_2^{7})a_3^{5}
        + (a_1^{8} + a_2^{8})a_3^{4}
      \bigr)t^{3} + (a_1^{6} + a_2^{6})a_3^{4} t^5
\\[-0.2ex]
&\phantom{={}-\frac{1}{6720\,a_1^{5} a_2^{5} a_3^{5}}\Bigl[}
      + 7\bigl(
        a_1^{10} a_2^{4}
        - 6a_1^{9} a_2^{5}
        + 15a_1^{8} a_2^{6}
        - 20a_1^{7} a_2^{7}
        + 15a_1^{6} a_2^{8}
         - 6a_1^{5} a_2^{9}
        + a_1^{4} a_2^{10}
        + (a_1^{4} + a_2^{4})a_3^{10}
\\[-0.2ex]
&\phantom{={}-\frac{1}{6720\,a_1^{5} a_2^{5} a_3^{5}}\Bigl[\ + 7\bigl(}
        - 6(a_1^{5} + a_2^{5})a_3^{9}
        + 15(a_1^{6} + a_2^{6})a_3^{8}
        - 20(a_1^{7} + a_2^{7})a_3^{7}
        + 5\bigl(3a_1^{8} + a_1^{4} a_2^{4} + 3a_2^{8}\bigr)a_3^{6}
\\[-0.2ex]
&\phantom{={}-\frac{1}{6720\,a_1^{5} a_2^{5} a_3^{5}}\Bigl[\ + 7\bigl(}
        - 6\bigl(a_1^{9} + 2a_1^{5} a_2^{4} + 2a_1^{4} a_2^{5} + a_2^{9}\bigr)a_3^{5}
        + \bigl(a_1^{10} + 5a_1^{6} a_2^{4} - 12a_1^{5} a_2^{5} + 5a_1^{4} a_2^{6} + a_2^{10}\bigr)a_3^{4}
      \bigr)t
   \Bigr]
,
\\[-0.2ex]
g_3^{(4,3)}
&= \frac{1}{13440\,a_1^{5} a_2^{5} a_3^{5}}\Bigl[
      a_1^{11} a_2^{4}
      - 7a_1^{10} a_2^{5}
      + 21a_1^{9} a_2^{6}
      - 35a_1^{8} a_2^{7}
      + 35a_1^{7} a_2^{8}
      - 21a_1^{6} a_2^{9}
      + 7a_1^{5} a_2^{10}
      - a_1^{4} a_2^{11}
      - (a_1^{4} + a_2^{4})a_3^{11}
\\[-0.2ex]
&\phantom{=\frac{1}{13440\,a_1^{5} a_2^{5} a_3^{5}}\Bigl[}
      + 7(a_1^{5} - a_2^{5})a_3^{10}
      - 21(a_1^{6} + a_2^{6})a_3^{9}
      + 35(a_1^{7} - a_2^{7})a_3^{8}
      - 5\bigl(7a_1^{8} + 3a_1^{4} a_2^{4} + 7a_2^{8}\bigr)a_3^{7}
\\[-0.2ex]
&\phantom{=\frac{1}{13440\,a_1^{5} a_2^{5} a_3^{5}}\Bigl[}
    - \bigl(a_1^{4} a_2^{4} + (a_1^{4} + a_2^{4})a_3^{4}\bigr)t^{7}
    +21(- 15a_1^{4} a_2^{7}
        - a_2^{11}) a_3^4
    + 35 t^4 (+ (a_1^{7} - a_2^{7})a_3^{4})
    -35 ( 6(a_1^{6} + a_2^{6})a_3^{6})
\\[-0.2ex]
&\phantom{=\frac{1}{13440\,a_1^{5} a_2^{5} a_3^{5}}\Bigl[}
      + 7\bigl(3a_1^{9} + 8a_1^{5} a_2^{4} - 8a_1^{4} a_2^{5} - 3a_2^{9}\bigr)a_3^{6}
      + 7\bigl(
        a_1^{5} a_2^{4}
        - a_1^{4} a_2^{5}
        - (a_1^{4} + a_2^{4})a_3^{5}
        + (a_1^{5} - a_2^{5})a_3^{4}
      \bigr)t^{6}
\\[-0.2ex]
&\phantom{=\frac{1}{13440\,a_1^{5} a_2^{5} a_3^{5}}\Bigl[}
      - 7\bigl(
        a_1^{10}
        + 8a_1^{6} a_2^{4}
        - 18a_1^{5} a_2^{5}
        + 8a_1^{4} a_2^{6}
        + a_2^{10}
      \bigr)a_3^{5}
      - 21\bigl(
        a_1^{6} a_2^{4}
        - 2a_1^{5} a_2^{5}
        + a_1^{4} a_2^{6}
        + (a_1^{4} + a_2^{4})a_3^{6}
\\[-0.2ex]
&\phantom{=\frac{1}{13440\,a_1^{5} a_2^{5} a_3^{5}}\Bigl[\ - 21\bigl(}
        - 2(a_1^{5} + 3a_2^{5})a_3^{5}
        + (a_1^{6} + a_2^{6})a_3^{4}
      \bigr)t^{5}
      + \bigl(
        a_1^{11}
        + 15a_1^{7} a_2^{4}
        - 56a_1^{6} a_2^{5}
        + 56a_1^{5} a_2^{6}
      \bigr)a_3^{4}
\\[-0.2ex]
&\phantom{=\frac{1}{13440\,a_1^{5} a_2^{5} a_3^{5}}\Bigl[}
      + 35\bigl(
        a_1^{7} a_2^{4}
        - 3a_1^{6} a_2^{5}
        + 3a_1^{5} a_2^{6}
        - a_1^{4} a_2^{7}
        - (a_1^{4} + a_2^{4})a_3^{7}
        + 3(a_1^{5} - a_2^{5})a_3^{6}
        - 3(a_1^{6} + a_2^{6})a_3^{5}
      \bigr)t^{4}
\\[-0.2ex]
&\phantom{=\frac{1}{13440\,a_1^{5} a_2^{5} a_3^{5}}\Bigl[}
      - 35\bigl(
        a_1^{8} a_2^{4}
        - 4a_1^{7} a_2^{5}
        + 6a_1^{6} a_2^{6}
        - 4a_1^{5} a_2^{7}
        + a_1^{4} a_2^{8}
        + (a_1^{4} + a_2^{4})a_3^{8}
        - 4(a_1^{5} + 3a_2^{5})a_3^{7}
\\[-0.2ex]
&\phantom{=\frac{1}{13440\,a_1^{5} a_2^{5} a_3^{5}}\Bigl[\ - 35\bigl(}
        - 4(a_1^{7} + 3a_2^{7})a_3^{5}
        + (a_1^{8} + a_2^{8})a_3^{4}
      \bigr)t^{3}
      + 21\bigl(
        a_1^{9} a_2^{4}
        - 5a_1^{8} a_2^{5}
        + 10a_1^{7} a_2^{6}
        - 10a_1^{6} a_2^{7}
        + 5a_1^{5} a_2^{8}
\\[-0.2ex]
&\phantom{=\frac{1}{13440\,a_1^{5} a_2^{5} a_3^{5}}\Bigl[\ + 21\bigl(}
        - (a_1^{4} + a_2^{4})a_3^{9}
        + 5(a_1^{5} - a_2^{5})a_3^{8}
        - 10(a_1^{6} + a_2^{6})a_3^{7}
        + 10(a_1^{7} - a_2^{7})a_3^{6}
\\[-0.2ex]
&\phantom{=\frac{1}{13440\,a_1^{5} a_2^{5} a_3^{5}}\Bigl[\ + 21\bigl(}
        + (a_1^{9} + a_1^{5} a_2^{4} - a_1^{4} a_2^{5} - a_2^{9})a_3^{4}
      \bigr)t^{2}
      - 7\bigl(
        a_1^{10} a_2^{4}
        - 6a_1^{9} a_2^{5}
        + 15a_1^{8} a_2^{6}
        - 20a_1^{7} a_2^{7}
        + 15a_1^{6} a_2^{8}
\\[-0.2ex]
&\phantom{=\frac{1}{13440\,a_1^{5} a_2^{5} a_3^{5}}\Bigl[\ - 7\bigl(}
        - 6a_1^{5} a_2^{9}
        + a_1^{4} a_2^{10}
        + (a_1^{4} + a_2^{4})a_3^{10}
        - 6(a_1^{5} + 3a_2^{5})a_3^{9}
        + 15(a_1^{6} + a_2^{6})a_3^{8}
        - 20(a_1^{7} + 3a_2^{7})a_3^{7}
\\[-0.2ex]
&\phantom{=\frac{1}{13440\,a_1^{5} a_2^{5} a_3^{5}}\Bigl[\ - 7\bigl(}
        - 6\bigl(a_1^{9} + 2a_1^{5} a_2^{4} + 6a_1^{4} a_2^{5} + 3a_2^{9}\bigr)a_3^{5}
        + \bigl(a_1^{10} + 5a_1^{6} a_2^{4} - 12a_1^{5} a_2^{5} + 5a_1^{4} a_2^{6} + a_2^{10}\bigr)a_3^{4}
      \bigr)t
\\[-0.2ex]
&\phantom{=\frac{1}{13440\,a_1^{5} a_2^{5} a_3^{5}}\Bigl[}
    -21 \bigl(5a_1^{8} + a_1^{4} a_2^{4} + 5a_2^{8}\bigr)a_3^{5}
    + 5\bigl(3a_1^{8} + a_1^{4} a_2^{4} + 3a_2^{8}\bigr)a_3^{6}
      - 21a_1^{4} a_2^{9} a_3^5
   \Bigr]
,
\\[-0.2ex]
g_4^{(4,3)}
&= -\frac{1}{6720\,a_1^{5} a_2^{5}}\Bigl[
      7a_1^{10}
      + 56a_1^{6} a_2^{4}
      - 126a_1^{5} a_2^{5}
      + 56a_1^{4} a_2^{6}
      + 7a_2^{10}
    + (a_1^{4} + a_2^{4})a_3^{6}
      + 7(a_1^{4} + a_2^{4})t^{6}
      - 42(a_1^{5} + a_2^{5})t^{5}
\\[-0.2ex]
&\phantom{=-\frac{1}{6720\,a_1^{5} a_2^{5}}\Bigl[}
      + 21(a_1^{6} + a_2^{6})a_3^{4}
      + 35\bigl(3a_1^{6} + 3a_2^{6} + (a_1^{4} + a_2^{4})a_3^{2}\bigr)t^{4}
    - 140\bigl(a_1^{7} + a_2^{7} + (a_1^{5} + a_2^{5})a_3^{2}\bigr)t^{3}
\\[-0.2ex]
&\phantom{=-\frac{1}{6720\,a_1^{5} a_2^{5}}\Bigl[}
      + 5\bigl(7a_1^{8} + 3a_1^{4} a_2^{4} + 7a_2^{8}\bigr)a_3^{2}
      + 21\bigl(
        5a_1^{8}
        + a_1^{4} a_2^{4}
        + 5a_2^{8}
        + (a_1^{4} + a_2^{4})a_3^{4}
        + 10(a_1^{6} + a_2^{6})a_3^{2}
      \bigr)t^{2}
\\[-0.2ex]
&\phantom{=-\frac{1}{6720\,a_1^{5} a_2^{5}}\Bigl[}
      - 14\bigl(
        3a_1^{9}
        + 6a_1^{5} a_2^{4}
        + 6a_1^{4} a_2^{5}
        + 3a_2^{9}
        + 3(a_1^{5} + a_2^{5})a_3^{4}
        + 10(a_1^{7} + a_2^{7})a_3^{2}
      \bigr)t
   \Bigr]
,
\\[-0.2ex]
g_5^{(4,3)}
&= -\frac{1}{13440\,a_1^{5} a_2^{5} a_3^{5}}\Bigl[
      a_1^{11} a_2^{4}
      + 7a_1^{10} a_2^{5}
      + 21a_1^{9} a_2^{6}
      + 35a_1^{8} a_2^{7}
      + 35a_1^{7} a_2^{8}
      + 21a_1^{6} a_2^{9}
      + 7a_1^{5} a_2^{10}
      + a_1^{4} a_2^{11}
      + (a_1^{4} + a_2^{4})a_3^{11}
\\[-0.2ex]
&\phantom{=-\frac{1}{13440\,a_1^{5} a_2^{5} a_3^{5}}\Bigl[}
      + 7(a_1^{5} + a_2^{5})a_3^{10}
      + 21(a_1^{6} + a_2^{6})a_3^{9}
      + 35(a_1^{7} + a_2^{7})a_3^{8}
      + 7\bigl(3a_1^{9} + 8a_1^{5} a_2^{4} + 8a_1^{4} a_2^{5} + 3a_2^{9}\bigr)a_3^{6}
\\[-0.2ex]
&\phantom{=-\frac{1}{13440\,a_1^{5} a_2^{5} a_3^{5}}\Bigl[}
      + 5\bigl(7a_1^{8} + 3a_1^{4} a_2^{4} + 7a_2^{8}\bigr)a_3^{7}
      - \bigl(a_1^{4} a_2^{4} + (a_1^{4} + a_2^{4})a_3^{4}\bigr)t^{7}  
\\[-0.2ex]
&\phantom{=-\frac{1}{13440\,a_1^{5} a_2^{5} a_3^{5}}\Bigl[}
      + 7\bigl(
        a_1^{5} a_2^{4}
        + a_1^{4} a_2^{5}
        + (a_1^{4} + a_2^{4})a_3^{5}
        + (a_1^{5} + a_2^{5})a_3^{4}
      \bigr)t^{6}
      + 7\bigl(
        a_1^{10}
        + 8a_1^{6} a_2^{4}
        - 54a_1^{5} a_2^{5}
        + 8a_1^{4} a_2^{6}
        + a_2^{10}
      \bigr)a_3^{5}
\\[-0.2ex]
&\phantom{=-\frac{1}{13440\,a_1^{5} a_2^{5} a_3^{5}}\Bigl[}
      - 21\bigl(
        a_1^{6} a_2^{4}
        + 2a_1^{5} a_2^{5}
        + a_1^{4} a_2^{6}
        + (a_1^{4} + a_2^{4})a_3^{6}
        + 2(a_1^{5} + a_2^{5})a_3^{5}
        + (a_1^{6} + a_2^{6})a_3^{4}
      \bigr)t^{5}
\\[-0.2ex]
&\phantom{=-\frac{1}{13440\,a_1^{5} a_2^{5} a_3^{5}}\Bigl[}
      + \bigl(
        a_1^{11}
        + 15a_1^{7} a_2^{4}
        + 56a_1^{6} a_2^{5}
        + 56a_1^{5} a_2^{6}
        + 15a_1^{4} a_2^{7}
        + a_2^{11}
      \bigr)a_3^{4}
\\[-0.2ex]
&\phantom{=-\frac{1}{13440\,a_1^{5} a_2^{5} a_3^{5}}\Bigl[}
      + 35\bigl(
        a_1^{7} a_2^{4}
        + 3a_1^{6} a_2^{5}
        + 3a_1^{5} a_2^{6}
        + a_1^{4} a_2^{7}
        + (a_1^{4} + a_2^{4})a_3^{7}
\\[-0.2ex]
&\phantom{=-\frac{1}{13440\,a_1^{5} a_2^{5} a_3^{5}}\Bigl[\ + 35\bigl(}
        + 3(a_1^{5} + a_2^{5})a_3^{6}
        + 3(a_1^{6} + a_2^{6})a_3^{5}
        + (a_1^{7} + a_2^{7})a_3^{4}
      \bigr)t^{4}
\\[-0.2ex]
&\phantom{=-\frac{1}{13440\,a_1^{5} a_2^{5} a_3^{5}}\Bigl[}
      - 35\bigl(
        a_1^{8} a_2^{4}
        + 4a_1^{7} a_2^{5}
        + 6a_1^{6} a_2^{6}
        + 4a_1^{5} a_2^{7}
        + a_1^{4} a_2^{8}
\\[-0.2ex]
&\phantom{=-\frac{1}{13440\,a_1^{5} a_2^{5} a_3^{5}}\Bigl[\ - 35\bigl(}
        + (a_1^{4} + a_2^{4})a_3^{8}
        + 4(a_1^{5} + a_2^{5})a_3^{7}
        + 6(a_1^{6} + a_2^{6})a_3^{6}
        + 4(a_1^{7} + a_2^{7})a_3^{5}
        + (a_1^{8} + a_2^{8})a_3^{4}
      \bigr)t^{3}
\\[-0.2ex]
&\phantom{=-\frac{1}{13440\,a_1^{5} a_2^{5} a_3^{5}}\Bigl[}
      + 21\bigl(
        a_1^{9} a_2^{4}
        + 5a_1^{8} a_2^{5}
        + 10a_1^{7} a_2^{6}
        + 10a_1^{6} a_2^{7}
        + 5a_1^{5} a_2^{8}
        + a_1^{4} a_2^{9}
        + (a_1^{9} + a_1^{5} a_2^{4} + a_1^{4} a_2^{5} + a_2^{9})a_3^{4}
\\[-0.2ex]
&\phantom{=-\frac{1}{13440\,a_1^{5} a_2^{5} a_3^{5}}\Bigl[\ + 21\bigl(}
        + (a_1^{4} + a_2^{4})a_3^{9}
        + 5(a_1^{5} + a_2^{5})a_3^{8}
        + 10(a_1^{6} + a_2^{6})a_3^{7}
\\[-0.2ex]
&\phantom{=-\frac{1}{13440\,a_1^{5} a_2^{5} a_3^{5}}\Bigl[\ + 21\bigl(}
        + 10(a_1^{7} + a_2^{7})a_3^{6}
        + \bigl(5a_1^{8} + a_1^{4} a_2^{4} + 5a_2^{8}\bigr)a_3^{5}
        \bigr)t^{2}
\\[-0.2ex]
&\phantom{=-\frac{1}{13440\,a_1^{5} a_2^{5} a_3^{5}}\Bigl[}
      - 7\bigl(
        a_1^{10} a_2^{4}
        + 6a_1^{9} a_2^{5}
        + 15a_1^{8} a_2^{6}
        + 20a_1^{7} a_2^{7}
        + 15a_1^{6} a_2^{8}
        + 6a_1^{5} a_2^{9}
        + a_1^{4} a_2^{10}
\\[-0.2ex]
&\phantom{=-\frac{1}{13440\,a_1^{5} a_2^{5} a_3^{5}}\Bigl[\ - 7\bigl(}
        + (a_1^{4} + a_2^{4})a_3^{10}
        + 6(a_1^{5} + a_2^{5})a_3^{9}
        + 15(a_1^{6} + a_2^{6})a_3^{8}
        + 6\bigl(a_1^{9} + 2a_1^{5} a_2^{4} + 2a_1^{4} a_2^{5} + a_2^{9}\bigr)a_3^{5}
\\[-0.2ex]
&\phantom{=-\frac{1}{13440\,a_1^{5} a_2^{5} a_3^{5}}\Bigl[\ - 7\bigl(}
        + 20(a_1^{7} + a_2^{7})a_3^{7}
        + 5\bigl(3a_1^{8} + a_1^{4} a_2^{4} + 3a_2^{8}\bigr)a_3^{6}
\\[-0.2ex]
&\phantom{=-\frac{1}{13440\,a_1^{5} a_2^{5} a_3^{5}}\Bigl[\ - 7\bigl(}
        + \bigl(a_1^{10} + 5a_1^{6} a_2^{4} + 12a_1^{5} a_2^{5} + 5a_1^{4} a_2^{6} + a_2^{10}\bigr)a_3^{4}
      \bigr)t
   \Bigr].
\end{align*}
\end{appendixformulas}

\subsection{Moment formulas: dimension four}
\label{appendix:4Dmoment}

Due to the size of the individual rational functions, the moment formulas for slices with $M=1,2,3,4$ and slab with $M=2,4$ are available in digital form in the repository associated to this article: \url{https://github.com/RainCamel/slab_of_the_poly_norms}.

\end{document}

%% file: Images/2D_chambers.tex
% In your preamble you need:
% \usepackage{tikz}
% \usepackage{xcolor}

\begin{tikzpicture}[scale=1.05]

% --- robust color definitions (no dvipsnames required) ---
\definecolor{CubeBlue}{RGB}{120,160,200}   % soft cube blue
\definecolor{SlabBlue}{RGB}{132, 173, 205}    % slightly darker slab blue
\definecolor{GridGray}{RGB}{150,150,150}
\definecolor{VecBlue}{RGB}{45,90,150} % darker, high-contrast blue
\tikzset{
  box/.style={draw=CubeBlue, very thick},
  gridline/.style={draw=GridGray, thin},
  sweep/.style={draw=SlabBlue, very thick},
  vdot/.style={fill=SlabBlue},
}

% vertical offset for captions
\def\capy{-1.35}

% =======================
% LEFT panel
% =======================
\begin{scope}[xshift=-5.2cm]
  % square
  \filldraw[fill=CubeBlue, fill opacity=0.35, draw=CubeBlue, thick]
    (-1,-1) -- (-1,1) -- (1,1) -- (1,-1) -- cycle;

  % slab
  \filldraw[fill=SlabBlue, fill opacity=0.6, draw=SlabBlue, thick]
    (-0.8, 1) -- (0, 1) -- (0.8,-1) -- (0.0,-1) -- cycle;

  % axes
  \draw[->] (-1.2,0) -- (1.2,0);
  \draw[->] (0,-1.2) -- (0,1.2);

  % caption
  % \node[align=center, yshift=-4mm] at (0,\capy) {parallelogram\\slab};
\end{scope}

% =======================
% CENTER panel
% =======================
\begin{scope}
  \filldraw[fill=CubeBlue, fill opacity=0.35, draw=CubeBlue, thick]
    (-1,-1) -- (-1,1) -- (1,1) -- (1,-1) -- cycle;

  % axes
  \draw[->] (-1.2,0) -- (1.2,0);
  \draw[->] (0,-1.2) -- (0,1.2);

  % clip sweeping lines
  \begin{scope}
    \clip (-1,-1) rectangle (1,1);
    \foreach \c in {-2.4,-1.8,-1.2,-0.6,0,0.6,1.2,1.8,2.4} {
      \draw[sweep] (-2,4+\c) -- (2,-4+\c);
    }
  \end{scope}

  % representative direction vector (perpendicular to the sweeping/diagonal lines)
  \draw[->, very thick, draw=VecBlue]
    (0,0) -- (0.89,0.45);

  % label placed safely at bottom-right
  \node[anchor=north west, text=VecBlue]
    at (1,0.6) {$a_{\mathrm{rep}}=\frac{1}{\sqrt{5}}(2,1)$};

  % dots
  \fill[vdot] (-1, 1) circle[radius=2pt];
  \fill[vdot] ( 1, 1) circle[radius=2pt];
  \fill[vdot] ( 1,-1) circle[radius=2pt];
  \fill[vdot] (-1,-1) circle[radius=2pt];

  % labels
  \node[above left]  at (-1, 1) {\textcolor{SlabBlue}{$v_4$}};
  \node[above right] at ( 1, 1) {\textcolor{SlabBlue}{$v_3$}};
  \node[below right] at ( 1,-1) {\textcolor{SlabBlue}{$v_2$}};
  \node[below left]  at (-1,-1) {\textcolor{SlabBlue}{$v_1$}};

  % caption
  % \node[align=center, yshift=-6mm] at (0,\capy) {Sweeping across\\the square};
\end{scope}

  % two highlighted parallel segments
  %\draw[very thick, draw=magenta!80!black]
    %$(-1,1) -- (0,-1);

\begin{comment}

    \draw[very thick, draw=magenta!80!black]
    (0,1) -- (1,-1);

    \draw[very thick, draw=magenta!80!black]
    (-0.5,1) -- (0.5,-1);

    \draw[very thick, draw=magenta!80!black]
    (1,1) -- (2,-1);
\end{comment}

% =======================
% RIGHT panel
% =======================
\begin{scope}[xshift=5.2cm]
  \filldraw[fill=CubeBlue, fill opacity=0.35, draw=CubeBlue, thick]
    (-1,-1) -- (-1,1) -- (1,1) -- (1,-1) -- cycle;

  \filldraw[fill=SlabBlue, fill opacity=0.6, draw=SlabBlue, thick]
    (-1, 0) -- (-1, 1) -- (0, 1) -- (1,0) -- (1,-1) -- (0,-1) -- cycle;

  % axes
  \draw[->] (-1.2,0) -- (1.2,0);
  \draw[->] (0,-1.2) -- (0,1.2);

  % caption
  % \node[align=center, yshift=-4mm] at (0,\capy) {hexagonal\\slab};
\end{scope}

\end{tikzpicture}

%% file: Images/2D_slices.tex
\begin{tikzpicture}[scale=1.8]

% --- robust color definitions (no dvipsnames required) ---
\definecolor{CubeBlue}{RGB}{120,160,200}   % soft cube blue
\definecolor{SlabBlue}{RGB}{132, 173, 205}    % slightly darker slab blue
\definecolor{GridGray}{RGB}{150,150,150}
\definecolor{VecBlue}{RGB}{45,90,150} % darker, high-contrast blue
\tikzset{
  box/.style={draw=CubeBlue, very thick},
  gridline/.style={draw=GridGray, thin},
  sweep/.style={draw=SlabBlue, very thick},
  vdot/.style={fill=SlabBlue},
}

% vertical offset for captions
\def\capy{-1.35}

% --- First (full-size) figure ---
\begin{scope}[scale=0.8]
  % axes
  \draw[->] (-1.2,0) -- (1.2,0);
  \draw[->] (0,-1.2) -- (0,1.2);

  % square
  \filldraw[fill=CubeBlue, fill opacity=0.3, draw=CubeBlue, thick]
    (-1,-1) -- (-1,1) -- (1,1) -- (1,-1) -- cycle;

  % Q original
  \draw[SlabBlue, very thick]
    (-0.2,1) -- (0.6,-1);

  \fill[SlabBlue] (-0.2,1) circle[radius=2pt];
  \fill[SlabBlue] (0.6,-1) circle[radius=2pt];
  
  % names polytopes
  \node[] at (-0.2,1.2) {\textcolor{SlabBlue}{$v_2$}}; 
  \node[] at (0.6,-1.2) {\textcolor{SlabBlue}{$v_1$}}; 
\end{scope}

% --- Second figure (Q translated by (0,2)) ---
\begin{scope}[xshift=4cm,yshift=0cm,scale=0.8]
  % axes
  \draw[->] (-1.2,0) -- (1.2,0);
  \draw[->] (0,-1.2) -- (0,1.2);

  % triangle
  \filldraw[fill=CubeBlue, fill opacity=0.3, draw=CubeBlue, thick]
    (-1,-1) -- (-1,1) -- (1,1) -- (1,-1) -- cycle;

  % Q original
  \draw[SlabBlue, very thick]
    (0.2,1) -- (1,-0.2);

  \fill[SlabBlue] (0.2,1) circle[radius=2pt];
  \fill[SlabBlue] (1,-0.2) circle[radius=2pt];
  
  % names polytopes
  \node[] at (0.2,1.2) {\textcolor{SlabBlue}{$v_2$}}; 
  \node[] at (1.25,-0.2) {\textcolor{SlabBlue}{$v_1$}}; 
\end{scope}
\end{tikzpicture}